\documentclass[preprint,sort&compress,12pt]{elsarticle}

\usepackage[top=1in, bottom=1in, left=1in, right=1in]{geometry}

\usepackage{graphicx}
\usepackage{amsmath}

\usepackage{natbib}
\usepackage{amssymb}

\usepackage{amsthm}

\usepackage{lineno}
\usepackage{subfig}
\usepackage{algorithm}
\usepackage{algpseudocode}
\usepackage{xcolor}
\usepackage[colorinlistoftodos]{todonotes}
\usepackage[toc,page]{appendix}
\usepackage{mathrsfs}

\biboptions{sort,compress} 

\usepackage{xcolor}

\newcommand*\patchAmsMathEnvironmentForLineno[1]{%
  \expandafter\let\csname old#1\expandafter\endcsname\csname #1\endcsname
  \expandafter\let\csname oldend#1\expandafter\endcsname\csname end#1\endcsname
  \renewenvironment{#1}%
     {\linenomath\csname old#1\endcsname}%
     {\csname oldend#1\endcsname\endlinenomath}}%
\newcommand*\patchBothAmsMathEnvironmentsForLineno[1]{%
  \patchAmsMathEnvironmentForLineno{#1}%
  \patchAmsMathEnvironmentForLineno{#1*}}%
\AtBeginDocument{%
\patchBothAmsMathEnvironmentsForLineno{equation}%
\patchBothAmsMathEnvironmentsForLineno{align}%
\patchBothAmsMathEnvironmentsForLineno{flalign}%
\patchBothAmsMathEnvironmentsForLineno{alignat}%
\patchBothAmsMathEnvironmentsForLineno{gather}%
\patchBothAmsMathEnvironmentsForLineno{multline}%
}

\usepackage{graphicx}
\usepackage{amssymb}
\usepackage{amsthm}
\usepackage{bbm}
\usepackage{bm}
\usepackage{lineno}
\usepackage{url}
\usepackage{listings}
\usepackage[colorlinks=true]{hyperref}

\newtheorem{remark}{Remark}[]
\newtheorem{example}{Example}[]
\newtheorem{proposition}{Proposition}[]

\numberwithin{equation}{section}
\numberwithin{theorem}{section}
\numberwithin{remark}{section}
\numberwithin{example}{section}

\newcommand{\bbE}{\mathbb{E}}
\newcommand{\bbR}{\mathbb{R}}

\newcommand{\cX}{\mathcal{X}}
\newcommand{\cF}{\mathcal{F}}
\newcommand{\cG}{\mathcal{G}}
\newcommand{\sigmad}{\sigma^\dagger}

\newcommand{\cV}{\mathfrak{V}}

\usepackage{color,soul}
\usepackage{mathtools}

\definecolor{lightblue}{rgb}{.90,.95,1}
\definecolor{darkgreen}{rgb}{0,.5,0.5}

\definecolor{lightgreen}{rgb}{.90,1,0.90}

\usepackage{gensymb}
\usepackage{array}
\usepackage{changes}
\usepackage{multirow}
\usepackage{enumerate}
\newcolumntype{P}[1]{>{\centering\arraybackslash}m{#1}}

\newcolumntype{L}[1]{>{\raggedright\let\newline\\\arraybackslash\hspace{0pt}}m{#1}}
\newcolumntype{C}[1]{>{\centering\let\newline\\\arraybackslash\hspace{0pt}}m{#1}}
\newcolumntype{R}[1]{>{\raggedleft\let\newline\\\arraybackslash\hspace{0pt}}m{#1}}

\DeclareMathOperator*{\argmin}{arg\,min}

\graphicspath{ {./figs/}}

\begin{document}

\begin{frontmatter}
\title{Ensemble Kalman Inversion For Sparse Learning Of Dynamical Systems From Time-Averaged Data}
\author[1]{Tapio Schneider}\ead{tapio@caltech.edu}
\author[1]{Andrew M. Stuart}\ead{astuart@caltech.edu}
\author[1]{Jin-Long Wu}\ead{jinlong@caltech.edu}
\address[1]{California Institute of Technology, Pasadena, CA 91125.}

\begin{abstract}
Enforcing sparse structure within learning has led to significant advances
in the field of data-driven discovery of dynamical systems. However,
such methods require access not only to time-series of the state of the
dynamical system, but also to the time derivative. In many
applications, the data are available only in the form of time-averages such as moments and autocorrelation
functions. We propose a sparse learning methodology to discover the vector fields defining a (possibly stochastic or partial) differential equation, using 
only time-averaged statistics. Such a formulation of sparse learning naturally leads to a
nonlinear inverse problem to which we apply the methodology of ensemble Kalman inversion (EKI). EKI is chosen because it may be formulated in terms of
the iterative solution of quadratic optimization problems;
sparsity is then easily imposed. We then apply the EKI-based sparse learning methodology to various examples governed by stochastic differential equations (a noisy Lorenz 63 system), ordinary differential equations (Lorenz 96 system and coalescence equations), and a partial differential equation (the Kuramoto-Sivashinsky equation). The results demonstrate that time-averaged statistics can be
used for data-driven discovery of differential equations using sparse EKI. The proposed sparse learning methodology extends the scope of data-driven discovery of differential equations to previously challenging applications and
data-acquisition scenarios. 
\end{abstract}


\end{frontmatter}

\clearpage
\section{Introduction}

\subsection{Overview and Literature Review}
The goal of this paper is to describe a sparse learning methodology 
to discover the vector fields defining a (possibly stochastic or partial) differential equation, using time-series data. The
approach is to use ensemble Kalman inversion (EKI) to learn the unknown vector fields by matching the model to time-averaged statistics derived from the time-series data.  
Sparse learning allows for
the discovery of dynamical models from within a large dictionary of models; learning from time-averaged statistics
is often necessary either because data are available only in that form,
or because use of time-averages avoids incompatibility issues between
model and data at small time increments and the lack of
differentiability of sample paths of stochastic differential
equations (SDEs). The EKI methodology is
an approach to parameter identification, which lends itself
naturally to the imposition of constraints such as sparsity,
and which is known empirically to be robust and flexible.
The work presented here leads to a new approach which widens the
scope of sparse learning problems in dynamical systems
and that is demonstrably both flexible (applying to a range
of examples) and robust (works very well in practice). 
Our work suggests wider deployment of the proposed
methodology, and the need for development of an underpinning
theory. 

 Seeking sparse structure in learning has played a significant role
 in recent decades. Sparse dictionary learning techniques are well known as compressed sensing~\cite{donoho2006compressed,candes2006stable,candes2008enhancing} and have already been extensively studied in application domains
 such as image and signal processing~\cite{bruckstein2009sparse}. The general concept of incorporating sparsity into optimization has also been studied in a variety of applied disciplines for several
 decades, for example in applications in geophysics~\cite{santosa1986linear}. Since then it has  been formulated as a theoretical framework known as LASSO~\cite{tibshirani1996regression,tibshirani2011regression}. In addition, sparsity-promoting techniques have been found useful in emerging areas in artificial intelligence, such as deep learning~\cite{goodfellow2016deep}.

Exploitation of sparsity in the data-driven discovery of differential 
equations was pioneered in a recent series of papers~\cite{brunton2016discovering,rudy2017data,schaeffer2013sparse,schaeffer2017learning}, all of which make the assumption that nature favors simplicity and that the vector fields to be discovered are sparse within a high-dimensional dictionary. More recently, a sparsity-promoting joint outlier detection and model discovery method was proposed in~\cite{tran2017exact}, and a sparsity-promoting method was proposed in~\cite{schaeffer2018extracting} for learning governing equations of dynamical systems from undersampled data. The data-driven discovery of differential equations with sparsity has also been investigated for the learning of stochastic differential equations~\cite{pereira2010learning,boninsegna2018sparse}. These methods need to be provided with,
or need to numerically evaluate, time-series of the time derivative of the state variables, as well as the time-series of state variables themselves.
In this setting, the learning problem may be phrased as an over-determined
system of linear equations, and the solution may be sought through
a regularized least squares approach in which the regularization imposes
sparsity. When numerical differentiation is required, it is susceptible to noise in the time-series of the state variables. Techniques such as total variation regularization~\cite{rudin1992nonlinear} have been adopted to alleviate this issue by denoising the time derivative \cite{chartrand2011numerical}. 
Nonetheless, the presence of noise in time-series presents a significant
issue for these approaches.

One way of circumventing the need to numerically differentiate time-series
is to fit dynamical models to statistics derived from time-series,
such as moments or autocorrelation functions. This idea is widely
used in the study of autoregressive (AR) processes \cite{shaman1988bias,martin1996autoregressive,Neumaier01a, Schneider01c, lutkepohl2013introduction,brockwell1991time} and in the study of SDEs \cite{krumscheid2013semiparametric,krumscheid2015data,kalliadasis2015new,klus2020data,SSW}. There are also a plethora of applied papers that take this
approach, in both discrete and continuous time, such as
\cite{brown1984time,kwasniok2009deriving}. 
In addition to avoiding numerical differentiation, such methods
also have the potential to learn models when only a subset
of the state variables are observed. Furthermore, there are settings
in which only time-averaged data are available. It is important
to note that the parameter-to-data map for such problems is
nonlinear. 

EKI is a general methodology for nonlinear inverse problems described in~\cite{iglesias2013ensemble}, building on algorithms designed
for the solution of inverse problems arising in oil reservoir simulation~\cite{chen2012ensemble,emerick2013ensemble}. 
The incorporation of regularization into EKI is discussed in
\cite{iglesias2016regularizing,chada2019tikhonov},
and the incorporation of constraints in \cite{albers2019ensemble,chada2019incorporation,wu2019adding,zhang2020regularized,strofer2020enforcing}.
In this work, we propose an EKI-based sparsity-promoting methodology for parameter learning. This sparse EKI method combines ideas from
\cite{iglesias2016regularizing,albers2019ensemble} to create a derivative-free optimization approach to parameter learning that enforces sparsity. We apply the method to the learning of vector fields in (possibly stochastic or partial) differential equations, building on the ideas in~\cite{brunton2016discovering}, but using nonlinear indirect measurements defined by time-averaging, rather than linear direct observations. It is a remarkable
property of EKI that, despite the nonlinearity of the observation
operator, the core computational task is the minimization of a 
quadratic objective functional to which a sparsity constraint maybe easily
added, just as it is in the original work on sparse
learning of dynamical systems in~\cite{brunton2016discovering}. 
This fact allows transfer of the learning framework introduced in
\cite{brunton2016discovering} to more complex indirect, nonlinearly and partially observed dynamics, and indeed to a wide range of nonlinear inverse and parameter identification problems.

\subsection{Our Contribution}

Our contributions in this paper are as follows:
\begin{itemize}
    \item We demonstrate how to impose a sparsity constraint within
    the EKI algorithm, by formulating the update step as an $\ell_1$ and/or $\ell_0$ regularized least squares problem.
    \item  We demonstrate the use of sparsity-promoting EKI to discover the governing equations of (possibly stochastic) dynamical systems based on statistics derived from averaging time-series. {The results are compared with those 
obtained using standard EKI to illustrate the merits of imposing sparsity.}
    \item We illustrate the methodology in two simulation studies, discovering the stochastic Lorenz '63 and the
    deterministic Lorenz '96 systems from data, and we also illustrate the methodology to find a closure model for the slow variables within a multiscale Lorenz '96 system.
    \item We illustrate the methodology by discovering coalescence equations for collisional dynamics, using both simulation studies and closure models.
    \item We illustrate the methodology in the context of
    discovering the Kuramoto-Sivashinsky equation from a larger family of
linear dissipative and dispersive linear systems subject to energy conserving
quadratic nonlinearities.
    \item We demonstrate how to impose constraints on the parameter
    learning, which ensure that the subset of parameters that are queried
    during the algorithm all lead to well-posed dynamical systems.
\end{itemize}

Furthermore, although we apply the method in the
context of learning dynamical systems, the EKI-based
sparse methodology may be more
widely applied within nonlinear inverse problems generally.

In Section \ref{sec:PF}, we formulate the inverse problem of interest
and introduce the four problem classes to which we will apply our methodology. Section \ref{sec:Alg} describes the ensemble Kalman-based methodology 
which we employ to solve the inverse problem. It also discusses
the quadratic programming approach we employ to incorporate 
the $\ell_1-$penalty into the ensemble Kalman-based methodology,
and the proximal gradient methodology used to incorporate 
the $\ell_0-$penalty.
In Section \ref{sec:N}, we describe numerical results relating to 
each of the four example problems. We conclude in Section \ref{sec:C}.

\subsection{Notation}

Throughout we use $|\cdot|_{\ell_p}$ to denote the $p-$norm on Euclidean
space, extended to include the case $p=0$, which counts the number of non-zero entries of the vector. The commonly occuring case $p=2$ is simply denoted
by $|\cdot|$, and the notation $|\cdot|_{A}:=|A^{-\frac12}\cdot|$ is used
for symmetric positive-definite $A$.

\section{Problem Formulation}
\label{sec:PF}
The aim of this work is to use time-series data to
learn the right hand side of a differential equation
\begin{equation}
\label{eq:tds}
    \frac{dx}{dt}=f(x),
\end{equation}
where $x \in \mathbb{R}^n$ and $f: \mathbb{R}^n \mapsto \mathbb{R}^n$,
from time-averaged information about $x$. To
this end, we first approximate $f$ with a set of basis functions $\phi=\{\phi_i\}$, $i \in \{1,...,p\}$, leading to a modeled 
differential equation
\begin{equation}
\label{eq:mds}
    \frac{dX_k}{dt}=\sum_{i=1}^p\theta_{ki}\phi_i(X), \quad k=1,...,n,
\end{equation}
where $X \in \mathbb{R}^n$ with components $X_k$, $\phi_i: \mathbb{R}^n \mapsto \mathbb{R}$,
and the parameter matrix $\theta \in \Theta \subseteq \mathbb{R}^{n \times p}$. We assume that, with appropriate choice of the basis functions
$\{\phi_i\}$, the function $X(t)$ provides a good approximation of $x(t)$ 
for some choice of parameter matrix $\theta$; furthermore, we assume that
this choice of $\theta$ is sparse in the sense that $|\theta|_{\ell_0} \ll np.$ We will also consider generalizations
to SDEs and to partial differential equations (PDEs). 

We assume that the data $y$ available to us is in the form
of time-averages of quantities derived from $x(t)$, or linear
transformations of such quantities. This includes
moments, autocorrelation functions, and the power spectral
density. If $x(\cdot), X(\cdot;\theta) \in \cX:=C(\bbR^+,\bbR^n)$ denote solutions of the true and modeled systems, $\Theta$ denotes the subset of
parameter space over which we seek modeled solutions,
and $\cF:\cX \mapsto \bbR^J$ is
a function on the space of solution trajectories, then define $\cG(\theta):= \cF(X(\cdot;\theta)):\Theta \mapsto \bbR^J$. In this work, $\cF$ corresponds to time-averaged functions of solution trajectories $x(\cdot)$.
We focus on solution of the
following inverse problem to determine $\theta$ from $y$:
\begin{equation}
\label{eq:ip}
y=\cG(\theta).
\end{equation}
For simplicity we have assumed independence of $\cG$ on the initial condition 
(and the driving Brownian noise in the SDE case), noting that for ergodic 
problems this dependence indeed disappears when time-averages 
over the infinite time horizon are used.
In practice, a noisy finite-time average is used to generate the
data, and the 
resulting fluctuations may be viewed as small noise around the infinite 
time average, and we will account for this in our algorithms.
The ergodic setting will obtain for most of the examples considered in 
this paper. However, one of the examples we study is not ergodic 
(the coalescence equations), and in that setting we study the 
dependence of our learned parameters on the initial condition.

The formulation in Eq.~\eqref{eq:ip} has the advantage that it does not involve the matching of trajectories $x(t)$ and $X(t)$, a problem that can be
difficult when noise is present in the data (for example
from using finite-time rather than ergodic averages) or when
the trajectory is not differentiable (as arises in SDEs). However 
the approach we adopt has the apparent disadvantage that
the data available may be of small volume. Indeed, it may be the
case that $J \ll np$ --- that is, the amount of data  is far less than the number of 
unknowns. Nonetheless, nature favors simplicity in many cases, and then a 
sparse solution for $\theta$  provides a 
better modeled system than a dense one and can still be identifiable with limited data. Therefore, we aim to solve the inverse problem formulated in Eq.~\eqref{eq:ip} by using a modified version of ensemble Kalman inversion (EKI) that promotes sparsity in $\theta$.

We now describe four examples that will be used to illustrate the methodology.
In all four cases, we demonstrate how to ensure that parameter learning
takes place within a subset of parameter models that lead to well-posed
dynamics. The issue of ensuring this does not arise in the approach of
\cite{brunton2016discovering} because the dynamical system is not simulated
as part of the algorithm. For the EKI approach adopted here, it is integral
to the method that the candidate model problems are simulated for a variety
of parameter values during the learning algorithm, and the
resulting outputs compared with the data available. This ensures that the
candidate parameter values lead to well-posed dynamics. The following four examples will be used in the numerical illustrations in Section \ref{sec:N}. The reader primarily
interested in the form of the algorithm can skip straight to Section \ref{sec:Alg} and return to these examples in conjunction with reading Section \ref{sec:N}.

\begin{example}[{\bf Lorenz 63 System \cite{lorenz1963deterministic}}]
\label{ssec:L63}
The noisy Lorenz equations are a system of three ordinary differential equations taking the form
\begin{equation}
\label{eq:l63}
    \dot{x} = f(x)+\sqrt{\sigmad}\dot{W},
\end{equation}
where $W$ is an $\mathbb{R}^3$-valued Brownian motion, $x=[x_1,x_2,x_3]^\top$, and $f:\mathbb{R}^3 \mapsto \mathbb{R}^3$ is given by
\begin{equation}
\label{eq:fl63}
\begin{aligned}
f_1(x)&=\alpha(x_2-x_1), \\
f_2(x)&=x_1(\rho-x_3)-x_2, \\
f_3(x)&=x_1x_2-\beta x_3.
\end{aligned}
\end{equation}

We will seek a modeled system of the form
\begin{equation}
\label{eq:l63m}
\dot{X_k}=\sum_{i=1}^9\theta_{ki}\phi_i(X)+\sqrt{\sigma}\dot{W}_k, \quad k=1,2,3, \\
\end{equation}
where $\phi=\{\phi_i\}$, $i \in \{1,...,9\}$ contains all the first ($i \in \{1,2,3\}$) and second-order ($i \in \{6,..,9\}$) polynomial basis functions and the $W_k$ are independent $\mathbb{R}$-valued Brownian motions. 

In this setting, the modeled system in Eq.~\eqref{eq:l63m} 
coincides with the true system in Eq.~\eqref{eq:l63} with a proper choice of parameters $\theta$ and noise level $\sigma$. This example thus serves as a simulation study, while also illustrating
the applicability of our sparse discovery method to SDEs, hence going
beyond~\cite{brunton2016discovering}. 

The parameter vector $\theta$ contains 27 unknowns. To ensure
well-posedness of the explored model-class we further impose that
the quadratic terms are energy conserving;
specifically, we enforce that the inner-product of the
quadratic terms with $X$ is identically zero:
\begin{equation}
\label{eq:QC}
\sum_{k=1}^3 \sum_{i=4}^9 X_k \theta_{ki} \phi_i(X) \equiv 0.
\end{equation}
This ensures that the quadratic term contributes zero energy
to the system and is natural from the viewpoint of the
geophysical modelling considerations that underpin the
model. Mathematically, imposition of \eqref{eq:QC} 
ensures that the stochastic 
differential equation does not explode in finite time as it
implies boundedness of the second moment at any fixed  positive time
\cite{mao2007stochastic,mattingly2002ergodicity}. 
The constraints in \eqref{eq:QC} number
$10$, corresponding to removal of  $X_1^3$, $X_2^3$, $X_3^3$, $X_1^2X_2$, $X_1^2X_3$, $X_2^2X_1$, $X_2^2X_3$, $X_3^2X_1$, $X_3^2X_2$, and $X_1X_2X_3$
from the energy.
Consequently, there remain $17$ independent unknown coefficients  in $\theta=\{\theta_{ki}\}$, $k \in \{1,2,3\}$, $i \in \{1,...,9\}$ after incorporating the energy constraint. Our goal is to learn a sparse solution $\theta$ which has less than $17$ non-zero elements, as well as the noise level $\sigma.$ Ideally, in this
simulation study, the learnt solution for $\theta$ will have $7$ non-zero elements that agree with the true system in Eq.~\eqref{eq:l63}, and a value of $\sigma$ which agrees with the
true value $\sigmad.$

\begin{proposition}
\label{prop:1}
Assume that the constraints on parameters $\{\theta_{ki}\}$ are
chosen as detailed above, so as to ensure \eqref{eq:QC} holds. 
Then for any $T>0$, there are constants $c_1, c_2>0$ 
such that  equation \eqref{eq:l63m} has, almost surely, a unique solution satisfying $u=(X_1,X_2,X_3) \in C([0,T];\mathbb{R}^3)$, and
$$\sup_{t \in [0,T]} \bbE|u|^2 \le  \bigl(|u_0|^2+c_1\bigr)e^{c_2T}.\quad\diamond$$
\end{proposition}
\end{example}

\begin{proof}
Define the Lyapunov function $V(u)=\frac12|u|^2$. Applying
the It\^o formula to $u$ solving \eqref{eq:l63m} gives
$$\frac{d}{dt}\{\bbE V(u)\}=\bbE \sum_{k=1}^3 \sum_{i=1}^9 X_k \theta_{ki} \phi_i(X)+\sigma.$$
(The precise interpretation of this inequality is in time-integrated form).
Applying \eqref{eq:QC} and noting that the $\phi_{i}$ are
linear in $u$ for $i=1,2,3$ leads, after application of the
Cauchy-Schwarz inequality, to the bound $$\frac{d}{dt}\{\bbE V(u)\} \le 
\alpha \bbE V(u)+\sigma$$
for some $\alpha>0$ (again to be interpreted in integrated form).
Integration of the inequality yields the
conclusion of the proposition, by application of the moment
bound theory of It\^o SDEs explained in \cite{mao2007stochastic}.
\end{proof}

\begin{example}[{\bf {Lorenz 96 System~\cite{lorenz1996predictability}}}]
\label{ssec:L96}
The Lorenz 96 single scale system describes the time evolution of a set of variables $\{x_k\}_{k=1}^K$ according to the equations
\begin{equation}
\begin{aligned}
\label{eq:l96ss}
\dot{x}_k&=-x_{k-1}(x_{k-2}-x_{k+1})-x_k+F, \quad k \in \{1,...,K\}, \\
x_{k+K}&=x_k.
\end{aligned}
\end{equation}
We choose $K=36$ and use the system in Eq.~\eqref{eq:l96ss} as the true system for a simulation study. We aim at modeling the unknown tendency with first and second-order polynomial basis functions:
\begin{equation}
\label{eq:this}
\dot{X_k}=\sum_{i=1}^{702}\theta_{ki}\phi_i(X), \quad k=1,\dots, 36\, , \\
\end{equation}
where $\phi=\{\phi_i\}$, $\mid i \in \{1,...702\}$ contains all the first ($\phi=\{\phi_i\}$, $\mid i \in \{1,...,36\}$) and second-order polynomial basis functions. 

As in Example~\ref{ssec:L63}, imposition of energy conserving quadratic nonlinearities is important from
a modeling point of view and as a means to ensure well-posedness, i.e., existence of
solutions to the equation for all time.
To this end, we work with a simpler modeled system, from a subclass of the
models \eqref{eq:this}, taking the form
\begin{equation}
\begin{aligned}
\label{eq:l96ssm}
\dot{X}_k=&-X_{k-1}(\beta_k^{(1)}X_{k-2}-\beta_{k+1}^{(1)}X_{k+1})-(\beta_k^{(2)}X_{k-1}X_{k}-\beta_{k+1}^{(2)}X_{k+1}^2)\\
&-(\beta_k^{(3)}X_{k}X_{k+1}-\beta_{k-1}^{(3)}X_{k-1}^2)-(\beta_k^{(4)}X_{k-1}X_{k+1}-\beta_{k+1}^{(4)}X_{k+1}X_{k+2}) \\
&-\alpha_k X_k+F, \quad k \in \{1,...,K\}, \\
X_{k+K}=&X_{k}.
\end{aligned}
\end{equation}
Thus we only introduce the second-order polynomial basis functions that are constructed by a single variable and its two nearest neighbors, together with a linear diagonal term. This incorporates the energy conservation  constraint --- the inner-product of the quadratic terms with $X$ is identically zero. The boundary conditions for the unknown parameters in Eq.~\eqref{eq:l96ssm} are $\beta_{k+K}^{(i)}=\beta_{k}^{(i)}$ and $\alpha_{k+K}=\alpha_{k}$ for $i \in \{1,2,3,4\}$ and $k \in \{1,...,K\}$. Therefore, we have $144$ unknowns in $\beta$ and $36$ unknowns in $\alpha$. Our goal is to learn a sparse solution $\{\{\beta_k^{(i)}\}_{i=1}^4,\alpha_k\}_{k=1}^{36}$ which has considerably
fewer than $180$ non-zero elements. Ideally, of course, in this simulation 
study setting, the sparse solution will have $72$ non-zero elements that 
agree with the true system in Eq.~\eqref{eq:l96ss}.

\begin{proposition}
\label{prop:2}
Equation \eqref{eq:l96ssm} has unique solution $u=(X_1,\cdots,X_K) \in C([0,\infty);\mathbb{R}^K). \quad \diamond$
\end{proposition}

\begin{proof}
Define the Lyapunov function $V(u)=\frac12|u|^2$.
Straightforward computation using \eqref{eq:l96ssm} gives

$$\frac{d}{dt}V(u) \le \alpha V(u)+\beta$$
for some $\alpha, \beta>0$ after using the fact that \eqref{eq:this} holds and using Cauchy-Schwarz. Integration of the inequality yields the conclusion of the proposition
since, for finite dimensional systems, blow-up in finite
time is the only way the solution can cease to exist, and the
bound precludes this. 
\end{proof}

In addition, we also consider a situation in which data are generated
by the \emph{multiscale} Lorenz 96 system:
\begin{equation}
\begin{aligned}
\label{eq:l96ms}
    \dot{x_k}&=-x_{k-1}(x_{k-2}-x_{k+1})-x_k+F-\frac{hc}{J}\sum_{j=1}^J y_{j,k}, \quad k \in \{1,\dots, K\}, \\
    \frac{1}{c}\dot{y}_{j,k}&=-by_{j+1,k}(y_{j+2,k}-y_{j-1,k})-y_{j,k}+\frac{h}{J}x_k, \quad (j,k) \in \{1,\dots, J\} \times \{1,\dots, K\}\\
    &\quad x_{k+K}   = x_k, \quad y_{j,k+K} = y_{j,k}, \quad y_{j+J,k} = y_{j,k+1}.
\end{aligned}
\end{equation}
We will take the values $K=36$ and $J=10$ and work with parameter values
in which the $X$ and $Y$ variables are scale-separated. In this setting,
the averaging principle enables elimination of the $Y$ variables
from the $X$ equation because they are a function of $X$. Thus,
it is natural to try and fit data from the $X$ variable in the multiscale
system Eq.~\eqref{eq:l96ms} to a closed equation in $X$ alone, of the form
\begin{equation}
\begin{aligned}
\label{eq:l96msm}
\dot{X}_k=&-X_{k-1}(\beta_k^{(1)}X_{k-2}-\beta_{k+1}^{(1)}X_{k+1})-(\beta_k^{(2)}X_{k-1}X_{k}-\beta_{k+1}^{(2)}X_{k+1}^2)\\
&-(\beta_k^{(3)}X_{k}X_{k+1}-\beta_{k-1}^{(3)}X_{k-1}^2)-(\beta_k^{(4)}X_{k-1}X_{k+1}-\beta_{k+1}^{(4)}X_{k+1}X_{k+2}) \\
&-\alpha_k X_k+F+g(X_k), \quad k \in \{1,...,K\}, \\
X_{k+K}=&X_{k}.
\end{aligned}
\end{equation}
Comparison with Eq.~\eqref{eq:l96ssm} shows that the only difference is an additional function $g(X_k)$. The averaging principle alone does not justify
the diagonal and universal form of the closure $g(\cdot)$ but
empirical evidence, and arguments based on $J \gg 1$, show that it is
a reasonable closure model to employ, an idea developed in \cite{fatkullin2004computational} and studied further in \cite{BGMS,SSW}. 
We use a hierarchical Gaussian process (GP) with $10$ unknowns to parameterize the function $g(X_k)$, as introduced in~\cite{SSW}, and learn the GP together with unknown parameters $\{\{\beta_k^{(i)}\}_{i=1}^4,\alpha_k\}_{k=1}^{36}$, based on the data from multiscale Lorenz 96 system in Eq.~\eqref{eq:l96ms}. The sparsity constraint is not put on the GP parameters but only on the
$\{\{\beta_k^{(i)}\}_{i=1}^4,\alpha_k\}_{k=1}^{36}.$

\end{example}

\begin{example}[{\bf {Coalescence Equations}}]

Coagulation and fragmentation  equations \cite{ball1986becker,ball1990discrete} for systems of particles or droplets may be found in  the modeling of a wide range of phenomena arising in science and engineering, for example in cloud microphysics \cite{pruppacher1980microphysics,Seifert06a}, or 3D printing \cite{gibson2014additive}. We consider
models in which fragmentation does not occur and refer to the resulting
process as one of coalescence. The transport equations in Eq.~\eqref{eq:ce} below describe the evolution of the coalescence of particles or droplets, by tracking the evolution of the moments $x_k$ of the mass distribution: 
\begin{equation}
\label{eq:ce}
\frac{d x_k}{d t} = \frac{1}{2}\int_0^\infty \int_0^\infty \left((m+m')^k - m^k - {m'}^k\right) \mathcal{C}(m, m')f(m)f(m') \, \text{d}m\, \text{d}m'.
\end{equation}
Here $f(\cdot)$ denotes the mass distribution, and the kernel $\mathcal{C}$ describes the probability of coalescence of two particles or droplets with masses $m$ and $m'$. We employ the polynomial kernel $\mathcal{C}$ defined via non-negative
weights $\{c_{ab}\}$ and the non-negative integer $r$:
\begin{equation}
\label{eq:ce-kernel}
\mathcal{C}(m,m') = \sum_{a,b=0}^{r} c_{ab} m^a{m'}^b. 
\end{equation}

By subsitituting Eq.~\eqref{eq:ce-kernel} into the governing equations~\eqref{eq:ce}, and truncating to consider only moments $k=1,\cdots, K$,
we derive a modeled system to describe the evolution of moments: 
\begin{equation}
\label{eq:cem}
\begin{aligned}
\frac{d X_k}{d t} &=
\frac{1}{2} \sum_{a,b=0}^{r} \sum_{j=1}^{k-1} c_{ab} {k \choose j} X_{a+j} X_{b+k-j}, \quad k=2,3,...,K, \\
\frac{d X_k}{d t} &=
\begin{cases}
-\frac{1}{2} \sum_{a,b=0}^{r} c_{ab} X_a X_b & \quad k=0,\\
0 & \quad k = 1.
\end{cases}
\end{aligned}
\end{equation}
It should be noted that Eq.~\eqref{eq:cem} is not a closed system:
$X_\ell$ for $\ell \in \{K+1,...,K+r-1\}$ are needed in the modeled system. We base the closure model for these higher-order moments on the fitting of a Gamma distribution for $f(\cdot)$.
The resulting moment-based coalescence equation with polynomial kernel and
Gamma distribution closure is proposed in~\cite{BBJ}. 

Since the mass of the system is $X_0$, all
integrals should be normalized by this number to have
the standard probabilistic interpretation. Then, the mean
of this probabilistic distribution is $X_1/X_0$ and the
variance is $X_2/X_0-(X_1/X_0)^2$. If $\kappa$ and $\eta$ are the shape and scale parameters of the Gamma distribution, 
$\kappa\eta$ is the mean and $\kappa\eta^2$ is the variance.
This leads to the following Gamma distribution closure, 
noting that $\Gamma(n)=(n-1)!\,$:
\begin{equation}
\begin{aligned}
\label{eq:gamma}
X_k = X_0\eta^k\frac{\Gamma(\kappa+k)}{\Gamma(\kappa)}, \quad k > K, \\
\kappa = \frac{X_1^2}{X_0X_2-X_1^2}, \quad \eta = \frac{X_2}{X_1}-\frac{X_1}{X_0}.
\end{aligned}
\end{equation}
We study the modeled system in Eq.~\eqref{eq:cem}, 
\eqref{eq:gamma} with $K=2$ and $r=3$; thus, we use the closure
to determine the variables $X_3, X_4$ in terms of the primary
moments $X_0, X_1$ and $X_2$.
Our goal in this example is to learn a sparse solution of coefficients $c_{ab}$ in Eq.~\eqref{eq:cem} based on the data in the following  different settings:

\begin{itemize}
    \item a simulation study where data are generated by and fitted to the model in Eq.~\eqref{eq:cem} with $K=2$ and $r=3$, and with the Gamma distribution closure in Eq.~\eqref{eq:gamma};
    \item data are generated by the model in Eq.~\eqref{eq:cem} with $K>2$ and $r=3$, and we fit a model for $K=2$ and $r=3$, both using the Gamma distribution closure in Eq.~\eqref{eq:gamma};
    \item data are generated by the model in Eq.~\eqref{eq:cem} with $K=2$ and $r=3$ and an exponential closure distribution, and we fit a model for $K=2$ and $r=3$,  with the Gamma distribution closure in Eq.~\eqref{eq:gamma}.
\end{itemize}

For the last bullet we note that the exponential distribution closure has the following form:
\begin{equation}
\begin{aligned}
\label{eq:exp}
X_k &= X_0\frac{k!}{\mu^k}, \quad k > 2, \\
\mu &= \frac{X_0}{X_1},
\end{aligned}
\end{equation}
where $\mu$ denotes the rate parameter of an exponential distribution, chosen to agree with information present
in $X_0$ and $X_1.$

To prevent unphysical responses, we constrain the
parameters $\kappa$ and $\eta$ of the Gamma distribution to prescribed intervals. In so doing, we obtain a closed pair of equations for $(X_0,X_2)$ 
of the form
\begin{equation}
\begin{aligned}
\label{eq:gamma3}
\dot{X}_0=-\frac12 \sum_{a,b=0}^{3} c_{ab} X_a X_b,\\
\dot{X}_2= \sum_{a,b=0}^{3} c_{ab} X_{a+1} X_{b+1},
\end{aligned}
\end{equation}
with $X_1(t) \equiv X_1(0)$ and
\begin{equation}
\begin{aligned}
\label{eq:gamma2}
X_k &= X_0\eta^k\frac{\Gamma(\kappa+k)}{\Gamma(\kappa)}, \quad k=3,4, \\
\kappa' &= \frac{X_1^2}{X_0X_2-X_1^2},\\
\eta' &= \frac{X_2}{X_1}-\frac{X_1}{X_0};\\
\kappa &= \max(\min(\kappa',\kappa_{\max}),\kappa_{\min}), \\
\eta &= \max(\min(\eta',\eta_{\max}),\eta_{\min}).
\end{aligned}
\end{equation}

The moment $X_1$ is a constant which, throughout the simulations
in this paper, is set to be $2$. Furthermore we take $\kappa_{\min}$ and $\kappa_{\max}$ to be $10^{-3}$ and $10$, and $\eta_{\min}$ and $\eta_{\max}$ are set to $10^{-3}$ and $1$.

\begin{proposition}
\label{prop:3}
Let $X_0, X_1, X_2$ be non-negative at $t=0$, assume that  $c_{11}=0$ and that $c_{ab} \ge 0$ for $0 \le a,b \le 4$ and that $0<\kappa_{\min}<\kappa_{\max}<\infty$,
$0<\eta_{\min}<\eta_{\max}<\infty$.
Let $T \in (0,\infty]$ be the first time at which $X_0$
or $X_2$ becomes zero. Then, the equations \eqref{eq:gamma3} and \eqref{eq:gamma2} for $u=(X_0,X_2)$ have a unique solution $u \in C^1([0,T];\mathbb{R}^2).$  $\quad\diamond$
\end{proposition}

\begin{proof}
Recall that $X_1$ is constant in time. We consider solutions
only in the time-interval $[0,T].$
The imposed upper and lower bounds on $\kappa$ and $\eta$
ensure that the closure model has the property that there
is a universal constant $\mathsf{c} \in (0,1)$ such that, whilst a solution
to the equations for $X_0,X_2$ exists (and necessarily 
remains non-negative) in $[0,T]$,
$$\mathsf{c}X_0 \le X_3 \le \mathsf{c}^{-1}X_0, \quad
\mathsf{c}X_0 \le X_4 \le \mathsf{c}^{-1}X_0.$$
It follows from the equation for $X_0$ that, for $t \in [0,T],$
$$0 \le X_0(t) \le X_0(0)$$
since all quantities on the right hand-side of the equation for $\dot{X}_0$ in \eqref{eq:gamma3} are negative.
Now note that the right hand side of the equation for $\dot{X}_2$ in \eqref{eq:gamma3} contains no quadratic terms
in $X_2$, since $c_{11}=0$, and that the constant
term and linear coefficient (with respect to $X_2$) on the
right hand side are, for $t \in [0,T],$ bounded. Multiplying
the equation by $X_2$ shows that $V(X_2)=|X_2|^2$ satisfies
$$\frac{d}{dt}V(X_2) \le \alpha+\beta V(X_2)$$
with $\alpha, \beta$ determined by the initial conditions
and $\mathsf{c}.$ Hence $X_2$ cannot blow-up in $[0,T]$ and the result is proved.
\end{proof}

The total number of unknowns to be learned is thus $9$, after
imposing symmetry on $c_{ab}$ and setting $c_{11}$ to zero.
We also have a positivity constraint on
all of the unknowns. Adding a further constraint $c_{22}=0$ can be used
to prevent $X_0$ from becoming negative, thereby extending the preceding
proposition to hold for all $t \ge 0.$ In practice, however, we find
that the sparse solution always enforces $c_{22}=0$ and that we do not
need to impose it.
\end{example}

\begin{example}[{\bf {Kuramoto-Sivashinsky Equation}}]
Let $\mathbb{T}^L$ denote the torus $[0,L]$ and consider the equation
\begin{equation}
\begin{aligned}
\label{eq:need1}
\partial_t u&=-\partial^4_x u-\partial^2_x u-u \partial_x u, \quad x\in \mathbb{T}^L,\\
u|_{t=0}&=u_0.
\end{aligned}
\end{equation}
We are interested in learning this model from a library of
equations of the form
\begin{equation}
\begin{aligned}
\label{eq:need2}
\partial_t u&=-\sum_{j=1}^5 \Bigl(\alpha_j \partial^j_x u +
\beta_j u^{j} \partial_x u\Bigr),\quad x\in \mathbb{T}^L,\\
u|_{t=0}&=u_0.
\end{aligned}
\end{equation}
Thus, we have 10 unknowns. We wish to constrain the model so that solutions do not blow up
in finite time. To do this, we ensure that $\|u\|_{L^2(\mathbb{T};\mathbb{R})}$ remains bounded. Note that the nonlinear
terms disappear when multiplied by $u$ and integrated over
$\mathbb{T}^L$, because of periodicity. On the other hand, the linear term will damp all high spatial frequencies, uniformly in wave-number sufficiently
large, provided that $\alpha_4>0.$ Thus we have:

\begin{proposition}
\label{prop:4}
Let $\alpha_4>0$. Then there is constant $\mathsf{c}>0$ 
such that the solution of equation \eqref{eq:need2}, if it
exists as a function $u \in C([0,T];L^2(\mathbb{T};\mathbb{R}))$, satisfies
$$\sup_{t \in [0,T]} \|u\|^2_{L^2(\mathbb{T};\mathbb{R})} \le \|u_0\|^2_{L^2(\mathbb{T};\mathbb{R})}e^{\mathsf{c}T}.\quad\diamond$$
\end{proposition}
\end{example}

\begin{proof}
Let
$$V(u)=\frac12\|u\|^2_{L^2(\mathbb{T};\mathbb{R})}$$
Using periodicity we note that
\begin{equation*}
\int_{\mathbb{T}} u
\times \beta_j u^{j} \partial_x u\,dx=0
\end{equation*}
and that, for $j$ odd,
\begin{equation*}
\int_{\mathbb{T}} u
\times \alpha_j \partial_x^j u\,dx=0.
\end{equation*}
Thus we obtain, provided a solution exists,
$$\frac{d}{dt}{V(u)} \le -\alpha_2\int_{\mathbb{T}} u \times 
\partial_x^2 u\,dx - \alpha_4\int_{\mathbb{T}} u \times 
\partial_x^4 u\,dx.$$
A straightforward calculation based on the spectrum of
the operator $L$ defined by
$Lu:=\alpha_2 \partial_x^2 u+\alpha_4 \partial_x^4 u$
on $\mathbb{T}$ shows that there is constant $\mathsf{c}>0$ such that
$$\frac{d}{dt}{V(u)}= -\int_{\mathbb{T}} u \times Lu\, dx \le  \mathsf{c} V(u),$$
and the proof is complete.
\end{proof}

In summary we have 10 unknowns, and a positivity constraint on one of the unknowns. 
Although we ensure that the $\ell_2$-norm of the simulated state remains bounded, it is still possible that the simulated state may blow-up due to numerical discretization. Therefore, we also implement numerical clipping to bound the simulated state at every time step. Details concerning the numerical
solution of the K-S equation, including the clipping used, are presented in~\ref{sec:KS_solver}.

\section{Algorithms}
\label{sec:Alg}

Recall the inverse problem of interest, encapsulated in \eqref{eq:ip}.
In practice the data we are given, $y \in \bbR^J$, is a noisy evaluation of the function $\cG(\theta)$. We use the notation $G(\theta)$ to denote this noisy
evaluation, and we typically envision the noisy evaluation coming
from finite-time averaging. Appealing to central-limit theorem type results,
which quantify rates of convergence towards ergodic averages, we assume that
$\cG(\theta)-G(\theta)$ is Gaussian. Then the inverse problem can be formulated 
as follows: given $y \in \bbR^J$, find $\theta \in \Theta$ so that
\begin{equation}
\label{eq:IP}
y=G(\theta)+\eta, \quad \eta \sim N(0,\Gamma).
\end{equation}
The natural objective function associated with the inverse
problem \eqref{eq:IP} is
\begin{equation}
\label{eq:ob}
\frac12\bigl|\Gamma^{-\frac12}\bigr(y-G(\theta)\bigr)\bigr|^2.
\end{equation}

We study the use of EKI based methods for solving
this inverse problem. The four primary reasons for using EKI to solve the inverse problem are: (a) EKI does not require derivatives, which are difficult to compute in this
setting of SDEs; nonetheless, EKI provably,
in the linear case \cite{schillings2017analysis}, and approximately, in the nonlinear case \cite{garbuno2020interacting} behaves like
a gradient descent with respect to objective \eqref{eq:ob}, projected into a finite dimensional space defined by the ensemble; 
(b) EKI is robust to noisy evaluations of the forward map as shown in \cite{DSW}. (c) EKI is inherently parallelizable and scales well to high-dimensional unknowns \cite{Kalnay}. (d) EKI lends itself 
naturally to the imposition of constraints on the unknown 
parameter \cite{albers2019ensemble}.  
The primary novelty of the approach proposed in this paper
is the demonstration that imposition
of sparsity within EKI is something that can be achieved easily
and that enables generalization of the approach pioneered
in \cite{brunton2016discovering} to settings in which the
observations are nonlinear and partial and in which the
dynamical system comes from an SDE. Such problems lead to
the parameter learning of $\theta$ from $y$ related by \eqref{eq:IP}.

In subsection \ref{ssec:EKI}, we recap the basic EKI algorithm to fit unknown parameters. In particular, we demonstrate how the
iterative algorithm, which is nonlinear, has at its core
a quadratic optimization problem. In subsection \ref{ssec:sEKI}, we describe a method of inducing sparsity within the EKI algorithm, by introducing an $\ell_1$ constraint and/or an $\ell_0$-type penalty 
on top of the core optimization task solved by the basic EKI algorithm. In subsection \ref{ssec:qp-lasso}, we describe how we reformulate the sparsity inducing step of the algorithm as a standard quadratic programming problem.

There are many variants on the precise manner in which sparsity constraints
are imposed, and the algorithms used to solve the resulting optimization
problems. Our purpose is not to determine the best way to impose
sparsity or to the best way to solve the optimization problems: these
are well-studied problems and we simply employ some successful approaches
to their resolution. Rather, our purpose is to
demonstrate that the EKI approach to parameter learning is easily
extended to incorporate sparisty constraints because its core
is solution of a quadratic optimization problem. 

\subsection{Ensemble Kalman Inversion}
\label{ssec:EKI}

The ensemble Kalman inversion (EKI) algorithm that we employ to solve the inverse problem in \eqref{eq:IP} is described
in~\cite{iglesias2013ensemble,albers2019ensemble}. First, we introduce a new variable $w=G(\theta)$ and variables $v$ and $\Psi(v)$:
\begin{equation}
\begin{aligned}
v&=(\theta,w)^\top, \\
\Psi(v)&=\left(\theta,G(\theta)\right)^\top.
\end{aligned}
\end{equation}
Using these variables we formulate the following noisily observed dynamical system:
\begin{equation}
\begin{aligned}
v_{m+1}&=\Psi(v_m) \\
y_{m+1}&=Hv_{m+1}+\eta_{m+1}.
\end{aligned}
\end{equation}
Here $H=[0,I], H^\perp=[I,0]$, and hence $Hv=w, H^\perp v=\theta.$ 
In this setting, $\{v_m\}$ is the state and $\{y_m\}$ are the
data. The objective is to estimate $H^\perp v_m=\theta_m$ from $\{y_\ell\}_{\ell=1}^{m}$ and to do so iteratively with respect to $m$. In practice we only have one data point
$y$ and not a sequence $y_m$; we address this issue in what follows below.

The EKI methodology creates an ensemble  $\{v_{m}^{(j)}\}_{j=1}^J$ defined iteratively  in $m$ as follows:
\begin{equation}
\label{eq:optim-EKI}
\begin{aligned}
    J_m^{(j)}(v):=&\frac12\big|y^{(j)}_{m+1}-Hv\big|^2_{\Gamma}+\frac12\big|v-\Psi\bigl({v}_{m}^{(j)}\bigr)\big|^2_{C^{\Psi\Psi}_{m}},\\
    v_{m+1}^{(j)}=&\argmin_v J_m^{(j)}(v).
\end{aligned}
\end{equation}
The matrix $C^{\Psi\Psi}$ is the  empirical covariance of $\{\Psi(v_m^{(j)})\}_{j=1}^J$. The data $y^{(j)}_{m+1}$ is either fixed
so that $y^{(j)}_{m+1}\equiv y$
or created by adding random draws to $y$ from the distribution of
the $\eta$, independently for all $m$ and $j$.
At each step, $m$ ensemble parameter estimates 
indexed by $j=1,\cdots, J$
are found from $\theta_{m}^{(j)}=H^\perp v_{m}^{(j)}.$

Using the fact that $v=(\theta,w)^T$, the minimizer 
$v_{m+1}^{(j)}$ in \eqref{eq:optim-EKI} 
decouples to give the update formulae
\begin{equation}
\label{eq:EKI}
\theta_{m+1}^{(j)}=\theta_{m}^{(j)}+C_m^{\theta G}\left(C_m^{GG}+\Gamma \right)^{-1}\left(y_{m+1}^{(j)}-G(\theta_m^{(j)}) \right),
\end{equation}
 the matrix $C_m^{GG}$ is the empirical covariance of $\{G(\theta_m^{(j)})\}_{j=1}^J$,
while matrix $C_m^{\theta G}$ is the empirical cross-covariance of 
$\{\theta_m^{(j)}\}_{j=1}^J$ with $\{G(\theta_m^{(j)})\}_{j=1}^J$.
Details of the derivation may be found 
in \cite{iglesias2013ensemble,albers2019ensemble}. 
The algorithm preserves the linear span of the initial ensemble $\{\theta_0^{(j)}\}_{j=1}^J$ 
for each $m$ and thus operates in a finite dimensional vector space, 
even if $\Theta$ is an infinite dimensional vector space.

\subsection{Sparse Ensemble Kalman Inversion (EKI)}
\label{ssec:sEKI}

We are interested in finding a sparse solution  $\theta$ of the inverse problem in \eqref{eq:IP}, building on the key features (a)--(d)
possessed by EKI and outlined in the preamble to this section. To impose sparsity on the solution of $\theta$ from EKI, we replace the step \eqref{eq:optim-EKI} with the step
\begin{equation}
\label{eq:optim-EKI-lasso}
\begin{aligned}
    J_m^{(j)}(v,\lambda):=&\frac12\big|y^{(j)}_{m+1}-Hv\big|^2_{\Gamma}+\frac12\big|v-\Psi\bigl({v}_{m}^{(j)}\bigr)\big|^2_{C^{\Psi\Psi}_{m}}+\lambda|H^\perp v|_{\ell_0},\\
    v_{m+1}^{(j)}=&\argmin_{v \in \cV} J_m^{(j)}(v),
\end{aligned}
\end{equation}
where
\begin{equation}
\label{eq:cset1}
\cV=\{v: |H^\perp v|_{\ell_1} \le \gamma\}.
\end{equation}
On occasion we will also impose positivity constraints on some
of the parameters and will then choose, for some matrix $A$,
\begin{equation}
\label{eq:cset2}
\cV=\{v: |H^\perp v|_{\ell_1} \le \gamma, \,AH^\perp v \ge 0\}.
\end{equation}
The parameters $\gamma$ and $\lambda$ may be adjusted and indeed
could be learned via cross-validation. 
To solve the resulting optimization problem, we alternate minimization of \eqref{eq:optim-EKI-lasso} for $\lambda=0$, which
approximates a gradient descent step for \eqref{eq:ob} subject to
an $\ell_1$ constraint, with a proximal gradient step on the $|\cdot|_{\ell_0}$ norm, projected into the
$\ell_1$ constraint set; however, the latter cannot
leave either of the constraint sets \eqref{eq:cset1} or \eqref{eq:cset2}, and so reduces to a simple thresholding. 
To this end, we introduce the function $\mathcal{T}$ on vectors defined by
 \begin{equation}
 \label{eq:threshold}
    \mathcal{T}(\theta_i)=
    \begin{cases}
      0, & \text{if}\ |\theta_i|<\sqrt{2\lambda} \\
      \theta_i, & \text{otherwise}
    \end{cases}
  \end{equation}
With this definition, we arrive Algorithm \ref{alg:a} below.

\begin{algorithm}
\caption{Sparse EKI algorithm}
\begin{algorithmic}[1]
\State Choose $\{\theta_m^{(j)}\}_{j=1}^J$ for $m=0$
\State $\{w_m^{(j)}=G(\theta_m^{(j)})\}_{j=1}^J$
\For{$j=1,2,...,J$}
        \State $v_{m+1}^{(j)} \leftarrow$ argmin of \eqref{eq:optim-EKI-lasso} with $\lambda=0$
        \State Extract $\theta_{m+1}^{(j)}=H^\perp v_{m+1}^{(j)}$
\State $\theta_{m+1}^{(j)}=\mathcal{T}(\theta_{m+1}^{(j)})$
\EndFor
\State $m \leftarrow m+1$, go to 2
\end{algorithmic}
\label{alg:a}
\end{algorithm}

We note that taking 
$\gamma=\infty$ results simply in an $\ell_0$ penalty and
alternation of standard EKI (which, recall, behaves like
a step of gradient descent) with a hard-thresholding algorithm. On 
the other hand, taking $\lambda=0$ results in a modification of
EKI that promotes smaller $\ell_1-$norm solutions.
 In the next subsection, we give details about how to formulate the optimization problem Eq.~\eqref{eq:optim-EKI-lasso} with
 $\lambda=0$ as a standard quadratic programming problem, rendering
 the preceding algorithm not only implementable, but efficient.
 
In practice, the coefficients $\theta^{(j)}$ identified by a single sparsity-promoting optimization, such as Algorithm \ref{alg:a}, will
exhibit bias. To enhance the performance of identifying the coefficients $\theta^{(j)}$, it is sometimes useful to run the sparse EKI Algorithm~\ref{alg:a} in multiple batches, removing unnecessary basis functions in each batch, until the number of basis functions cannot be further reduced. Similar concepts, employing multiple optimizations sequentially, are also advocated in~\cite{brunton2016discovering,schaeffer2017learning}. More specifically, iteratively thresholded least squares optimization is recommended in~\cite{brunton2016discovering}, that is, iteratively solving the least squares optimization on reduced basis functions identified by the optimization in
the previous step. On the other hand, a second least squares optimization restricted to the features identified from the original $\ell_1$ penalized least squares optimization is recommended in~\cite{schaeffer2017learning}.
There is also theoretical work related to debiasing the output of
sparse solution algorithms; see \cite{javanmard2018debiasing}. However
our approach is more closely linked to the ad hoc approaches advocated
in \cite{brunton2016discovering,schaeffer2017learning}.

\subsection{Quadratic Programming with $\ell_1$ Penalty}
\label{ssec:qp-lasso}
The objective function in \eqref{eq:optim-EKI-lasso} with
$\lambda=0$ can be rewritten (neglecting constants in $v$) as, for
$C_m=C_m^{\Psi\Psi}$,
\begin{equation}
\label{eq:optim-EKI-qp}
\frac{1}{2}v^\top\left(H^\top\Gamma^{-1}H+C_{m}^{-1}\right)v-
\left(C_{m}^{-1} \Psi(v_m^{(j)})+H^\top \Gamma^{-1} y^{(j)} \right)^\top v.
\end{equation}
We wish to minimize over $\cV$ defined
in \eqref{eq:cset2} (the case \eqref{eq:cset1} may be extracted
from what follows simply by setting $A=0$).
By appropriate definition of $Q$
and $q$, we may write the resulting minimization problem as
\begin{equation}
\label{eq:optim-l1}
\begin{aligned}
    \min_{v} \quad & \frac{1}{2}v^\top Q v + q^\top v \\
    \textrm{s.t.} \quad & AH^\perp v \geq 0, \, 
     |H^\perp v\big|_{\ell_1} \le \gamma.
\end{aligned}
\end{equation}
The following decomposition as described in \cite{tibshirani1996regression} can be employed to convert \eqref{eq:optim-l1} into the standard form of quadratic programming:
introduce variables
\begin{equation}
\begin{aligned}
v_i &= v_i^+ - v_i^-, \\
|v_i| &= v_i^+ + v_i^-,
\end{aligned}
\end{equation}
where $v_i^+ \geq 0$ and $v_i^- \geq 0$ denote the positive and negative part of the $i^\textrm{th}$ element of $v$, respectively. This decomposition leads to the following minimization problem:
\begin{equation}
\begin{aligned}
    \min_{v^+, v^-} \quad & \frac{1}{2}\left(v^+-v^-\right)^\top Q \left(v^+-v^-\right) + q^\top \left(v^+-v^-\right)\\
    \textrm{s.t.} \quad & AH^\perp\left(v^+-v^-\right) \geq a, \,
    H^\perp \left( v^+ + v^-\right) \le \gamma,\, v^+ \ge 0,\, v^- \ge 0.
\end{aligned}
\end{equation}
If we define the augmented vector $u^\top =\left[v^+,v^-\right]
\in \mathbb{R}^{2z}$, we see that the problem takes the form of a standard quadratic programming problem; alternatively one may work with the variable $u^\top=\left[\left(v^+-v^-\right)^\top, \left(v^+ + v^-\right)^\top \right] \in \mathbb{R}^{2z}$.

\begin{remark}{
It should be noted that the quadratic programming problem arising here
by application of the approach pioneered in \cite{tibshirani1996regression}
is a classic well-studied problem; we solve it using a state-of-the-art 
package~\cite{andersen2011interior}.
In addition to the approach we adopt to imposing sparsity, there are 
other methods that solve convex optimization problems subject to
$\ell_1$ regularization or penalty, such as in \eqref{eq:optim-EKI-lasso}. 
These include the alternating direction method of multipliers 
(ADMM)~\cite{boyd2011distributed} and split Bregman 
method~\cite{goldstein2009split}. Our formulation of the problem in 
\eqref{eq:optim-EKI-lasso} as a standard quadratic programming problem 
has the benefit that other equality and/or inequality constraints
are readily imposed on the EKI methodology,
along with the $\ell_1$ regularization of penalization, as explained 
in~\cite{albers2019ensemble}.}
\end{remark}

\begin{remark}
In some applications, it may be of interest to impose sparsity only
on a subset of the parameters $\theta$. The modification required
to do this is straightforward and so we do not detail it here. Such
a modification is employed in the next section when we learn a 
closure model for the Lorenz 96 multiscale equations.
\end{remark}

\section{Numerical Results}
\label{sec:N}
We demonstrate the capability of the proposed methodology by studying the four examples introduced in Section \ref{sec:PF}. In all cases, the unknown parameters
are detailed in Section~\ref{sec:PF} and the data used to learn
them are detailed in what follows. For the noisy Lorenz 63 system, we use the Euler-Maruyama method to solve the It\^o SDEs. For  the
Lorenz 96 systems, we use an adaptive numerical integrator~\cite{petzold1983automatic,hindmarsh1983odepack} that automatically chooses between the nonstiff Adams method and the stiff BDF method. For the coalescence equation, we use the fourth-order Runge--Kutta method. The numerical integrator for Kuramoto-Sivashinsky equation is presented in~\ref{sec:KS_solver}. 
In all cases, the results are initially presented in two figures, one showing the ability of the sparse EKI method to fit the data, and a second showing that the proposed methodology indeed provides a sparse solution in terms of the $\ell_1$ norm of redundant coefficients. For the canonical chaotic systems, we then show how well the fitted dynamical system performs in terms of reproducing the invariant measure and time correlation. For the coalescence equation, we show how well the fitted dynamical system performs in terms of reproducing the time trajectories of states with a different initial condition. For the Kuramoto-Sivashinsky equation, we present the results of an additional sparse EKI with reduced basis functions identified by the first sparse EKI. All these numerical studies confirm that the sparsity-promoting EKI is able to discover the governing equations of dynamical systems based on statistics derived from averaging time series.

\subsection{Lorenz 63 System}
\label{ssec:NL63}

We first study a noisy Lorenz 63 system for which the data are obtained by simulating \eqref{eq:l63}, with a given set of parameters $\alpha=10$, $\rho=28$, $\beta=8/3$, and $\sigma=10$. The goal is to fit a modeled system \eqref{eq:l63m} by learning unknown coefficients $\theta_{ki}$ and $\sigma$. In this study, $\phi=\{\phi_i \mid i \in \{1,...,9\}\}$ contains all the first ($\phi=\{\phi_i \mid i \in \{1,2,3\}\}$) and second order polynomial basis functions. It is well known that the existing sparsity-promoting model discovery frameworks such as SINDy~\cite{brunton2016discovering} would encounter some difficulties for such a system due to noise in the time trajectories. We show that the sparse EKI is able to learn a noisy chaotic system based on statistics derived from averaging time series. Results are presented in Figs.~\ref{fig:G_nL63} to~\ref{fig:autocorr_nL63}.

\begin{figure}[!htbp]
  \centering
  \includegraphics[width=0.2\textwidth]{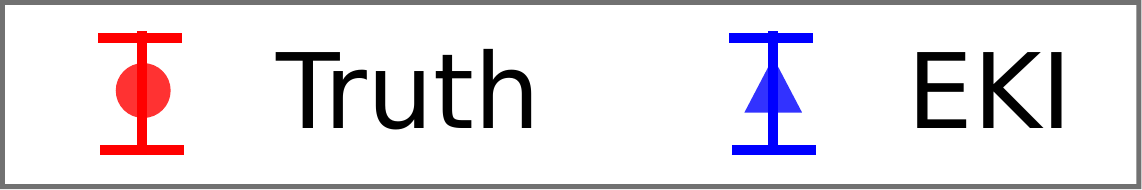}
  \subfloat[Standard EKI]{\includegraphics[width=0.49\textwidth]{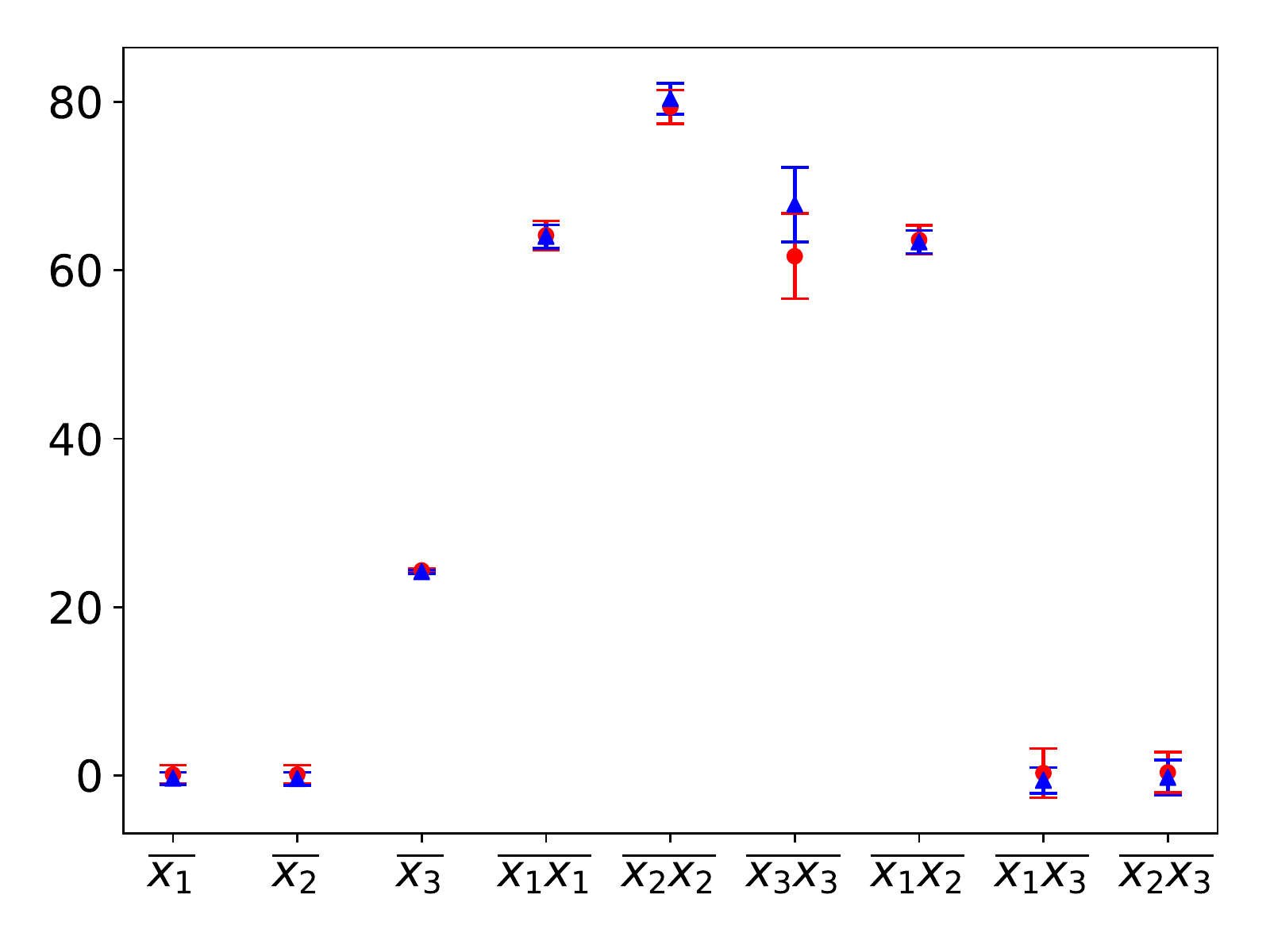}}
  \subfloat[Sparse EKI]{\includegraphics[width=0.49\textwidth]{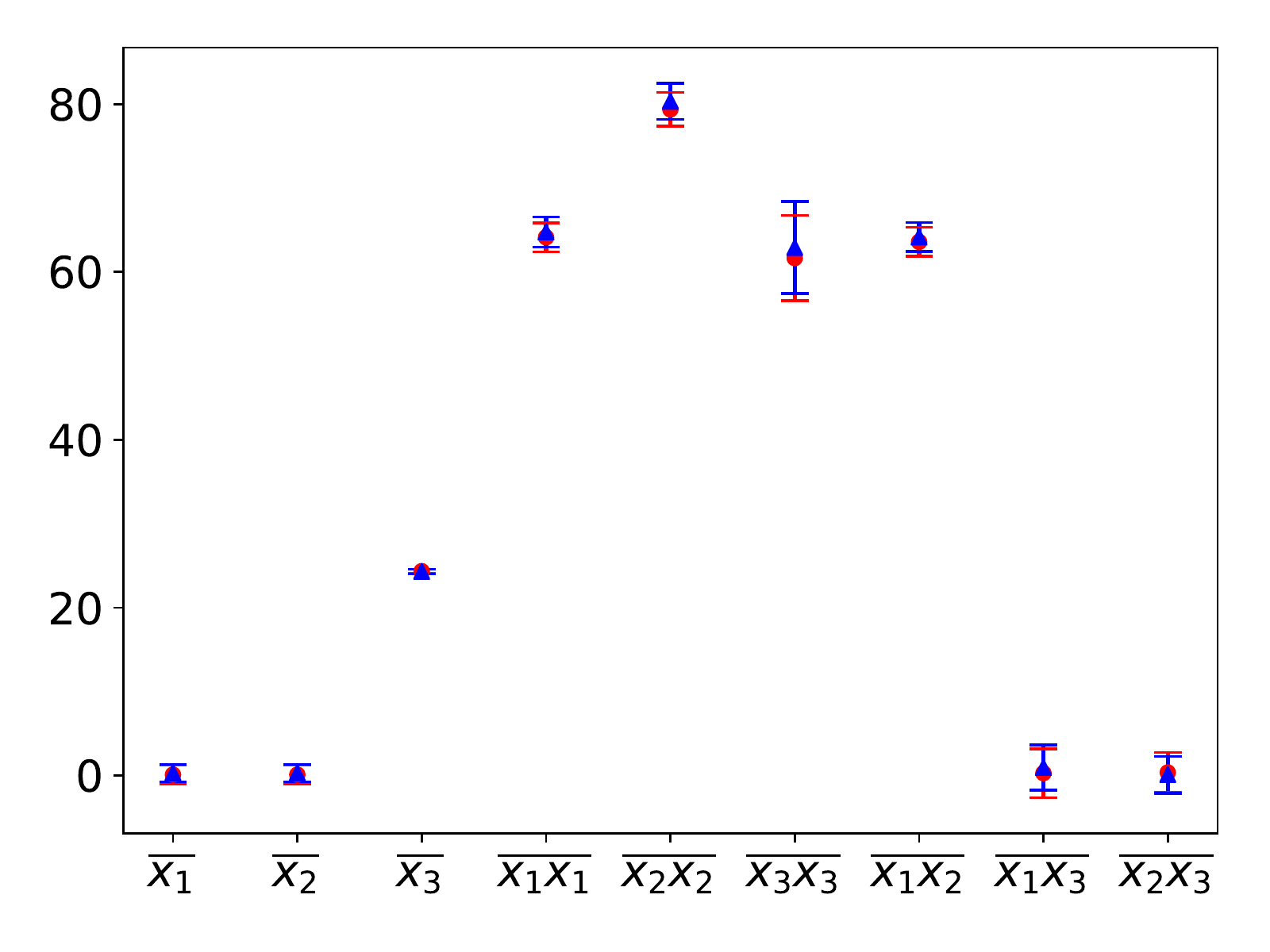}}
    \caption{First two moments of state $X$ for noisy Lorenz 63 system found by using (a) standard EKI and (b) sparse EKI.}
  \label{fig:G_nL63}
\end{figure}

The data in this case are finite-time averaged approximations of $\{\mathcal{G}_1(X),\mathcal{G}_2(X)\}$, i.e., first and second moments of simulated states. The time-interval used to gather time-averaged statistics is $T=100$. Therefore, we are learning 18 unknown coefficients (17 independent coefficients in $\{\theta_{ki}\}$ and a constant $\sigma$), using a data vector $y$ of dimension 9, as shown in Fig.~\ref{fig:G_nL63}. The comparison between the true data and the results of estimated systems in Fig.~\ref{fig:G_nL63} shows that the sparse EKI has slightly better agreement with the true data than does standard EKI.

\begin{figure}[!htbp]
  \centering
  \includegraphics[width=0.55\textwidth]{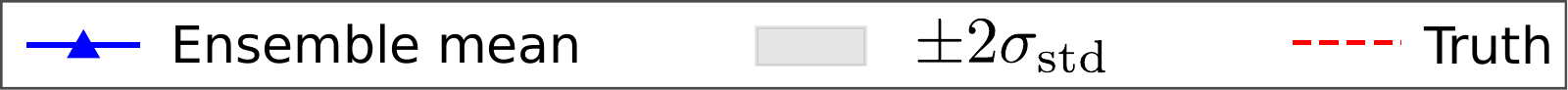}
  \subfloat[Standard EKI]{\includegraphics[width=0.49\textwidth]{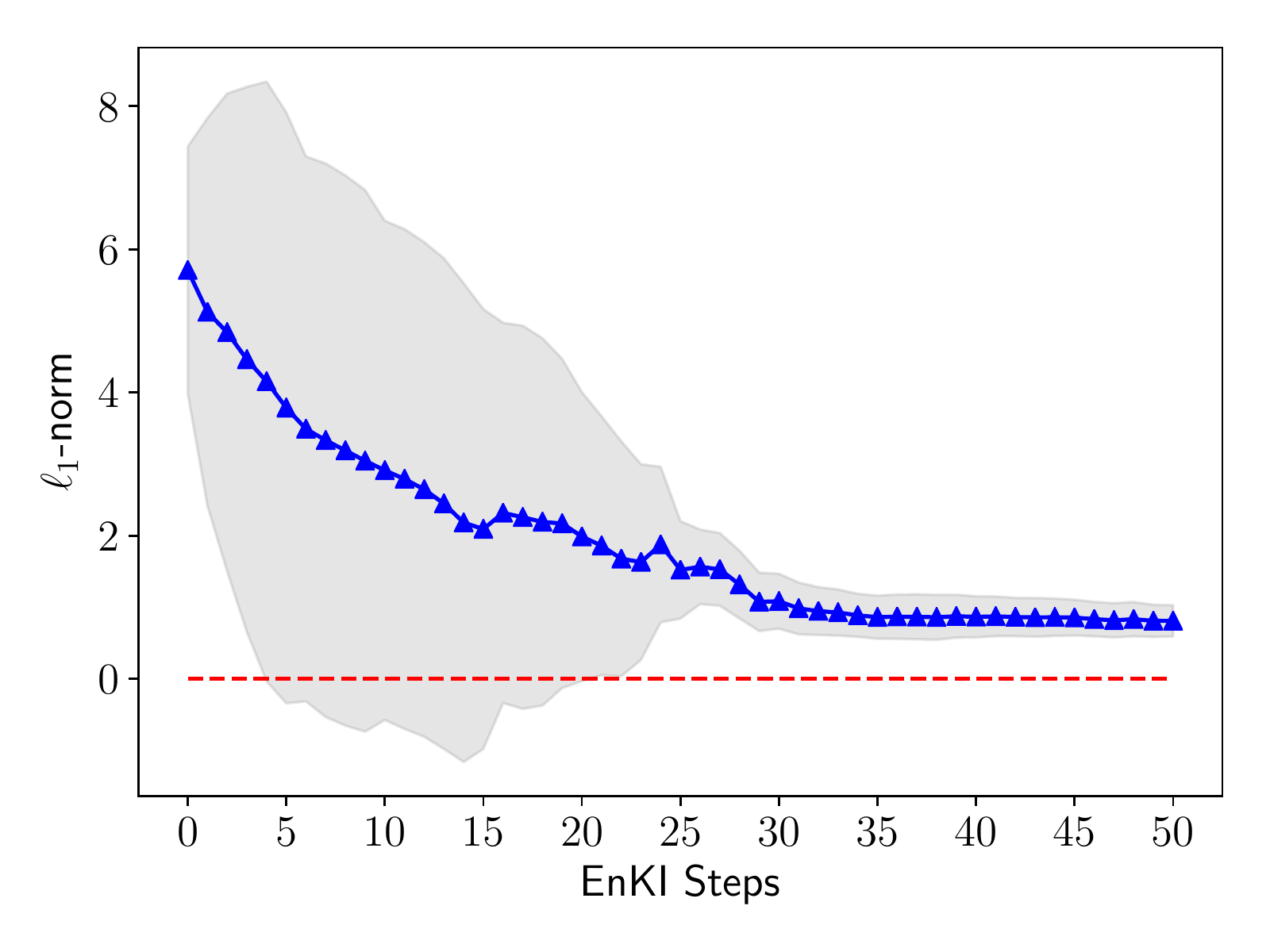}}
  \subfloat[Sparse EKI]{\includegraphics[width=0.49\textwidth]{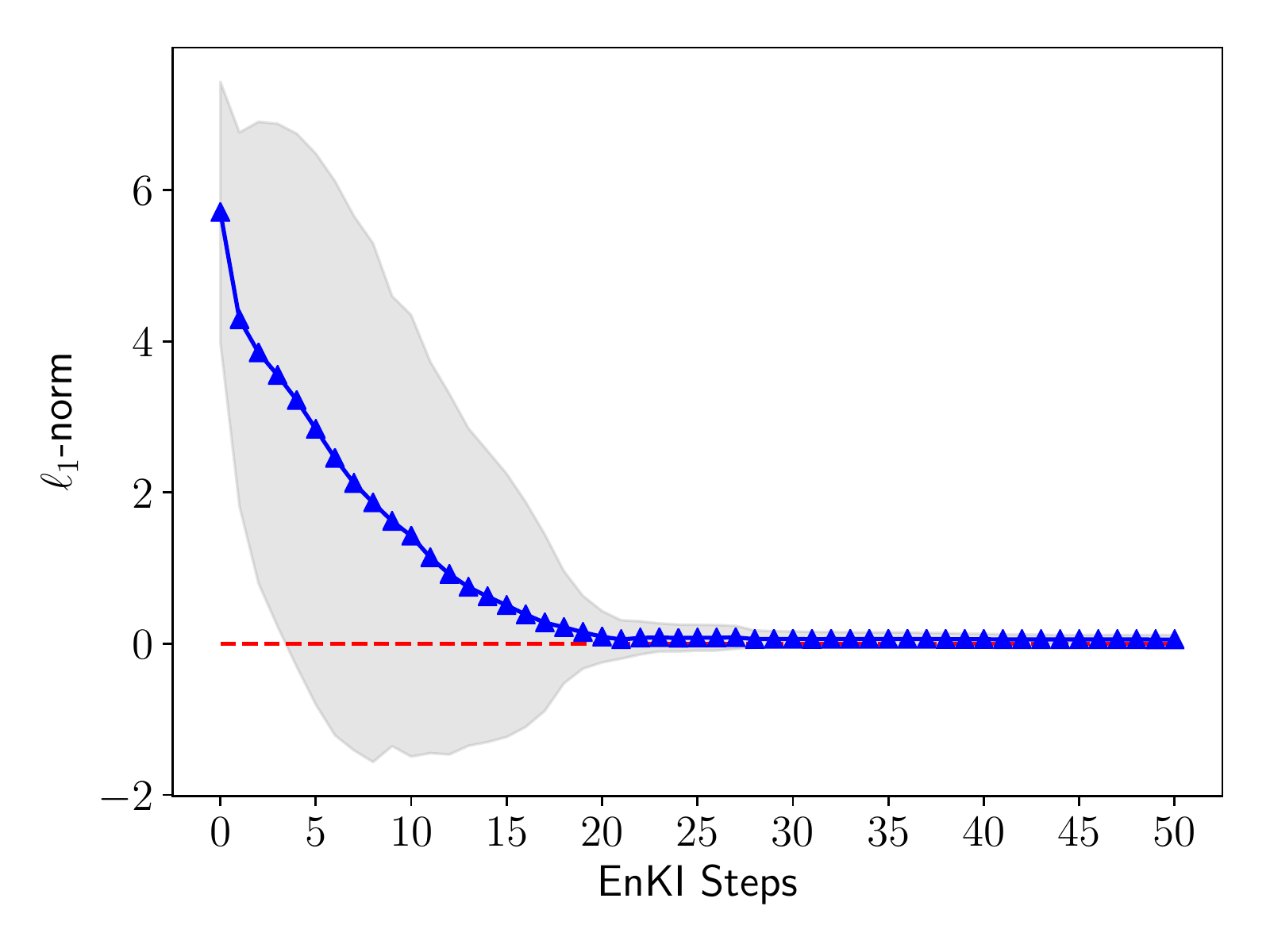}}
    \caption{$\ell_1$-norm of redundant coefficients for noisy Lorenz 63 system found by using (a) standard EKI and (b) sparse EKI.}
  \label{fig:l1_norm_nL63}
\end{figure}

The $\ell_1$-norm of all redundant coefficients also demonstrates the improved performance using sparse EKI, as presented in Fig.~\ref{fig:l1_norm_nL63}. Compared to the results of standard EKI, the $\ell_1$-norm of all redundant coefficients is driven much closer to zero using sparse EKI. The coefficients of redundant terms estimated by standard EKI are presented in Table~\ref{tab:coeffs-nL63}, where we can see that there are a few terms being identified with coefficients noticeably larger than zero, such as the
linear term $X_3$ in the equation for $X_1$. On the other hand, sparse EKI drives all coefficients of the redundant terms close to zero as shown by Fig.~\ref{fig:l1_norm_nL63}b, and the detailed results are omitted here for simplicity.

\begin{table}[htbp]
\caption{Mean value of coefficients estimated by standard EKI for all redundant terms.}
\centering
\begin{tabular}{c|cccccc}
\hline
& \multicolumn{5}{c}{Redundant terms} \\
\hline
Equation $X_1$ & $X_3$ & $X_2X_2$ & $X_3X_3$ & $X_1X_2$ & $X_1X_3$ & $X_2X_3$ \\
Coefficient & -0.229 & 0.026 & 0.011 & -0.062 & -0.093 & 0.094 \\
\hline
Equation $X_2$ & $X_3$ & $X_1X_1$ & $X_3X_3$ & $X_1X_2$ & $X_2X_3$ \\
Coefficient & 0.099 & 0.062 & -0.008 & -0.026 & -0.001 \\
\hline
Equation $X_3$ & $X_1$ & $X_2$ & $X_1X_1$ & $X_2X_2$ & $X_1X_3$ & $X_2X_3$ \\
Coefficient & -0.107 & -0.066 & 0.093 & 0.001 & -0.011 & 0.008 \\
\hline
\end{tabular}
\label{tab:coeffs-nL63}
\end{table}

\begin{figure}[!htbp]
  \centering
  \includegraphics[width=0.2\textwidth]{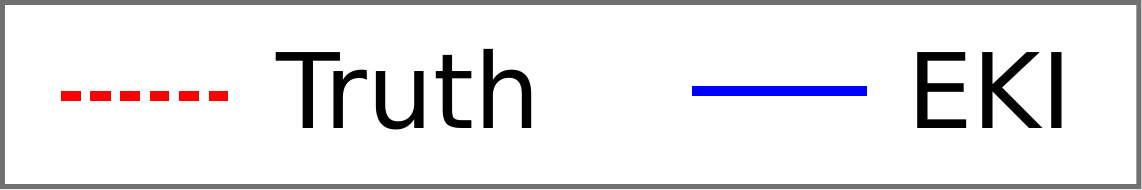}
  \subfloat[$X_1$ (Standard EKI)]{\includegraphics[width=0.33\textwidth]{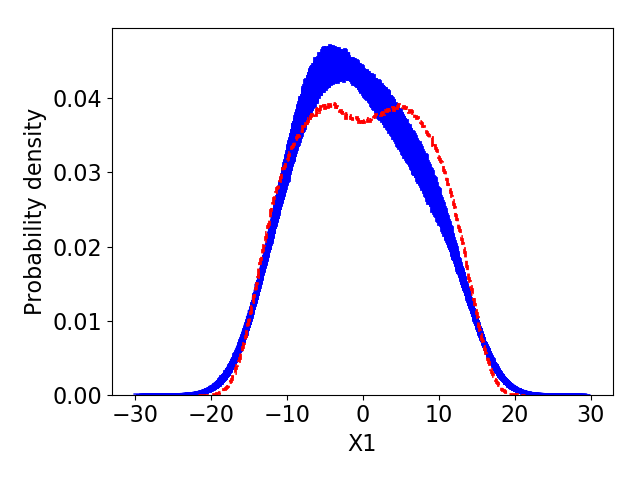}}
  \subfloat[$X_2$ (Standard EKI)]{\includegraphics[width=0.33\textwidth]{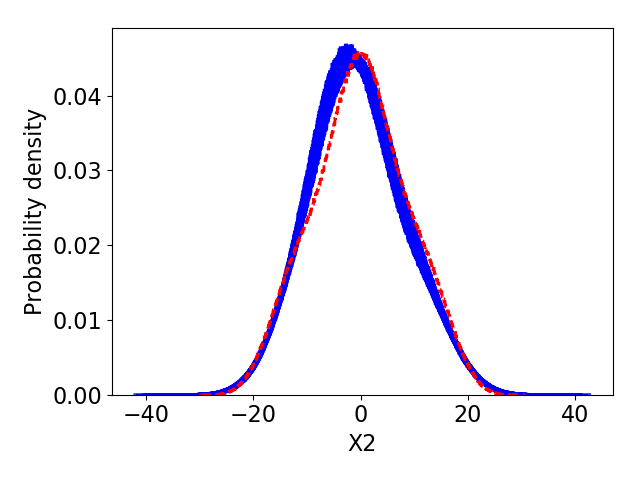}}
  \subfloat[$X_3$ (Standard EKI)]{\includegraphics[width=0.33\textwidth]{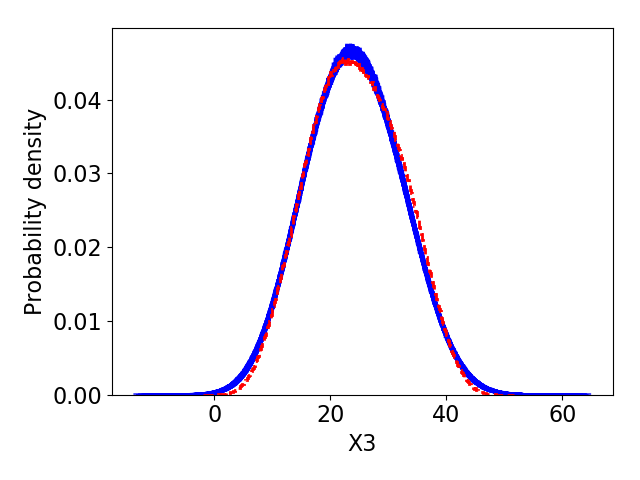}} \\
  \subfloat[$X_1$ (Sparse EKI)]{\includegraphics[width=0.33\textwidth]{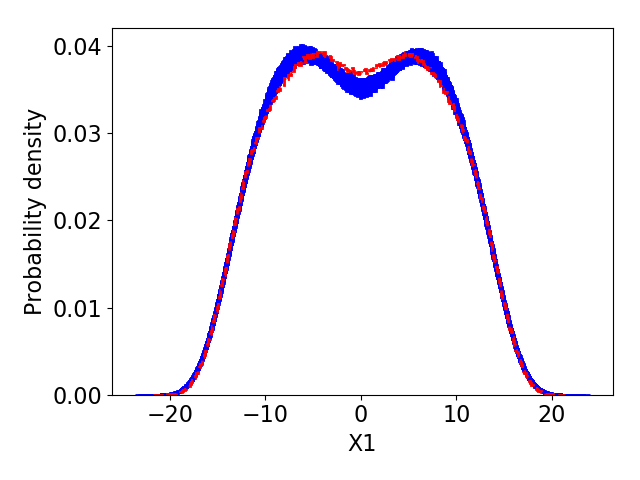}}
  \subfloat[$X_2$ (Sparse EKI)]{\includegraphics[width=0.33\textwidth]{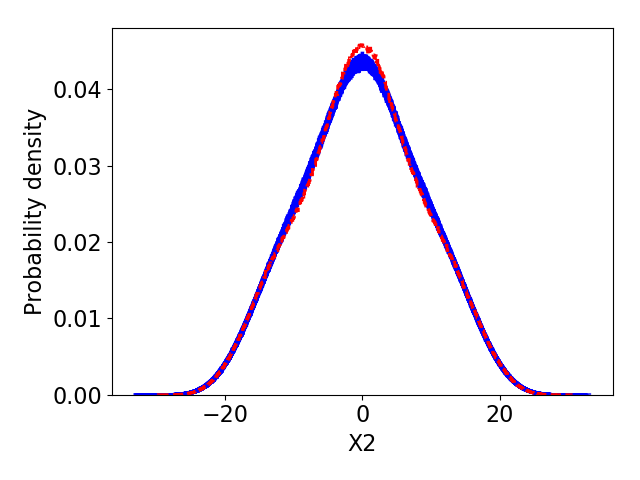}}
  \subfloat[$X_3$ (Sparse EKI)]{\includegraphics[width=0.33\textwidth]{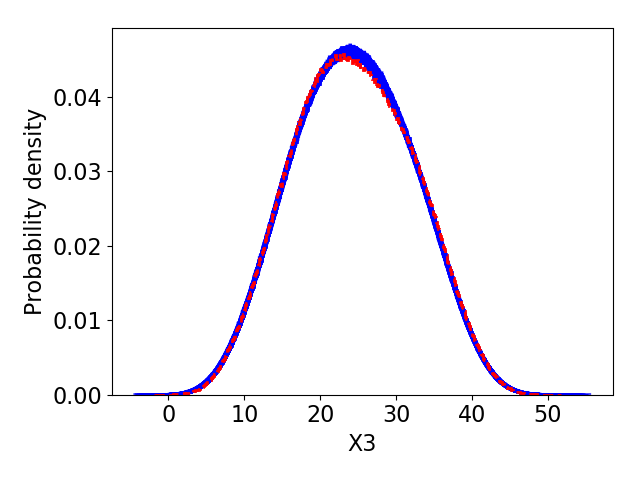}}
    \caption{Invariant measure for noisy Lorenz 63 system found by using (a-c) standard EKI and (d-f) sparse EKI.}
  \label{fig:measure_nL63}
\end{figure}

We further investigate the performances of the estimated systems by evaluating the invariant measure. As presented in Fig.~\ref{fig:measure_nL63}, the results of sparse EKI show better agreement with the true invariant measure for all three states, confirming the improved performance of the sparse EKI-estimated system over that found from standard EKI, in the long time limit. The comparison of autocorrelation functions is presented in Fig.~\ref{fig:autocorr_nL63}, demonstrating a good agreement of both EKI estimated systems with the true system in terms of time correlation. When comparing results in Fig.~\ref{fig:autocorr_nL63}, the time correlation results of sparse EKI show slightly better performance than the ones of standard EKI, especially for the simulated states $X_1$.

\begin{figure}[!htbp]
  \centering
  \includegraphics[width=0.4\textwidth]{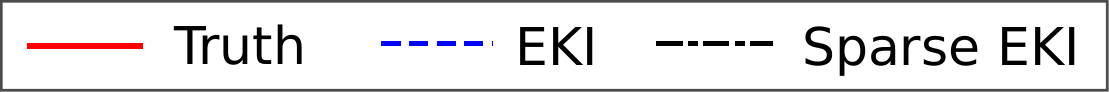}
  \subfloat[$X_1$]{\includegraphics[width=0.33\textwidth]{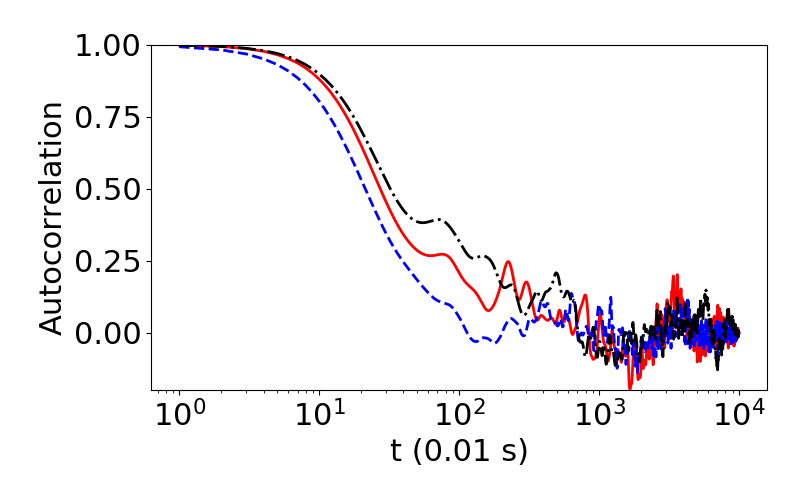}}
  \subfloat[$X_2$]{\includegraphics[width=0.33\textwidth]{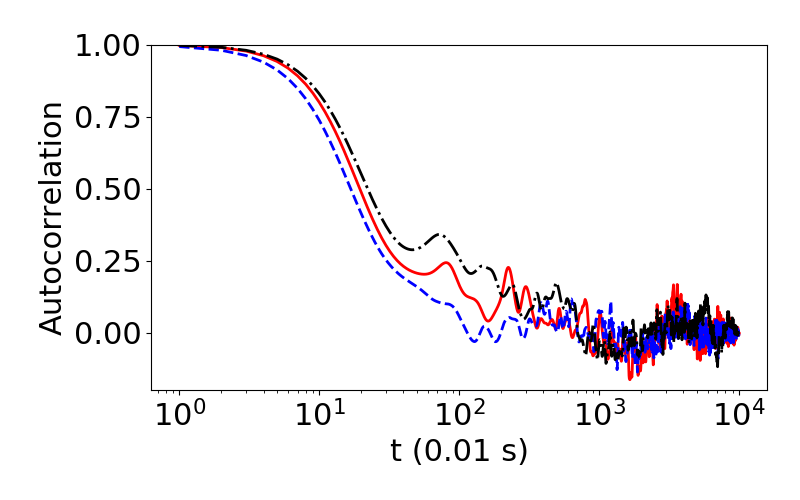}}
  \subfloat[$X_3$]{\includegraphics[width=0.33\textwidth]{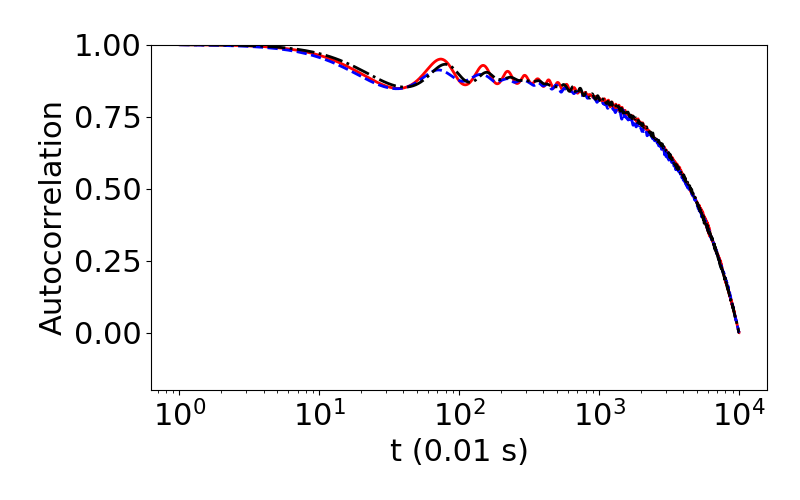}}
    \caption{Autocorrelation for noisy Lorenz 63 system found by using standard EKI and sparse EKI.}
  \label{fig:autocorr_nL63}
\end{figure}

\begin{remark}
For this example we studied in detail the choice of the $\ell_1$ penalty 
parameter $\gamma$ arising in \eqref{eq:optim-l1} and the thresholding parameter $\lambda$ in \eqref{eq:threshold}. As presented in 
Fig.~\ref{fig:hyperparams_nL63}a, smaller $\gamma$ tends to provide smaller 
$|\theta|_{\ell_1}$, which indicates a simpler model. The data mismatch dramatically increases when $\gamma$ is less than 
around $\gamma=50$, demonstrating underfitting. On the other hand, 
both data mismatch and model complexity increase slowly for $\gamma \ge 60$, 
showing that the performance of the proposed method is not sensitive 
to the choice of $\gamma$ provided it is sufficiently large. The comparison of results using different thresholding parameter $\lambda$ in Fig.~\ref{fig:hyperparams_nL63}b shows that the performance of the proposed method is not sensitive to $\lambda$. In the example 
of the noisy Lorenz 63 system we choose $\gamma=60$ and $\lambda=0.1$, based on these
observations. Since the proposed method is not sensitive to $\gamma$ outside 
the underfitting regime and $\lambda$, the parametric study is not presented 
in other examples for simplicity. It would, of course, be useful 
to automate the choice of $\gamma$ and $\lambda$ via cross-validation.
\end{remark}

\begin{figure}[!htbp]
  \centering
  \subfloat[$\ell_1$ penalty coefficient]{\includegraphics[width=0.5\textwidth]{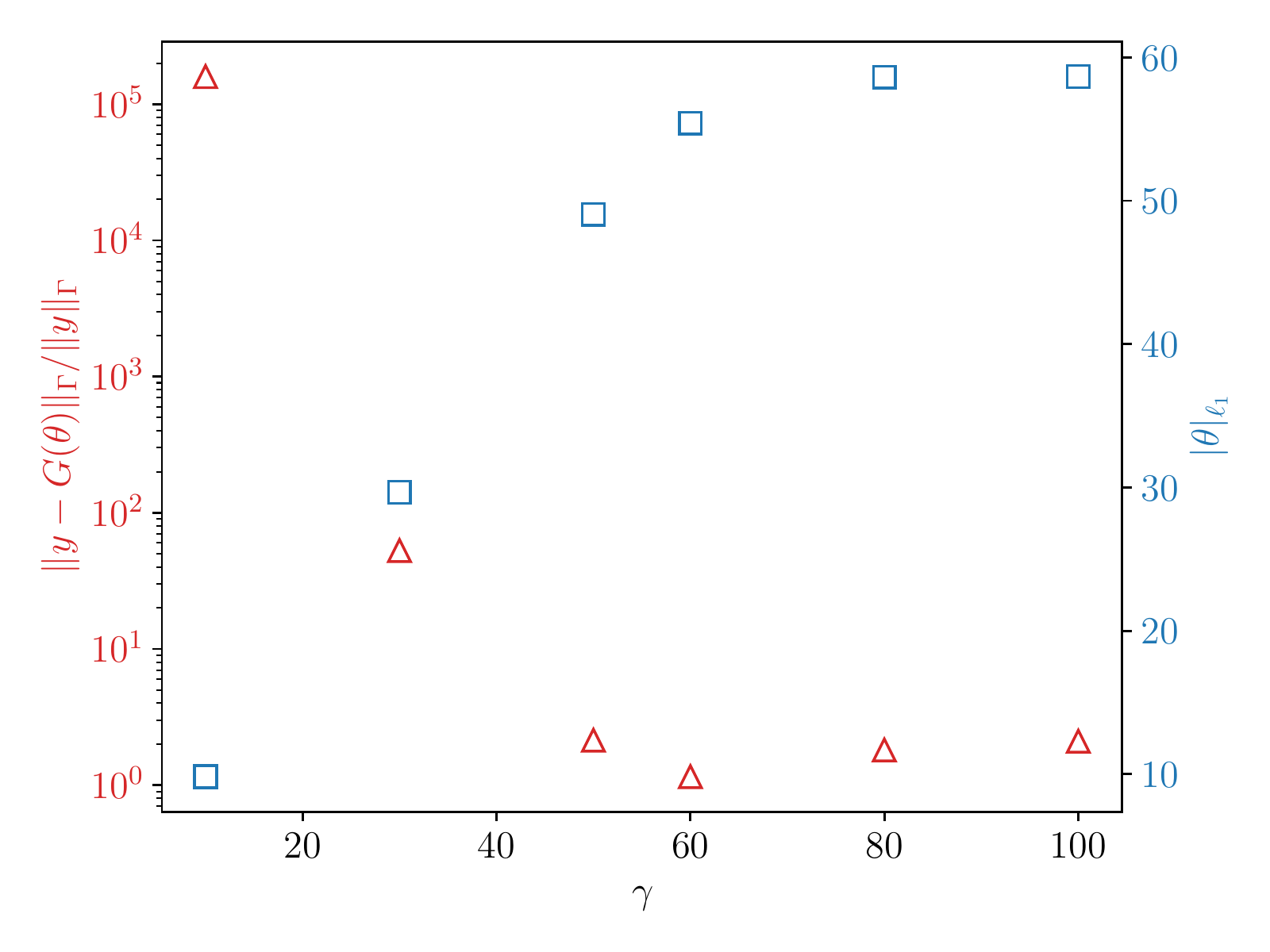}}
  \subfloat[thresholding parameter]{\includegraphics[width=0.5\textwidth]{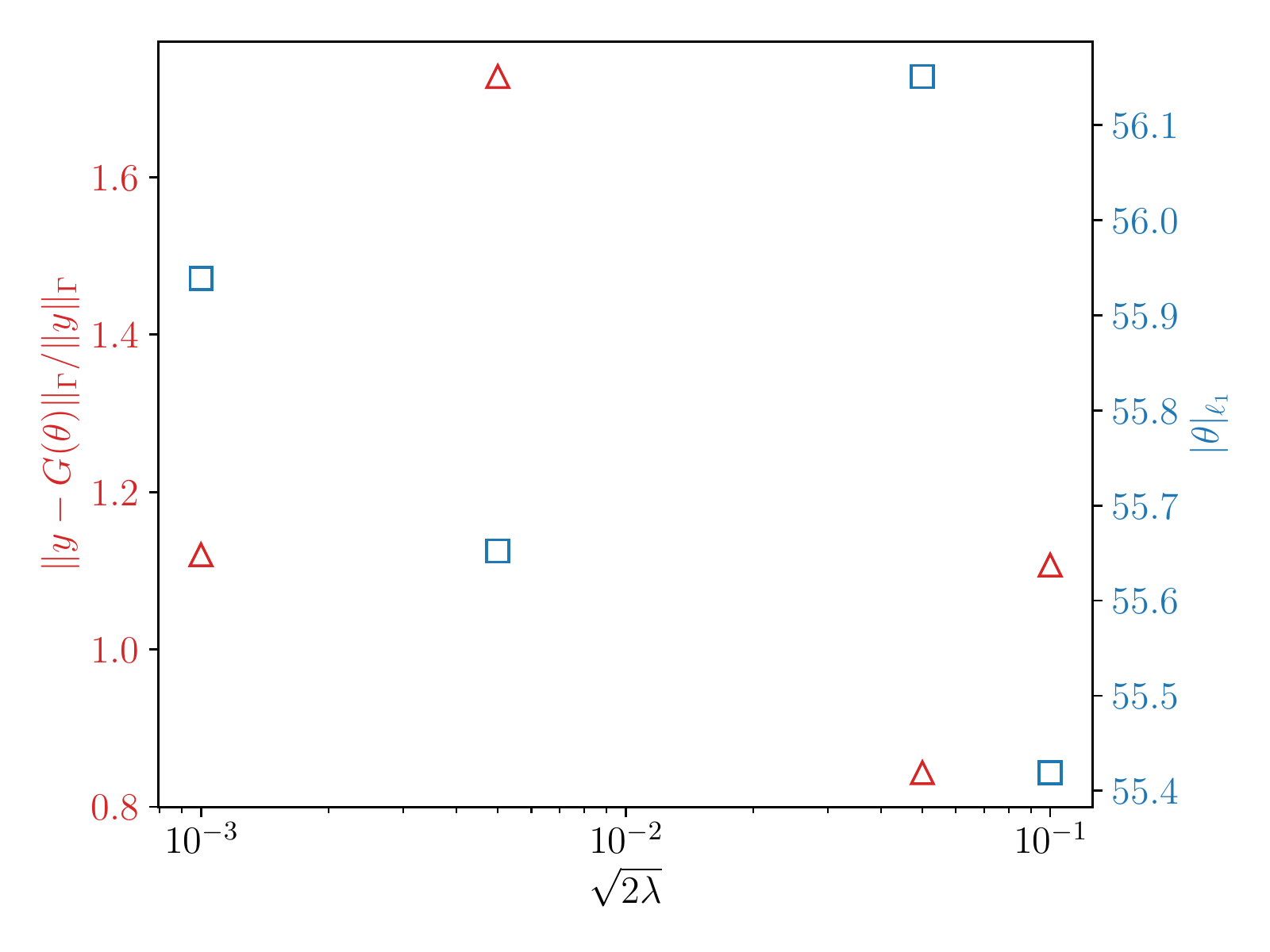}}
    \caption{Parametric study of (a) $\ell_1$ penalty coefficient $\gamma$ and (b) thresholding parameter $\lambda$. The data mismatch (left axis) is denoted by triangles ($\triangle$), and the model complexity (right axis) is denoted by squares ($\square$).}
  \label{fig:hyperparams_nL63}
\end{figure}

\subsection{Lorenz 96 System}
\label{ssec:NL96}
In this subsection, we study two examples of the Lorenz 96 system: (1) a simulation study for which the true and modeled systems are both single-scale Lorenz 96 systems; (2) a more realistic study for which the true system is a multi-scale Lorenz 96 system and the modeled system only resolves the slow variables.

\subsubsection{Simulation Study}
For this simulation study, the data are generated from the single-scale Lorenz 96 system in \eqref{eq:l96ss}. The goal is to fit a model as shown in \eqref{eq:l96ssm} by using sparse EKI. Therefore, we are fitting 180 unknown coefficients in total as denoted by $\{\{\beta_k^{(i)}\}_{i=1}^4,\alpha_k\}_{k=1}^{36}$, using a data vector y of dimension 44 (only observing the finite-time average approximation of first and second moments $\{\{\mathcal{G}_1(X),\mathcal{G}_2(X)\}$ for the first 8 state variables). The duration used for time-averaging is $T=100$. Results are presented in Figs.~\ref{fig:G_sL96} to~\ref{fig:autocorr_sL96}.

The comparison of EKI results with data from the true system is presented in Fig.~\ref{fig:G_sL96}. This shows that the results of both standard EKI and sparse EKI have good agreement with the true system in data space. However, the $\ell_1$-norm of redundant coefficients presented in Fig.~\ref{fig:l1_norm_sL96} indicates that some redundant coefficients are not close to zero for the system identified by the standard EKI, while most redundant coefficients are driven to zero using sparse EKI.

\begin{figure}[!htbp]
  \centering
  \includegraphics[width=0.2\textwidth]{G_legend}
  \subfloat[Standard EKI]{\includegraphics[width=0.49\textwidth]{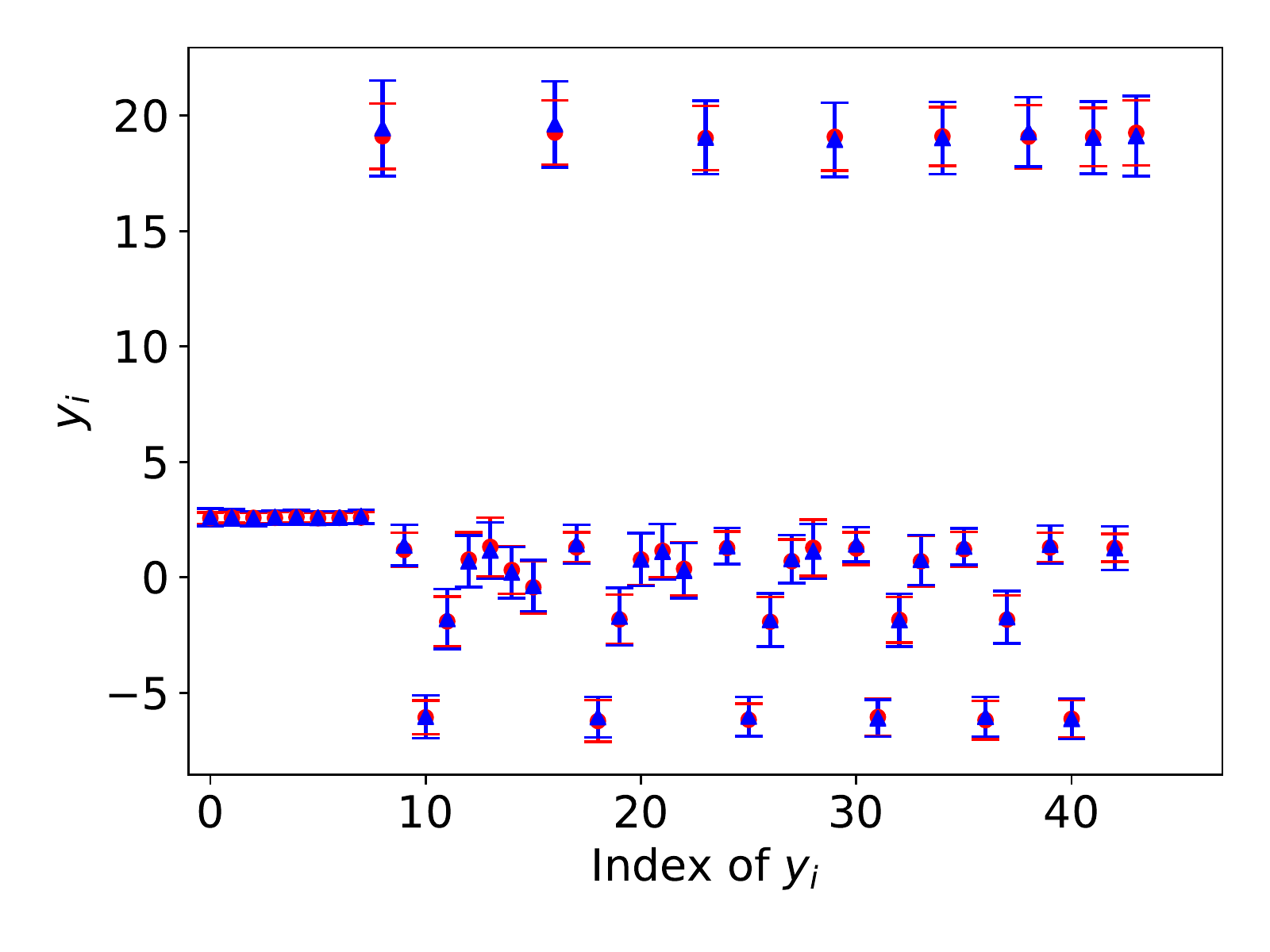}}
  \subfloat[Sparse EKI]{\includegraphics[width=0.49\textwidth]{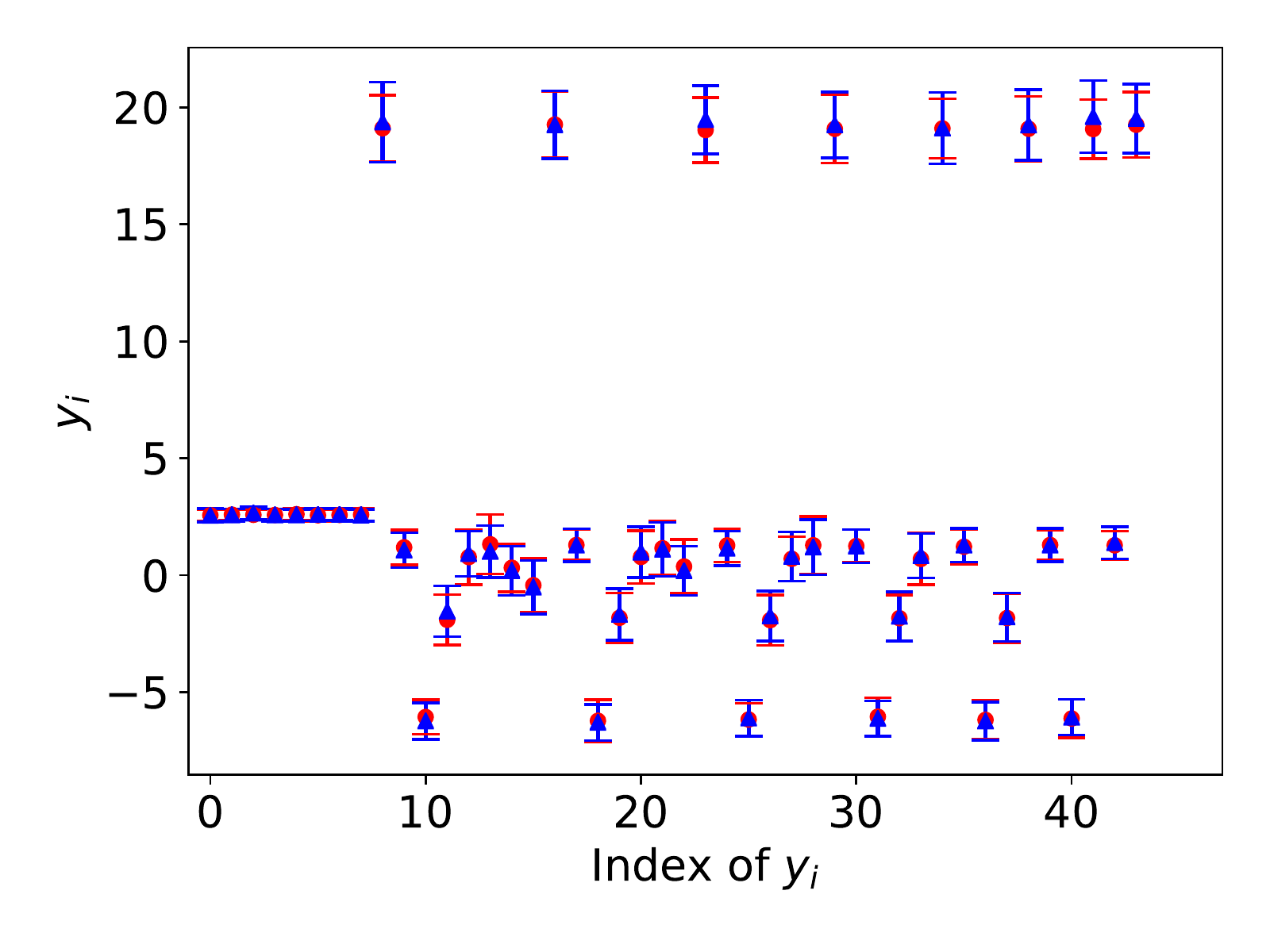}}
    \caption{First two moments of state $X$ for single-scale Lorenz 96 system found by using (a) standard EKI and (b) sparse EKI.}
  \label{fig:G_sL96}
\end{figure}

\begin{figure}[!htbp]
  \centering
  \includegraphics[width=0.55\textwidth]{params_legend_sigma}
  \subfloat[Standard EKI]{\includegraphics[width=0.49\textwidth]{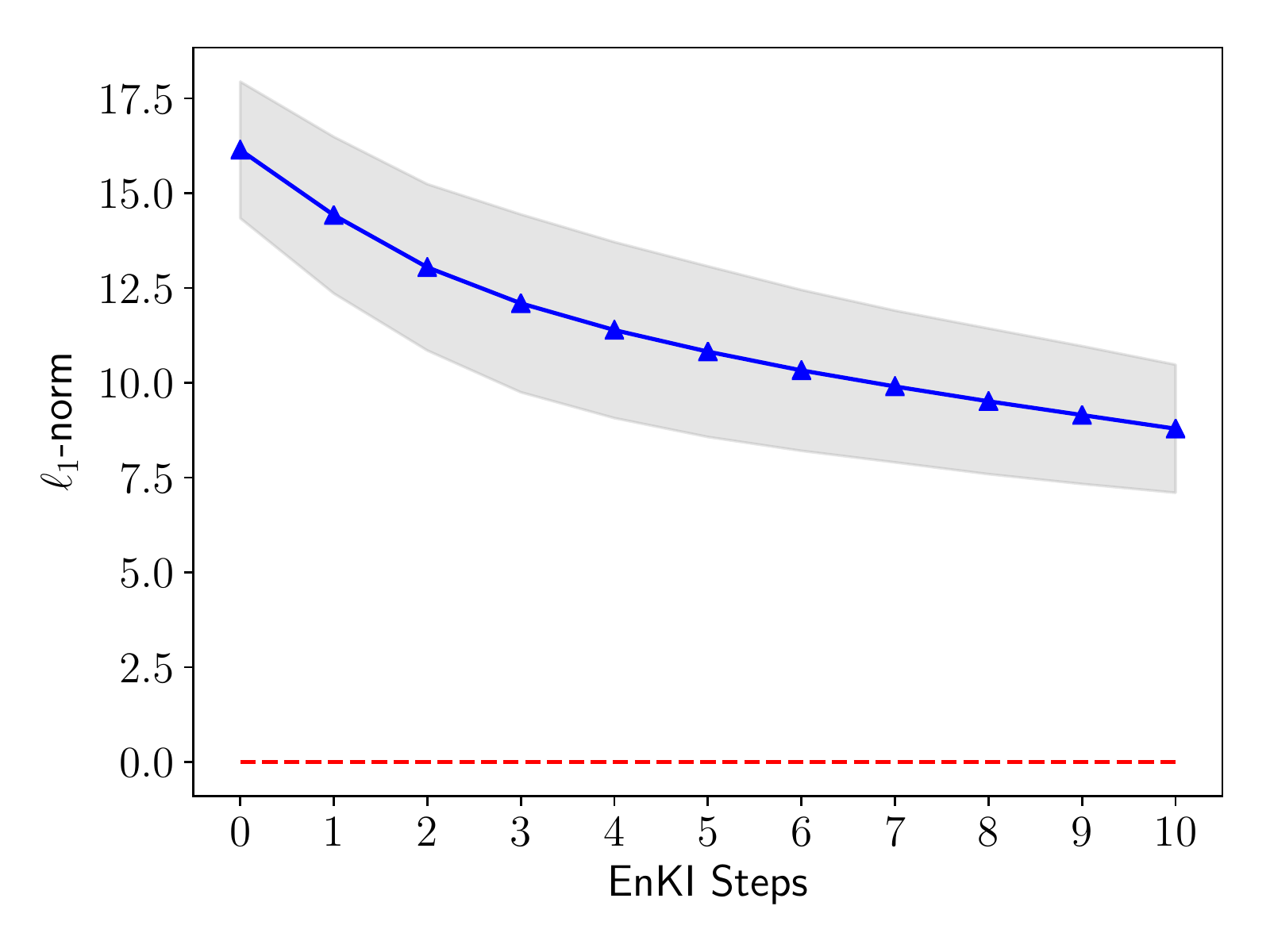}}
  \subfloat[Sparse EKI]{\includegraphics[width=0.49\textwidth]{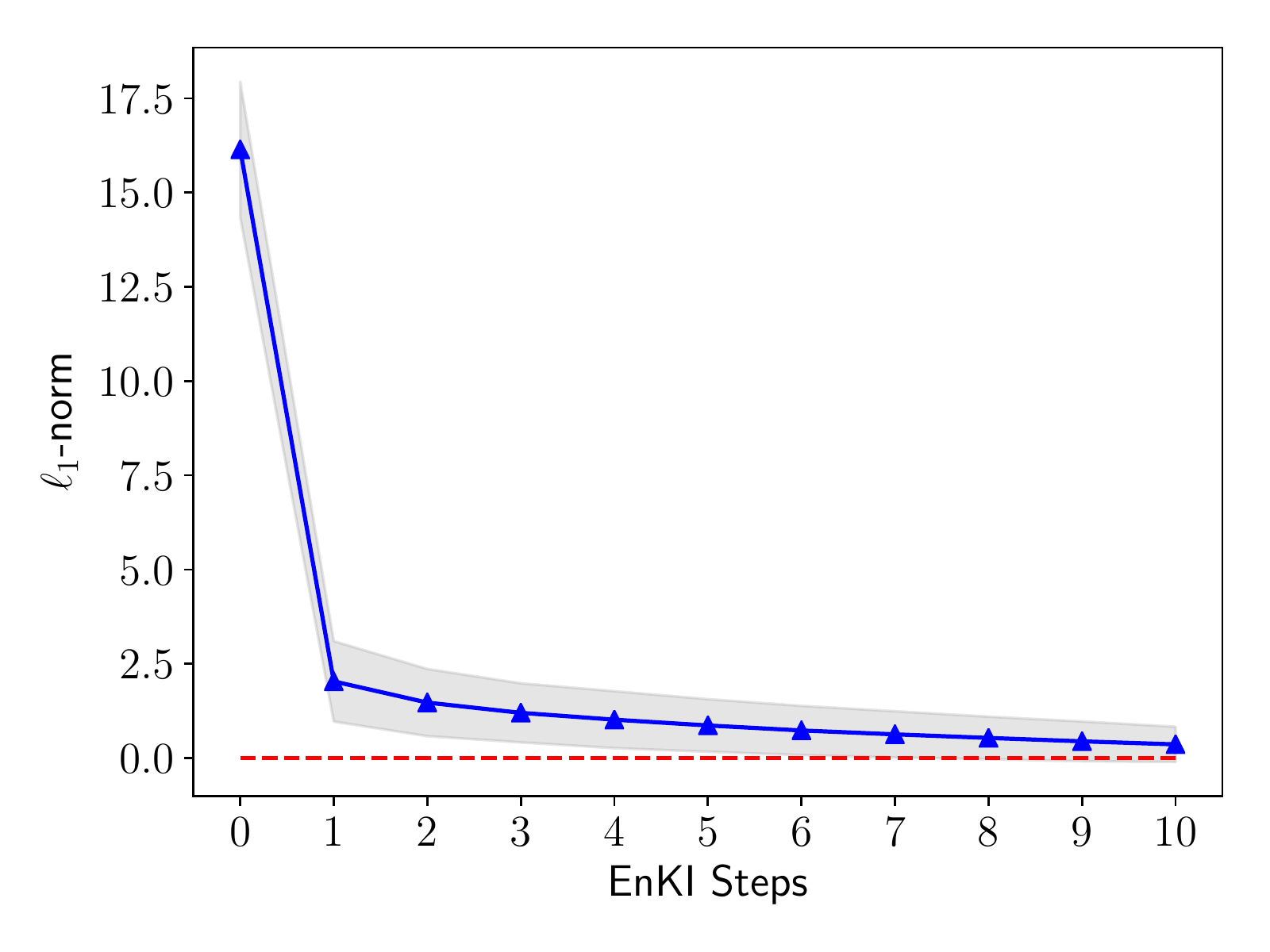}}
    \caption{$\ell_1$-norm of redundant coefficients for single-scale Lorenz 96 system found by using (a) standard EKI and (b) sparse EKI.}
  \label{fig:l1_norm_sL96}
\end{figure}

The long-time limit performance is investigated by evaluating the invariant measure, as presented in Fig.~\ref{fig:measure_sL96}. We can see that both the systems identified by standard EKI and sparse EKI show a good agreement with the invariant measure of the true system, while there is slightly greater uncertainty in the invariant measures of the ensemble simulations from standard EKI. As for the invariant measures, the comparison of autocorrelation functions
presented in Fig.~\ref{fig:autocorr_sL96} shows similar performance of the systems identified by standard EKI and sparse EKI, and both systems have a good agreement with the autocorrelation of the true system.

\begin{figure}[!htbp]
  \centering
  \includegraphics[width=0.2\textwidth]{measure_legend}
  \subfloat[Standard EKI]{\includegraphics[width=0.49\textwidth]{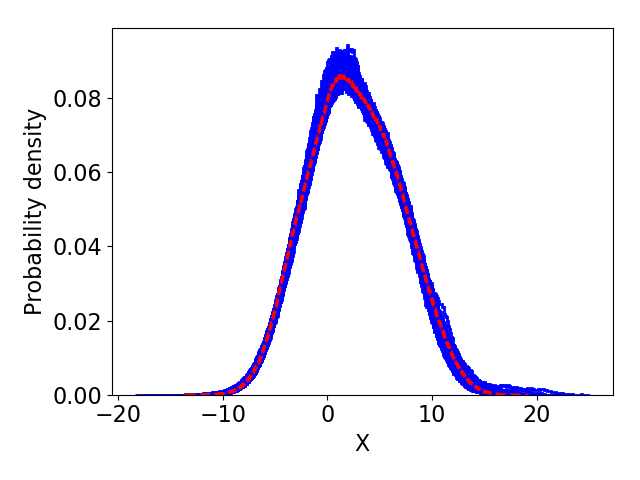}}
  \subfloat[Sparse EKI]{\includegraphics[width=0.49\textwidth]{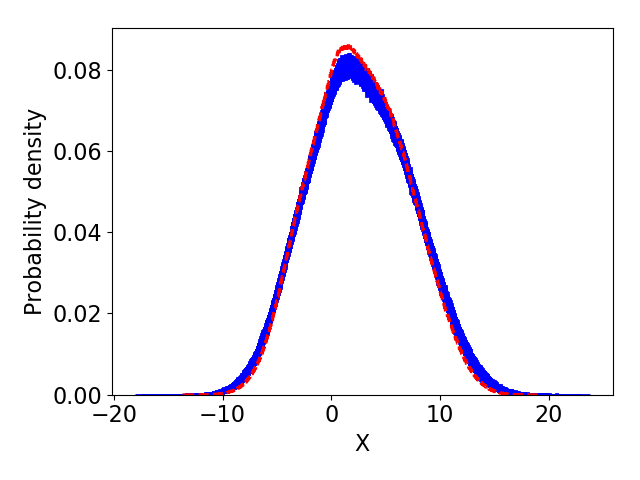}}
    \caption{Invariant measure for single-scale Lorenz 96 system found by using (a) standard EKI and (b) sparse EKI.}
  \label{fig:measure_sL96}
\end{figure}

\begin{figure}[!htbp]
  \centering
  \hspace{1em}\includegraphics[width=0.4\textwidth]{autocorr_legend}
  \includegraphics[width=0.6\textwidth]{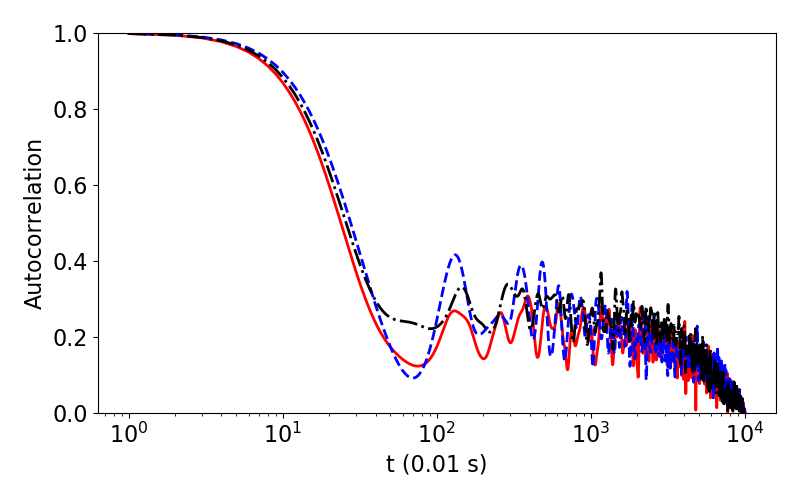}
    \caption{Autocorrelation for single-scale Lorenz 96 system found by using standard EKI and sparse EKI.}
  \label{fig:autocorr_sL96}
\end{figure}

\subsubsection{Multi-scale Data}
We now study a more realistic problem for which the data are generated from the multi-scale Lorenz 96 system in \eqref{eq:l96ms}, and the goal is to fit a reduced-order model as shown in \eqref{eq:l96msm} by using sparse EKI. Therefore, we are fitting 190 unknown coefficients (180 coefficients as denoted by $\{\{\beta_k^{(i)}\}_{i=1}^4,\alpha_k\}_{k=1}^{36}$ and 10 coefficients of GP), using a data vector $y$ of dimension 44 (only observing the finite-time average approximation of first and second moments $\{\{\mathcal{G}_1(X),\mathcal{G}_2(X)\}$ for the first 8 state variables). The time for gathering averaged statistics is $T=100$. Results are presented in Figs.~\ref{fig:G_mL96} to~\ref{fig:autocorr_mL96}.

\begin{figure}[!htbp]
  \centering
  \includegraphics[width=0.2\textwidth]{G_legend}
  \subfloat[Standard EKI]{\includegraphics[width=0.49\textwidth]{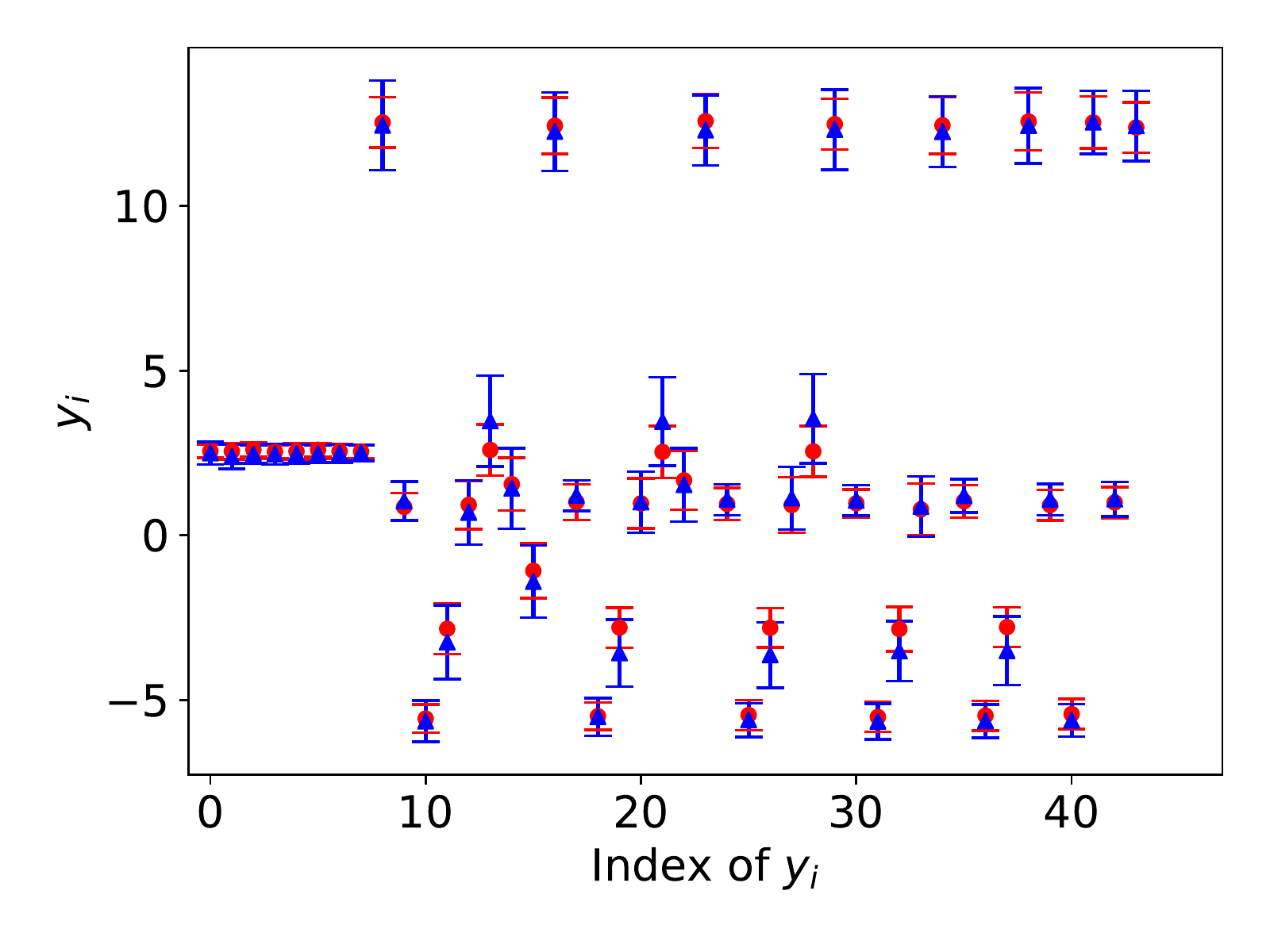}}
  \subfloat[Sparse EKI]{\includegraphics[width=0.49\textwidth]{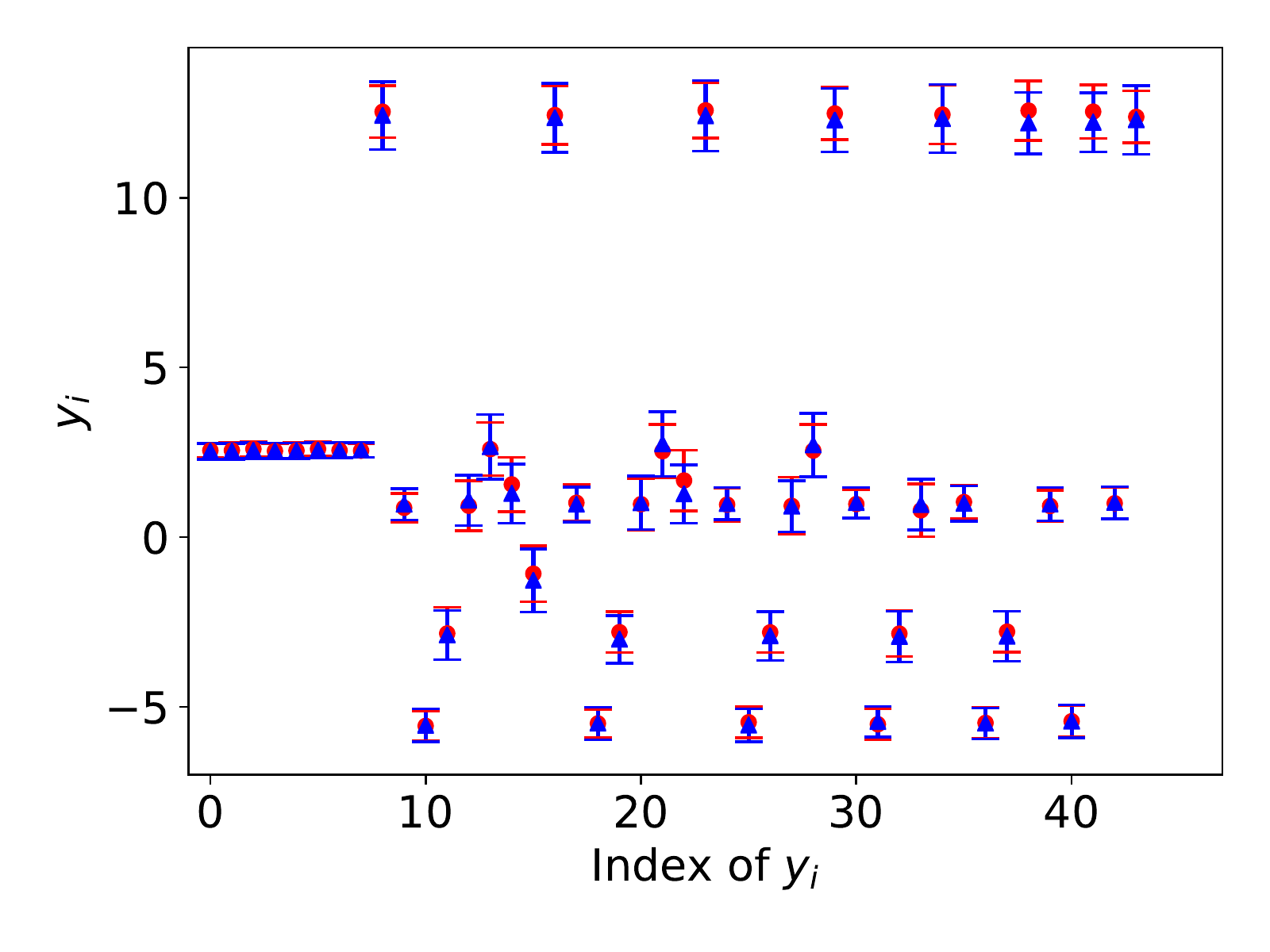}}
    \caption{First two moments of state $X$ for multi-scale Lorenz 96 system found by using (a) standard EKI and (b) sparse EKI.}
  \label{fig:G_mL96}
\end{figure}

We first present the comparison between EKI results and observation data from the true system in Fig.~\ref{fig:G_mL96}. Although the results of standard EKI show relatively good agreement with the true observation data in Fig.~\ref{fig:G_mL96}a, the results of sparse EKI demonstrate a better agreement with true data in Fig.~\ref{fig:G_mL96}b. The better performance of sparse EKI is also confirmed by the comparison of the $\ell_1$-norm of all redundant coefficients as presented in Fig.~\ref{fig:l1_norm_mL96}. The comparison in Fig.~\ref{fig:l1_norm_mL96} indicates that most redundant coefficients are successfully driven to zero using sparse EKI, while there are still some non-zero redundant coefficients in the system identified by standard EKI.

\begin{figure}[!htbp]
  \centering
  \includegraphics[width=0.55\textwidth]{params_legend_sigma}
  \subfloat[Standard EKI]{\includegraphics[width=0.49\textwidth]{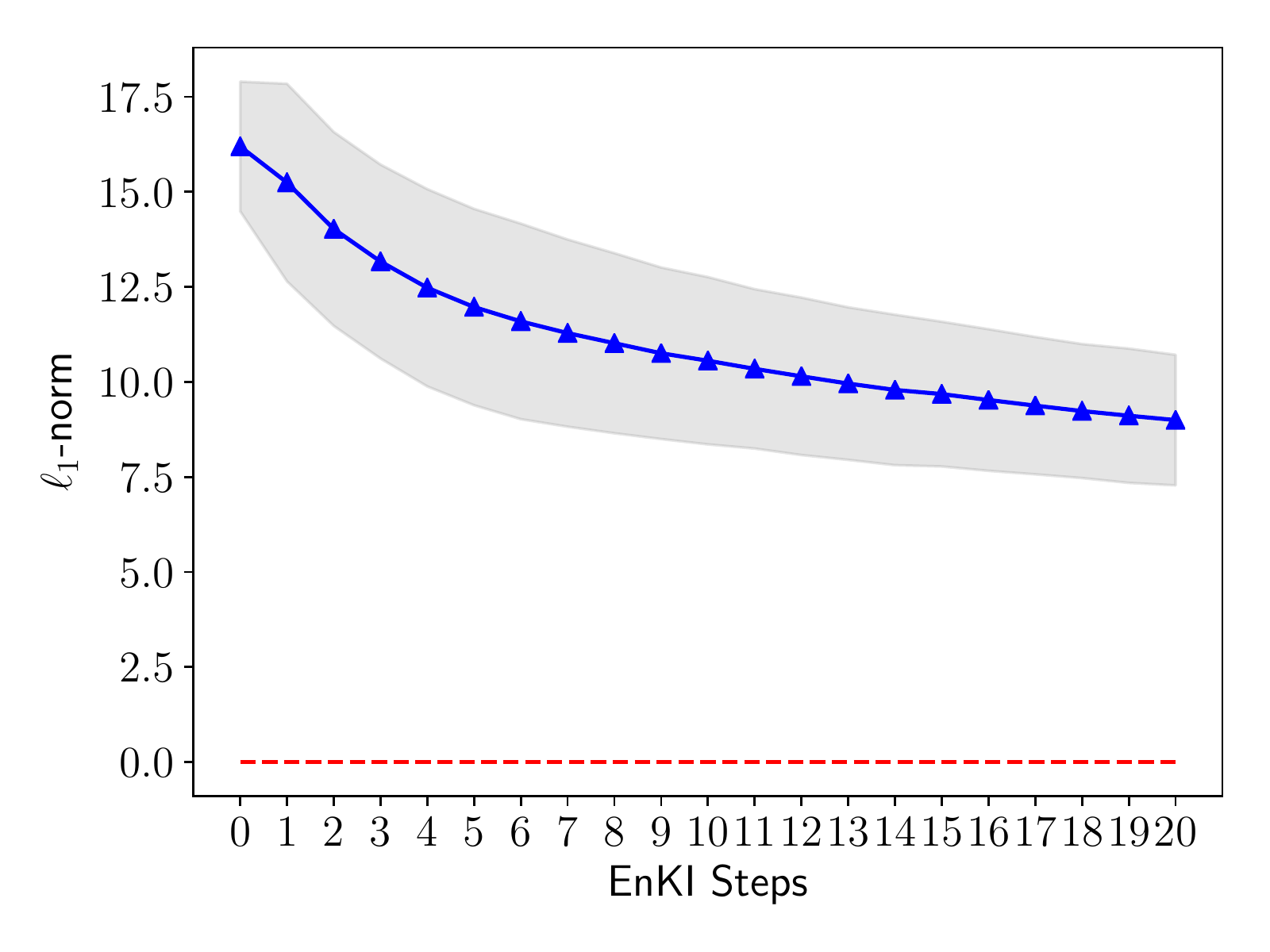}}
  \subfloat[Sparse EKI]{\includegraphics[width=0.49\textwidth]{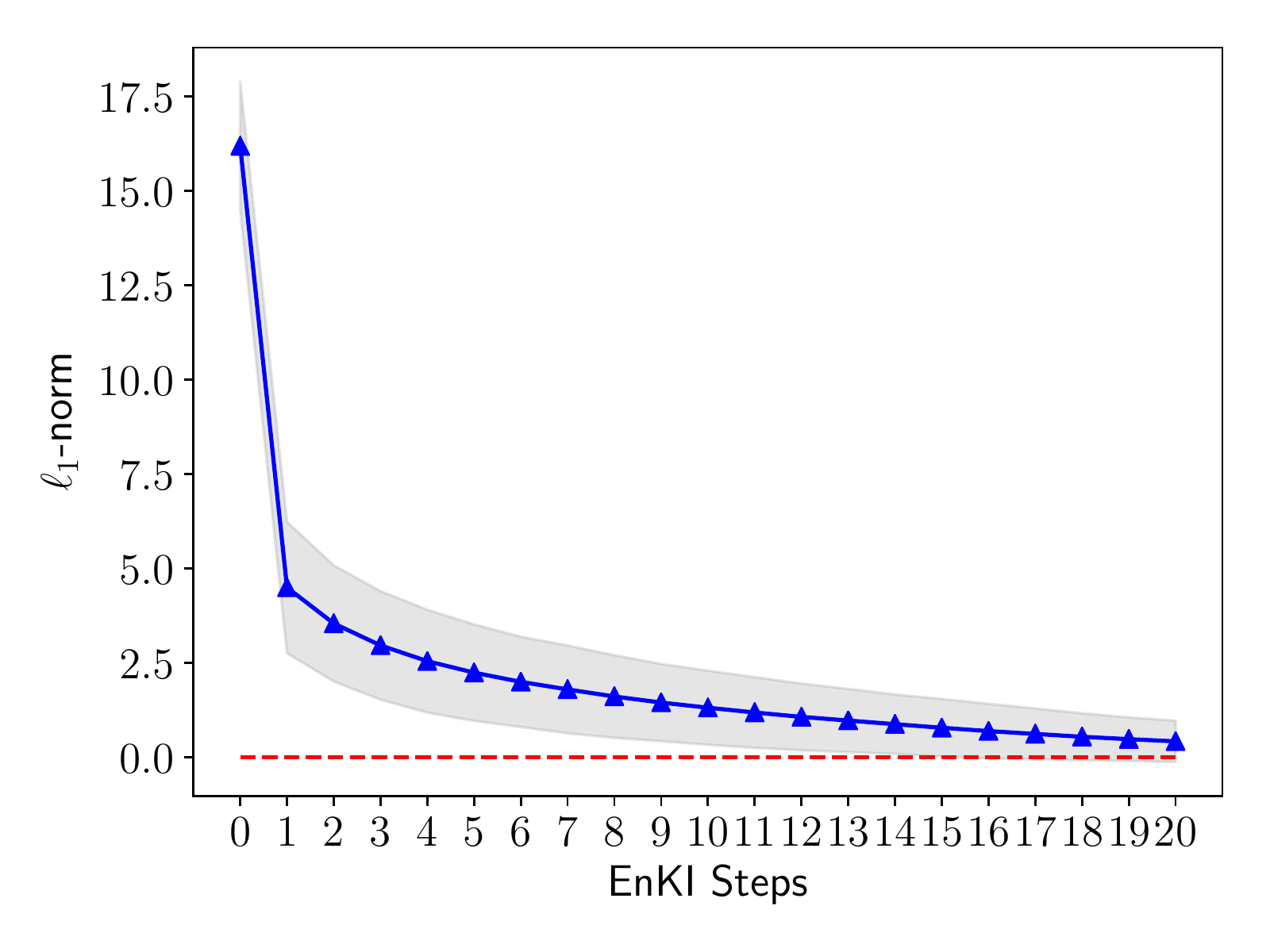}}
    \caption{$\ell_1$-norm of redundant coefficients for multi-scale Lorenz 96 system found by using (a) standard EKI and (b) sparse EKI.}
  \label{fig:l1_norm_mL96}
\end{figure}

The generalization capability of the identified systems is investigated by evaluating the invariant measure. As presented in Fig.~\ref{fig:measure_mL96}, the invariant measure of the system identified by sparse EKI shows a much better agreement with the true system, indicating a better performance in the long-time limit. The comparison of the autocorrelation for a chosen ensemble is also studied in Fig.~\ref{fig:autocorr_mL96}. It demonstrates a better agreement with the true system for the system identified by sparse EKI, in terms of capturing the autocorrelation information.

\begin{figure}[!htbp]
  \centering
  \includegraphics[width=0.2\textwidth]{measure_legend}
  \subfloat[Standard EKI]{\includegraphics[width=0.49\textwidth]{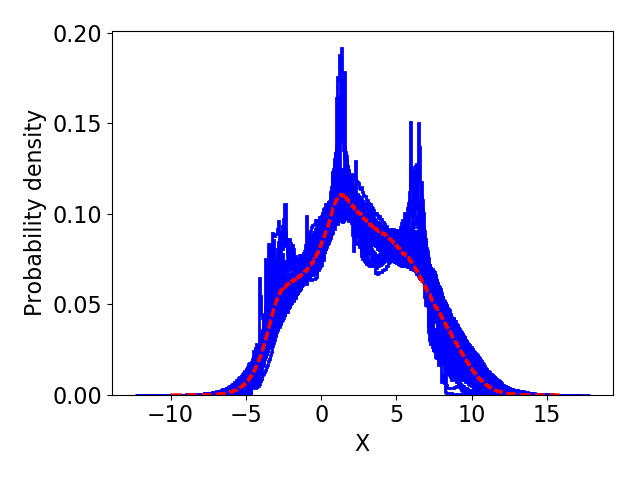}}
  \subfloat[Sparse EKI]{\includegraphics[width=0.49\textwidth]{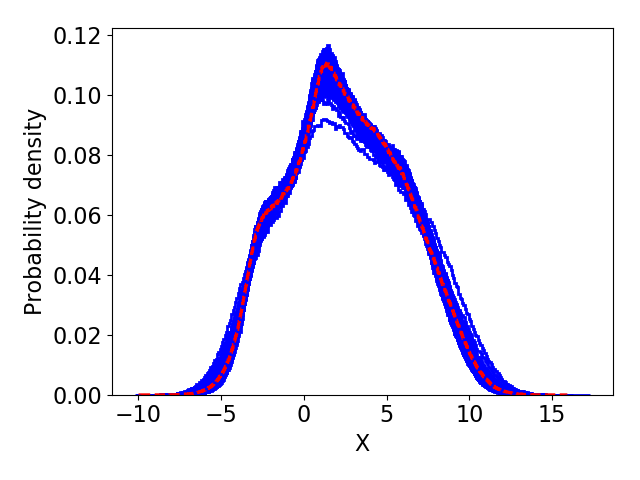}}
    \caption{Invariant measure for multi-scale Lorenz 96 system found by using (a) standard EKI and (b) sparse EKI.}
  \label{fig:measure_mL96}
\end{figure}

\begin{figure}[!htbp]
  \centering
  \hspace{1em}\includegraphics[width=0.4\textwidth]{autocorr_legend}
  \includegraphics[width=0.6\textwidth]{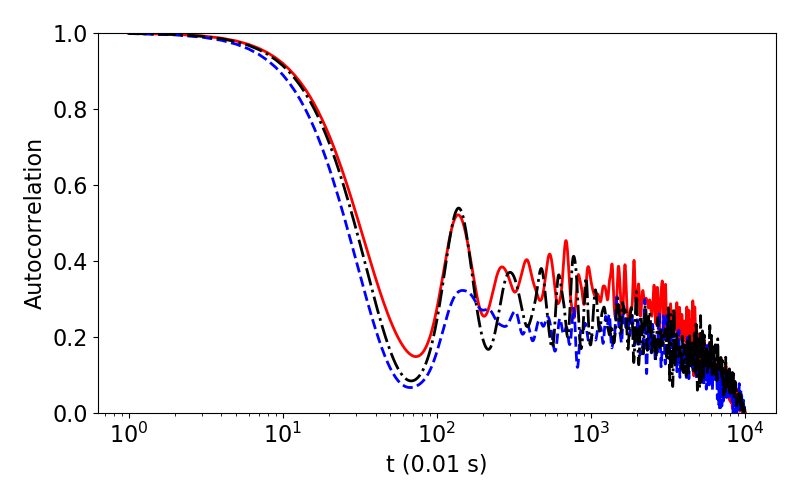}
    \caption{Autocorrelation for multi-scale Lorenz 96 system found by using standard EKI and sparse EKI.}
  \label{fig:autocorr_mL96}
\end{figure}

\subsection{Coalescence Equations}
\label{ssec:NCE}
We further apply the sparse EKI to fit coalescence equations based on statistics derived from time averaging. Specifically, we study three examples: (i) a simulation study where the true system and modeled system share the same closure (Gamma distribution closure) and the same number of resolved states ($K=2$); (ii) an example where true system and modeled system share the same closure (Gamma distribution closure), while the true system resolves more states ($K=3$); (iii) an example where true system and modeled system share the same number of resolved states ($K=2$), while the true system has a different closure (exponential distribution closure). For all three tests of coalescence equations, we impose the symmetry $c_{ab}=c_{ba}$ and thus fit 9 unknown coefficients (recall that we always set $r=3$, and $c_{11}=0$), using a data vector of dimension 5 (observing the finite-time average approximation of first and second moments $\{\{\mathcal{G}_1(X),\mathcal{G}_2(X)\}$). The time used for gathering averaged statistics is $T=50$. Furthermore we impose
positivity on all the learned parameters $c_{ab}$ to ensure
searching in the space of well-posed models.

\subsubsection{Simulation Study}
In this simulation study, the data are generated by simulating the coalescence equations in \eqref{eq:cem} with $K=2$, $r=3$ and the Gamma distribution closure in \eqref{eq:gamma}. The goal is to fit a model with the same $K$, $r$, and closure by using EKI to estimate unknown coefficients $c_{ab}$. Results are presented in Figs.~\ref{fig:G_ce} to~\ref{fig:states_ce}.

\begin{figure}[!htbp]
  \centering
  \includegraphics[width=0.2\textwidth]{G_legend}
  \subfloat[Standard EKI]{\includegraphics[width=0.49\textwidth]{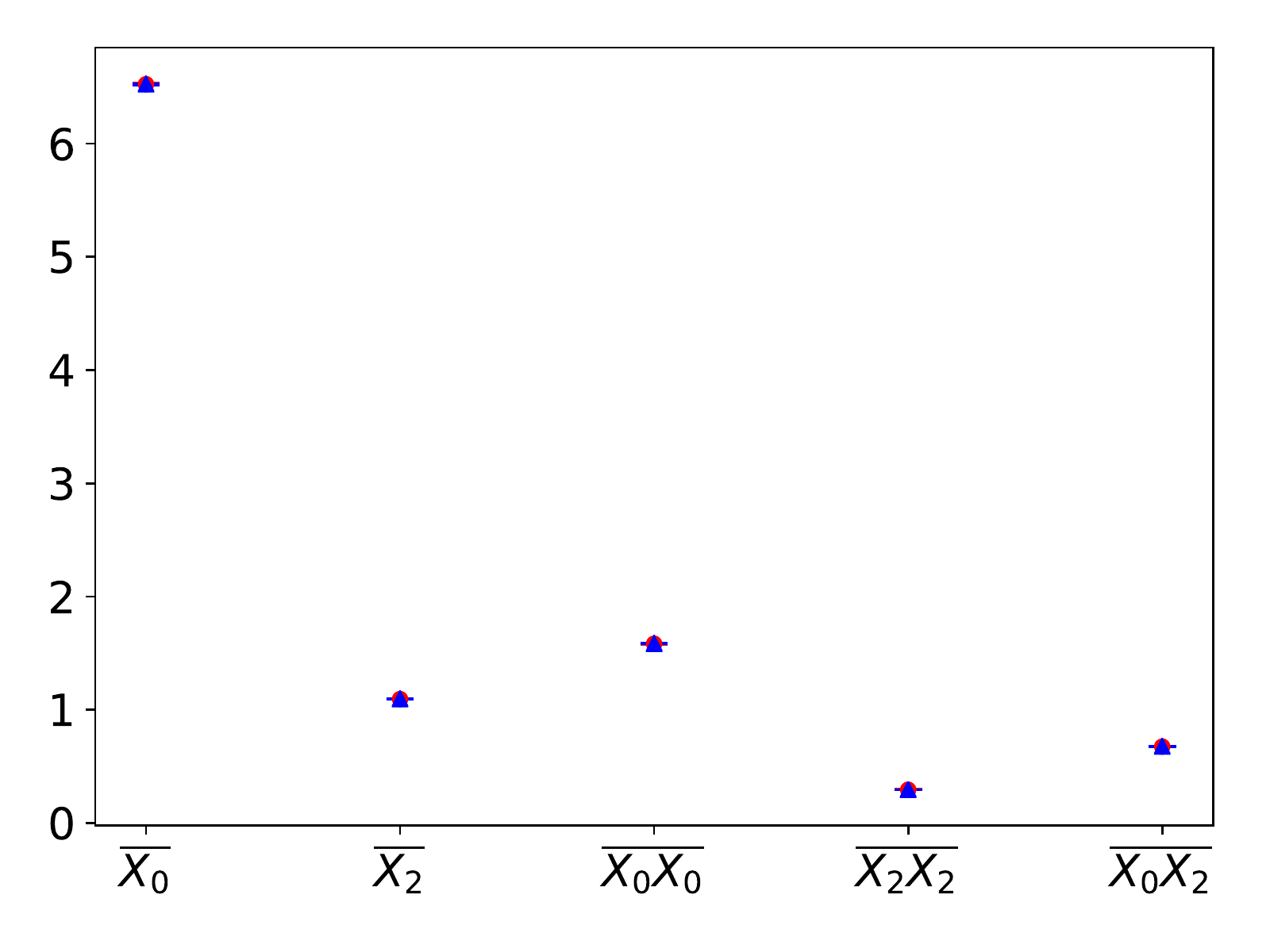}}
  \subfloat[Sparse EKI]{\includegraphics[width=0.49\textwidth]{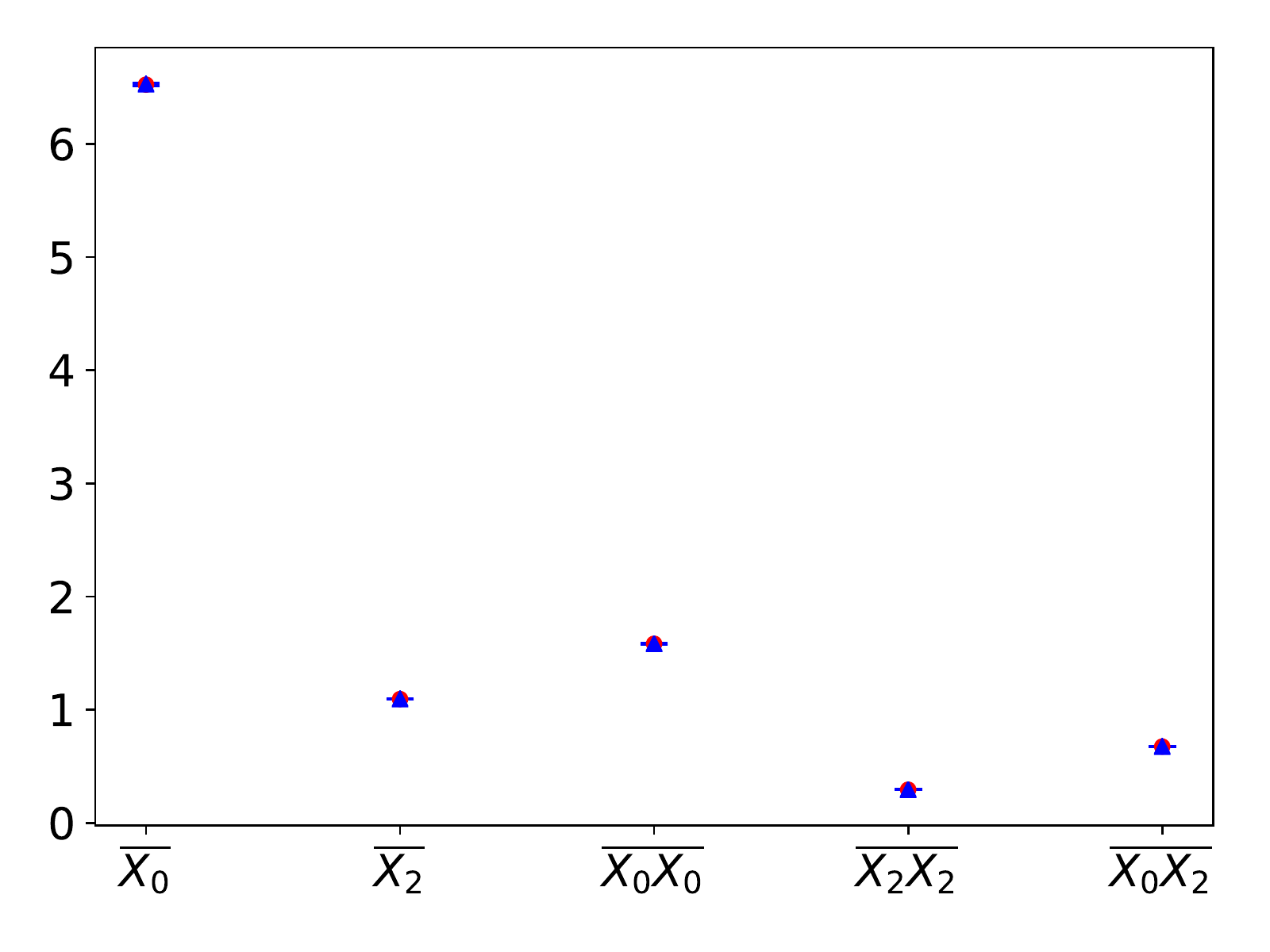}}
    \caption{First two moments of state $X$ for coalescence equations found by using (a) standard EKI and (b) sparse EKI.}
  \label{fig:G_ce}
\end{figure}

The comparison of moments data are presented in Fig.~\ref{fig:G_ce}, which shows that the system identified by using either standard EKI or sparse EKI can have a very good agreement with the true system in terms of matching the first two moments of simulated states. However, it is clear in Fig.~\ref{fig:l1_norm_ce} that the sets of parameters $c_{ab}$ identified by standard EKI and sparse EKI are quite different. The $\ell_1$-norm of redundant coefficients in Fig.~\ref{fig:l1_norm_ce}a indicates that some of the redundant coefficients are still non-zero for the system identified by standard EKI, while all redundant coefficients are driven to zero as presented in Fig.~\ref{fig:l1_norm_ce}b by using sparse EKI.

\begin{figure}[!htbp]
  \centering
  \includegraphics[width=0.55\textwidth]{params_legend_sigma}
  \subfloat[Standard EKI]{\includegraphics[width=0.49\textwidth]{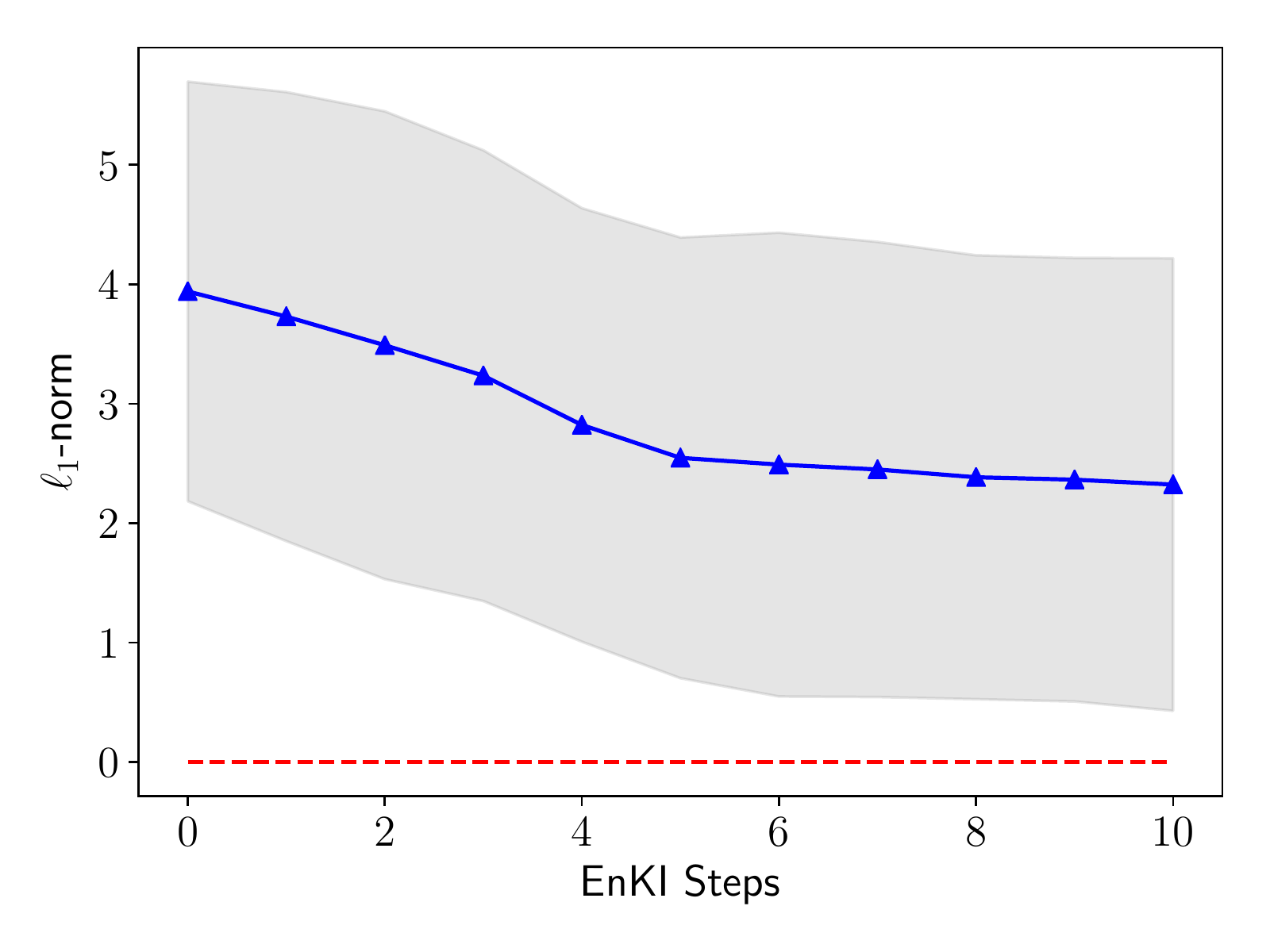}}
  \subfloat[Sparse EKI]{\includegraphics[width=0.49\textwidth]{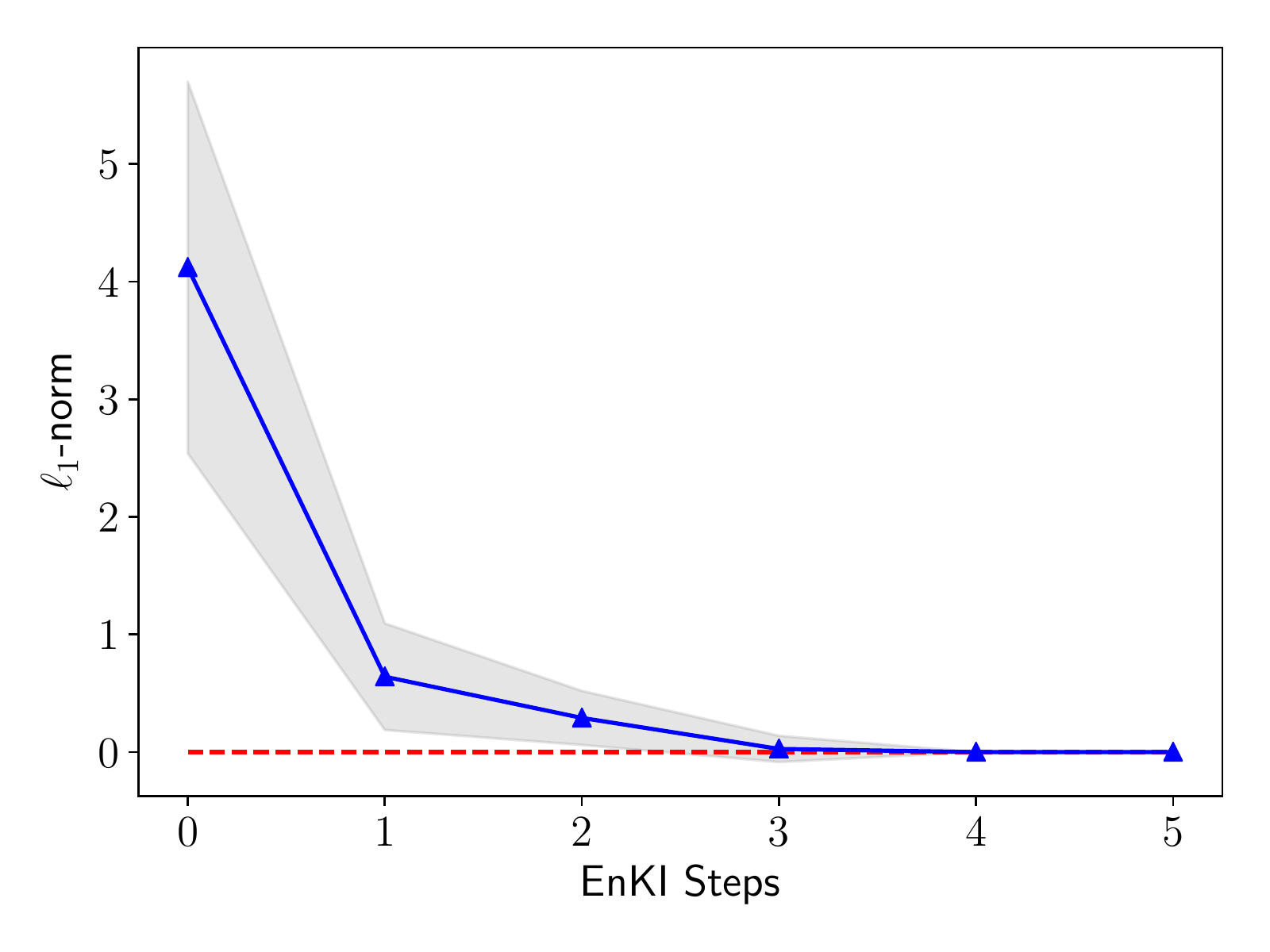}}
    \caption{$\ell_1$-norm of redundant coefficients for coalescence equations found by using (a) standard EKI and (b) sparse EKI.}
  \label{fig:l1_norm_ce}
\end{figure}

The comparison of the non-zero coefficient $c_{00}$ in the true system is presented in Fig.~\ref{fig:param_ce}. The ensemble mean of the estimated parameter matches with its true value using either standard EKI or sparse EKI, while the result of sparse EKI demonstrates better convergence of the ensemble to the true value.

\begin{figure}[!htbp]
  \centering
  \includegraphics[width=0.55\textwidth]{params_legend_sigma}
  \subfloat[Standard EKI]{\includegraphics[width=0.49\textwidth]{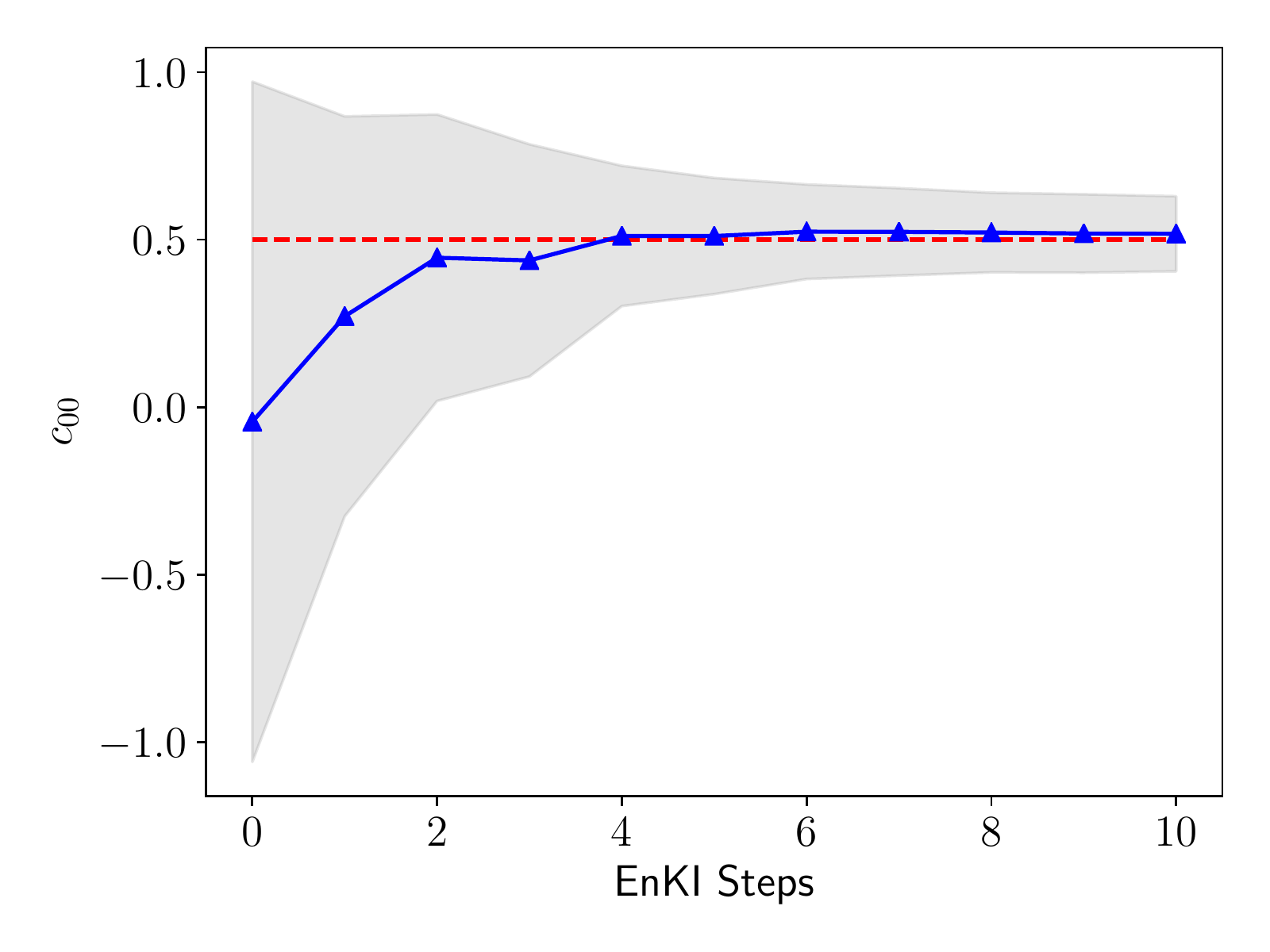}}
  \subfloat[Sparse EKI]{\includegraphics[width=0.49\textwidth]{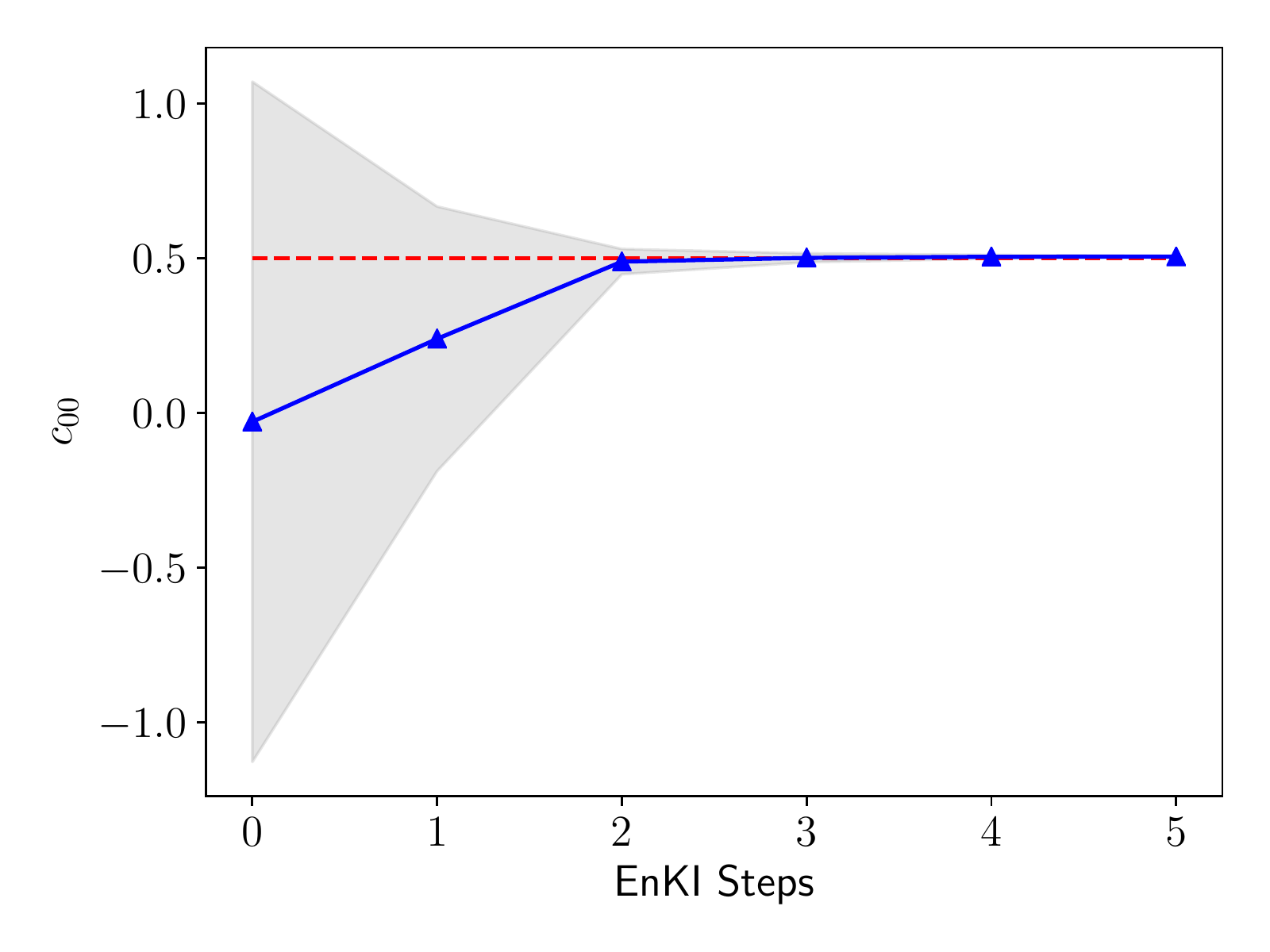}}
    \caption{Non-zero parameter for coalescence equations found by using (a) standard EKI and (b) sparse EKI.}
  \label{fig:param_ce}
\end{figure}

We further investigate the generalization capability of EKI identified systems by comparing the simulated trajectories of states with an initial condition different from the training set. It can be seen in Fig.~\ref{fig:states_ce} that non-zero redundant coefficients and the disagreement of $c_{00}$ among ensemble members do have a negative effect upon the generalization capability of the identified system, as the ensemble of simulated trajectories start to diverge after some time. On the other hand, the system identified by sparse EKI shows a much better agreement with the true trajectories in Fig.~\ref{fig:states_ce}b, even though the initial condition in this test is different from the one used in the training of the sparse model.

\begin{figure}[!htbp]
  \centering
  \includegraphics[width=0.6\textwidth]{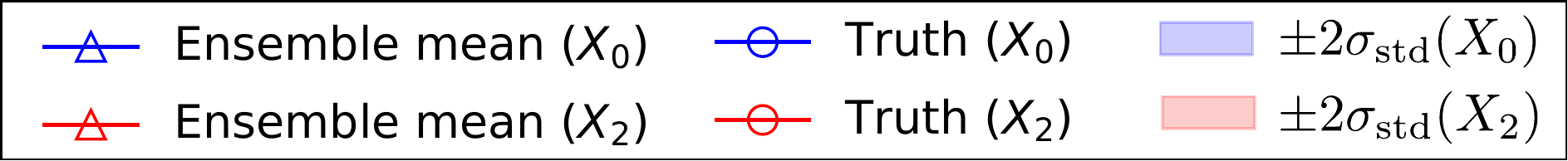}
  \subfloat[Standard EKI]{\includegraphics[width=0.49\textwidth]{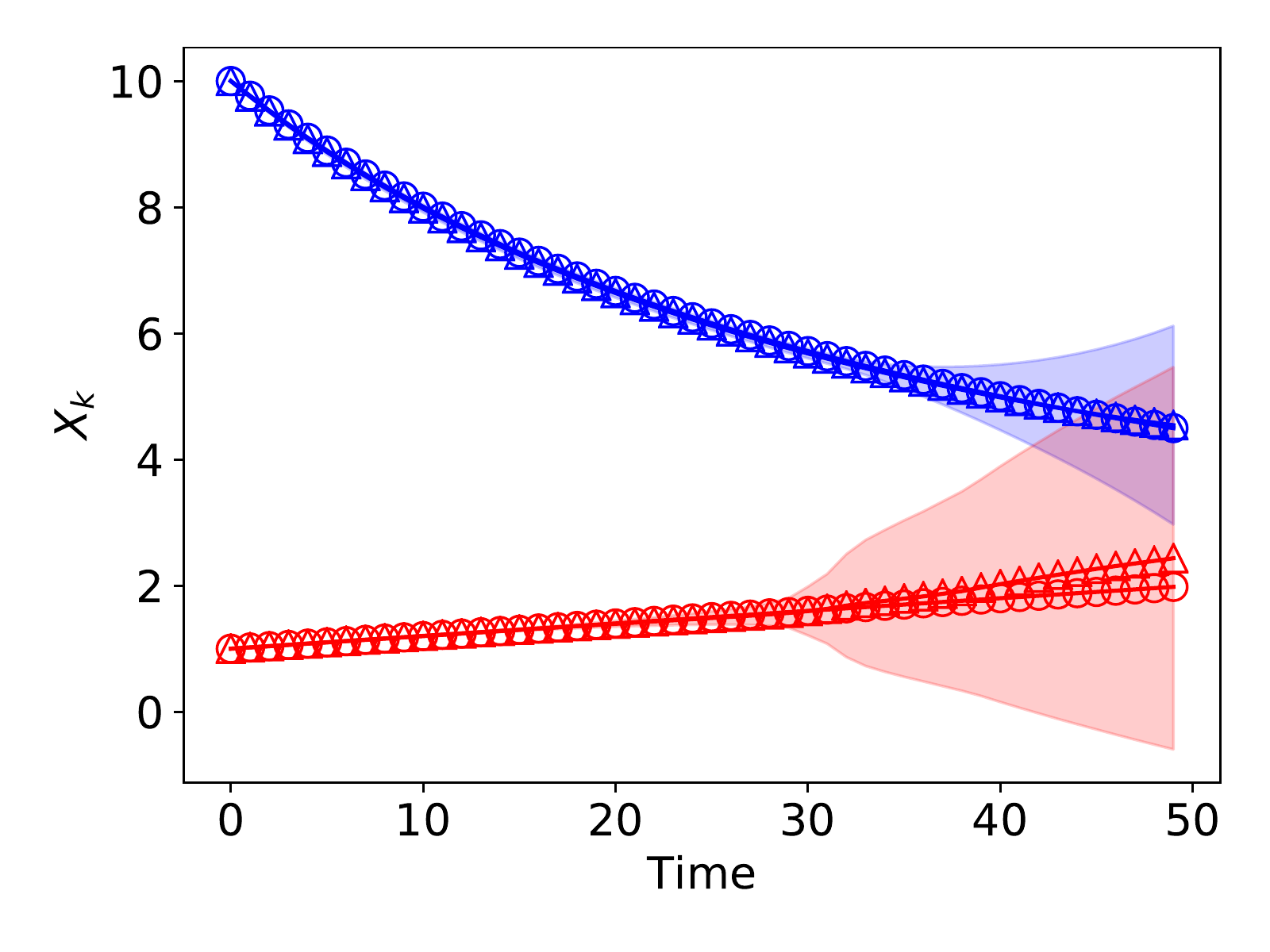}}
  \subfloat[Sparse EKI]{\includegraphics[width=0.49\textwidth]{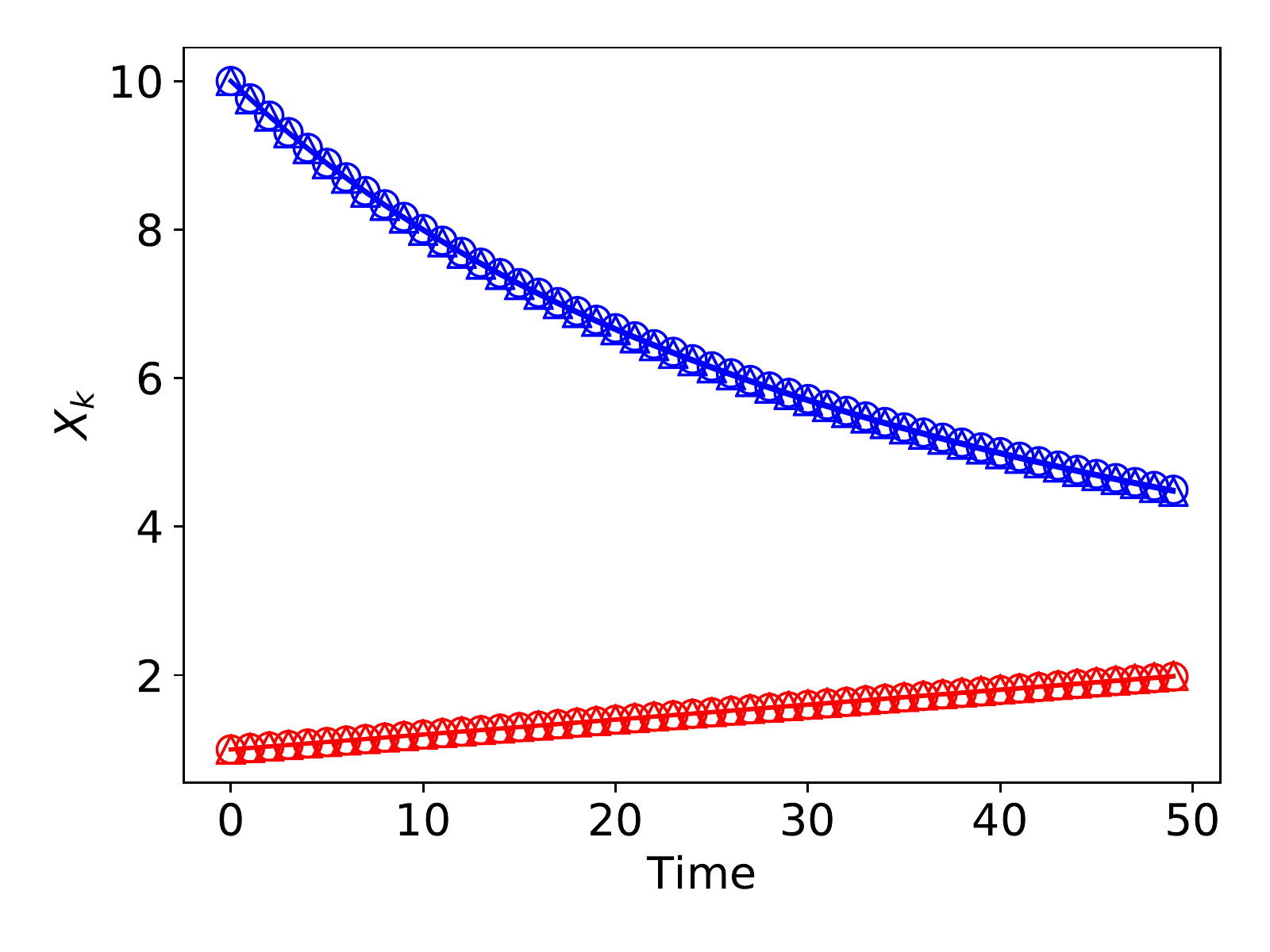}}
    \caption{Simulated states for coalescence equations with coefficients found by using (a) standard EKI and (b) sparse EKI. The initial condition of the training dataset is $(X_0,X_1,X_2)=(10,2,0.6)$, and the initial condition of the simulations here is $(X_0,X_1,X_2)=(10,2,1)$.}
  \label{fig:states_ce}
\end{figure}

\subsubsection{Higher-Order Closure Data}
We now perform a more realistic study, in which the data are generated by simulating the coalescence equations in \eqref{eq:cem} with $K=3$, $r=3$ and the Gamma distribution closure in \eqref{eq:gamma}. The goal is to fit a model with the same $r$ and closure but  different $K$, namely $K=2$, using EKI to estimate unknown coefficients $c_{ab}$. Results are presented in Figs.~\ref{fig:G_ce_order_3} to~\ref{fig:states_ce_order_3_multi_training}.

\begin{figure}[!htbp]
  \centering
  \includegraphics[width=0.2\textwidth]{G_legend}
  \subfloat[Standard EKI]
  {\includegraphics[width=0.49\textwidth]{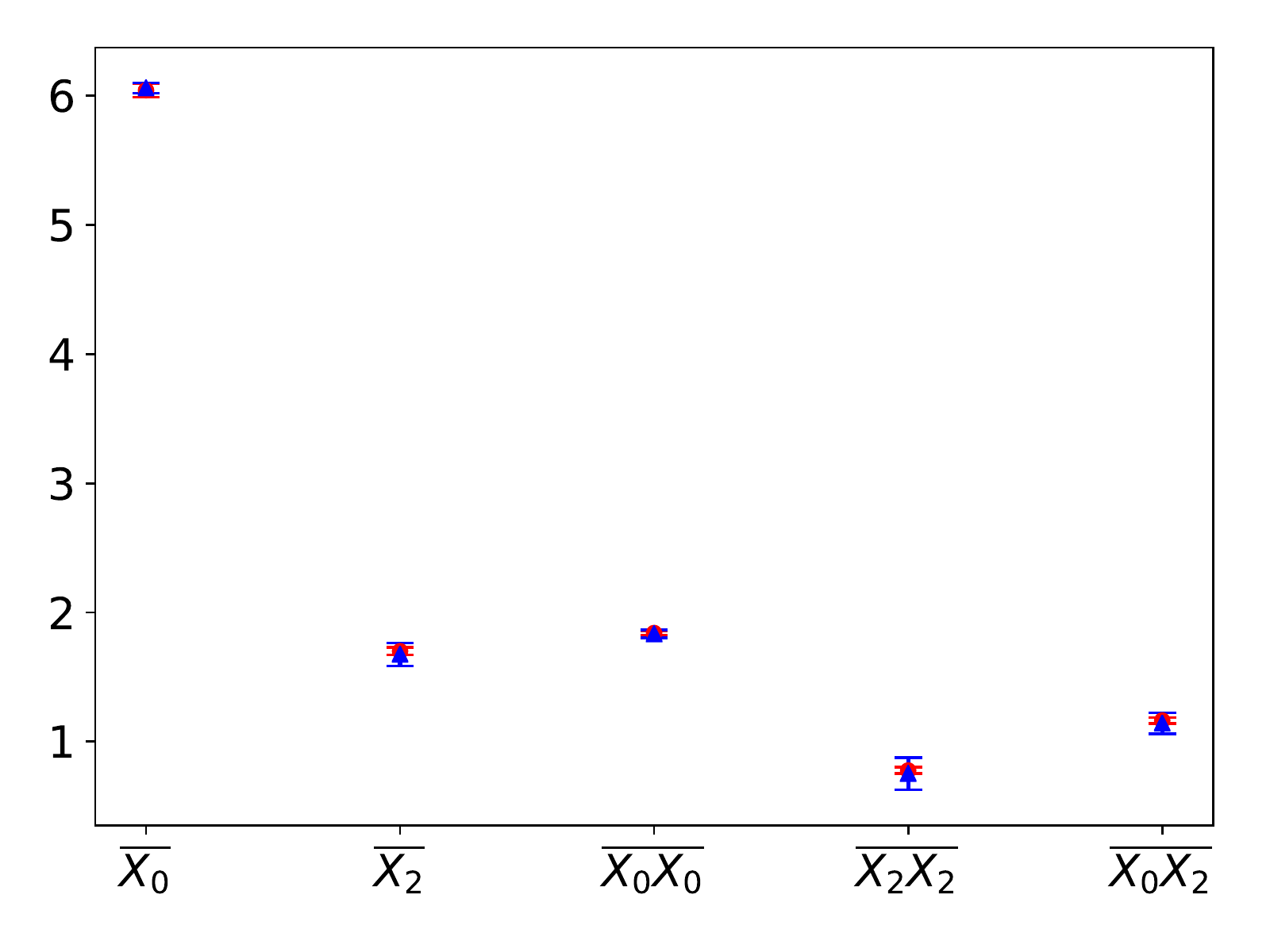}}  \subfloat[Sparse EKI]{\includegraphics[width=0.49\textwidth]{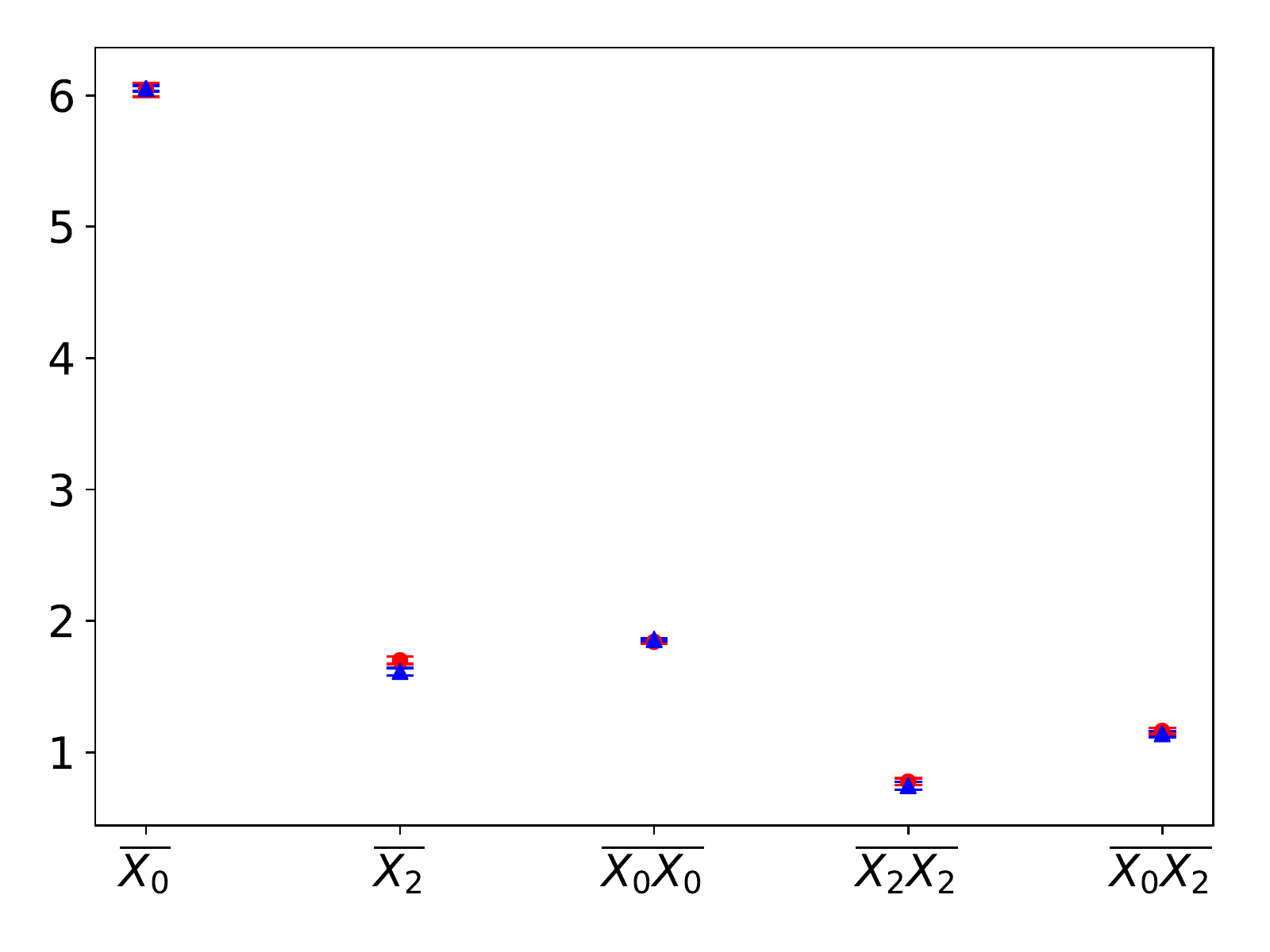}}
    \caption{First two moments of state $X$ for coalescence equations found by using (a) standard EKI and (b) sparse EKI. The data are generated by coalescence equations with a higher-order closure ($K=3$).}
  \label{fig:G_ce_order_3}
\end{figure}

The comparison of data in Fig.~\ref{fig:G_ce_order_3} shows a comparable performance of the identified system by using either standard EKI or sparse EKI. However, the $\ell_1$-norm of all coefficients is significantly different (as presented in (Fig.~\ref{fig:l1_norm_ce_order_3}). It shows that sparse EKI leads to a set of parameters with smaller $\ell_1$-norm. However, we cannot directly tell whether such a set of parameters is better, since the identified system has a closure distribution different from the one of the true system.

\begin{figure}[!htbp]
  \centering
  \includegraphics[width=0.4\textwidth]{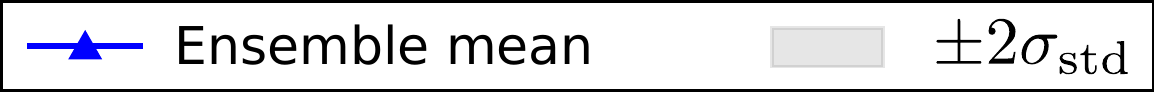}
  \subfloat[Standard EKI]{\includegraphics[width=0.49\textwidth]{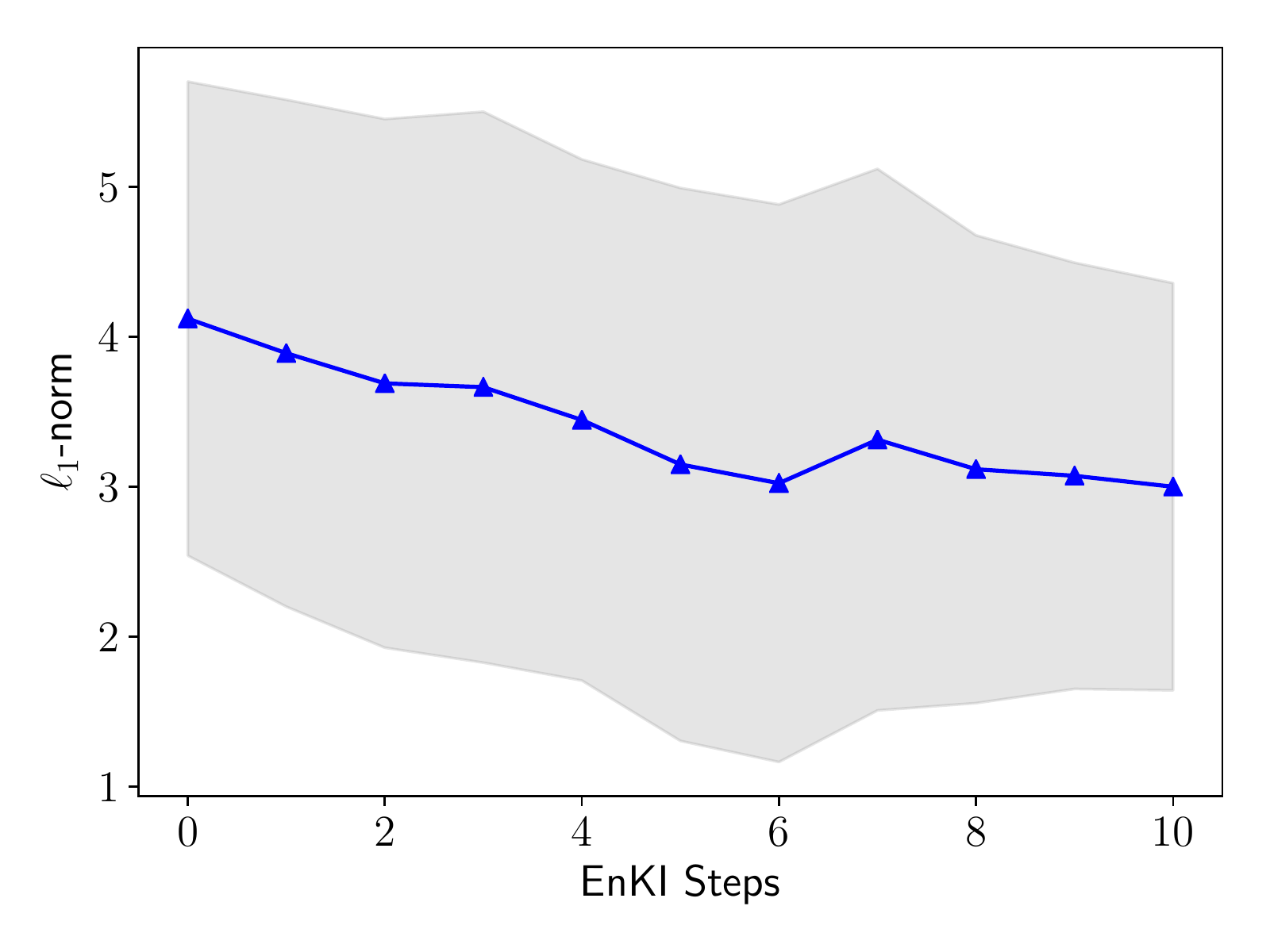}}
  \subfloat[Sparse EKI]{\includegraphics[width=0.49\textwidth]{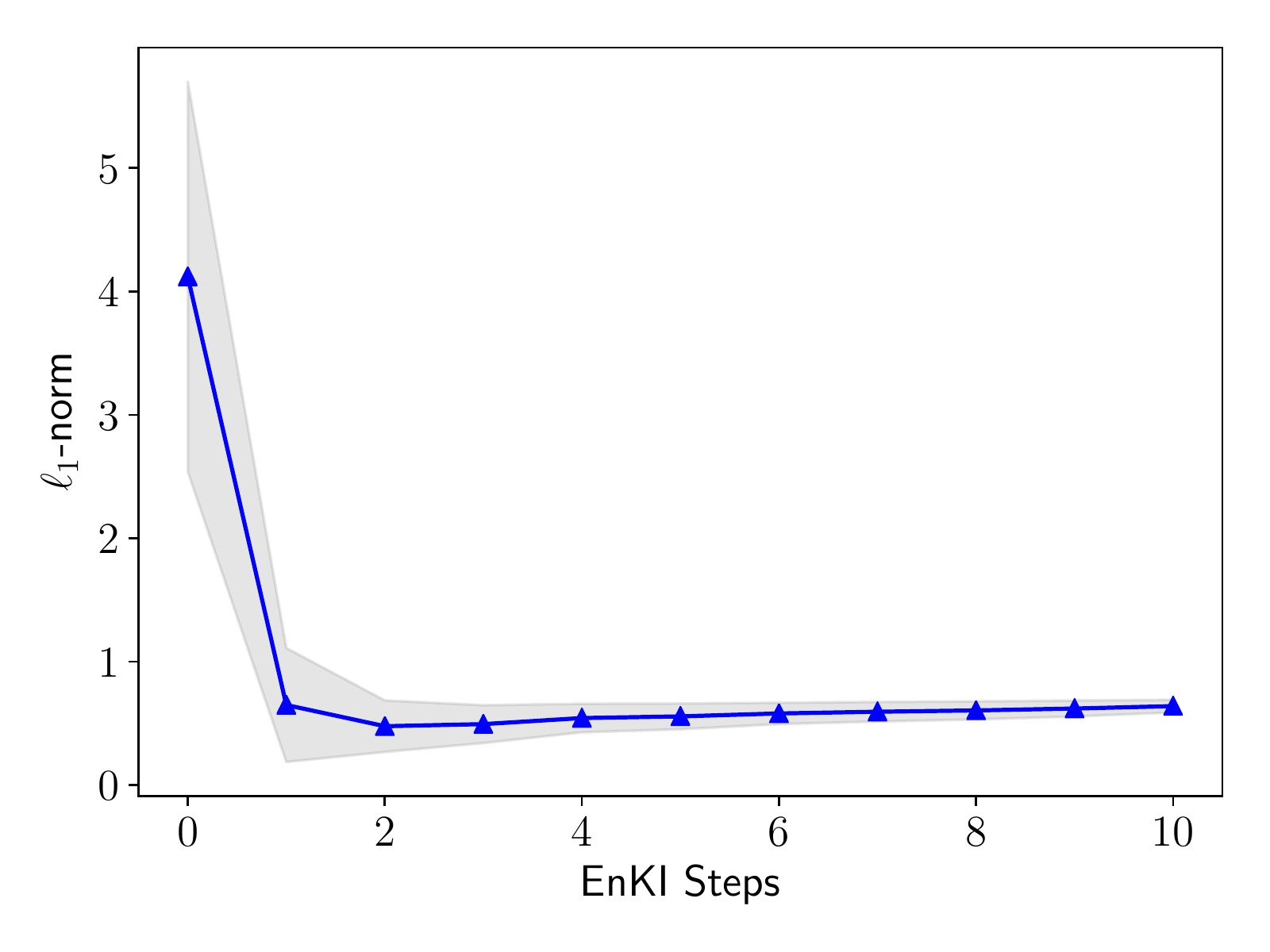}}
    \caption{$\ell_1$-norm of all coefficients for coalescence equations found by using (a) standard EKI and (b) sparse EKI. The data are generated by coalescence equations with a higher order closure ($K=3$).}
  \label{fig:l1_norm_ce_order_3}
\end{figure}

Therefore, we investigate the performances of the EKI identified systems by studying the generalization capability in Fig.~\ref{fig:states_ce_order_3}, i.e., simulating identified systems with an initial condition different from the training data. It is clear in Fig.~\ref{fig:states_ce_order_3} that the simulated trajectories of the system identified by sparse EKI generally have better agreement with the trajectories of the true system.

\begin{figure}[!htbp]
  \centering
  \includegraphics[width=0.6\textwidth]{state_ce_legend}
  \subfloat[Standard EKI]{\includegraphics[width=0.49\textwidth]{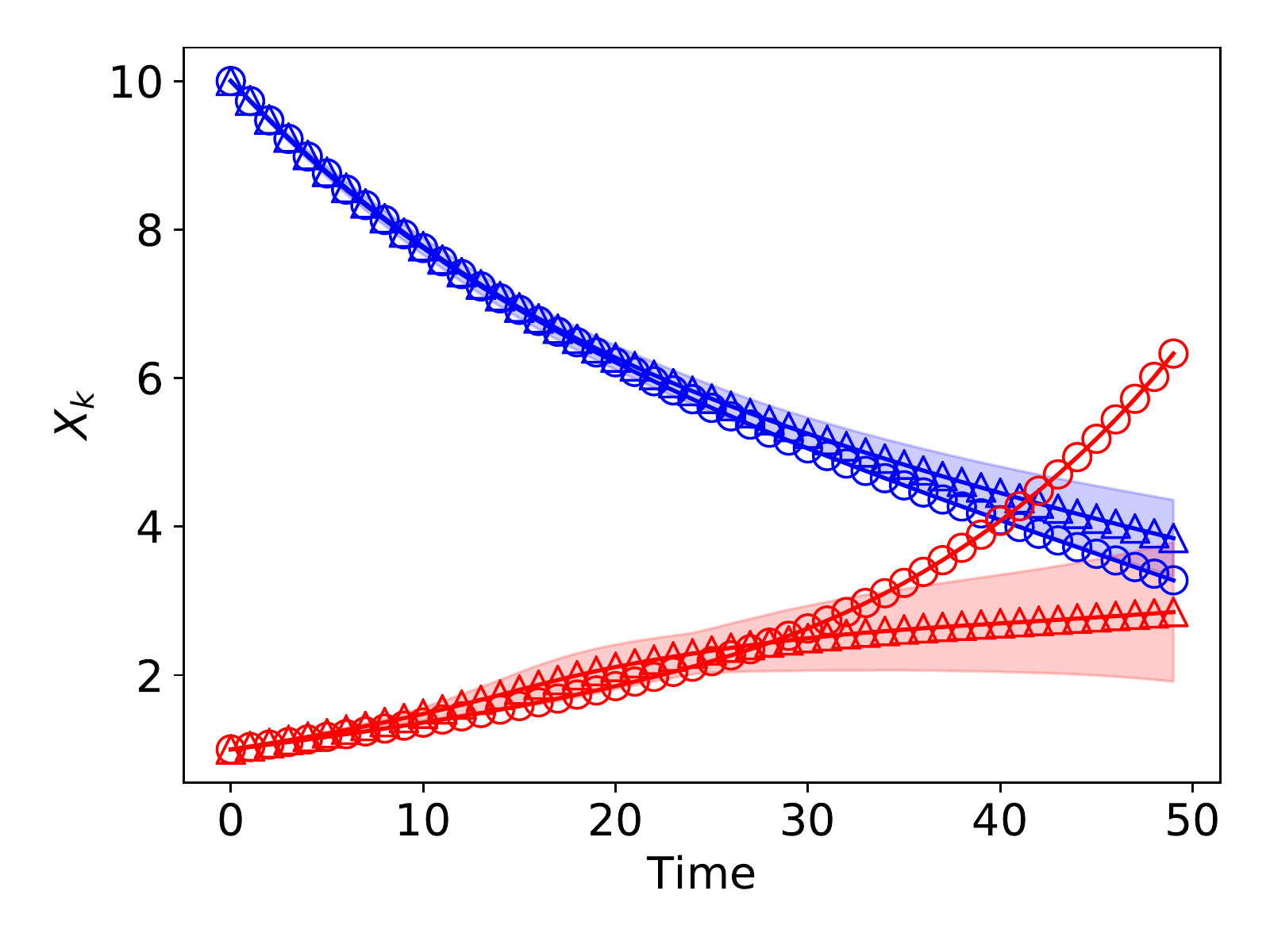}}
  \subfloat[Sparse EKI]{\includegraphics[width=0.49\textwidth]{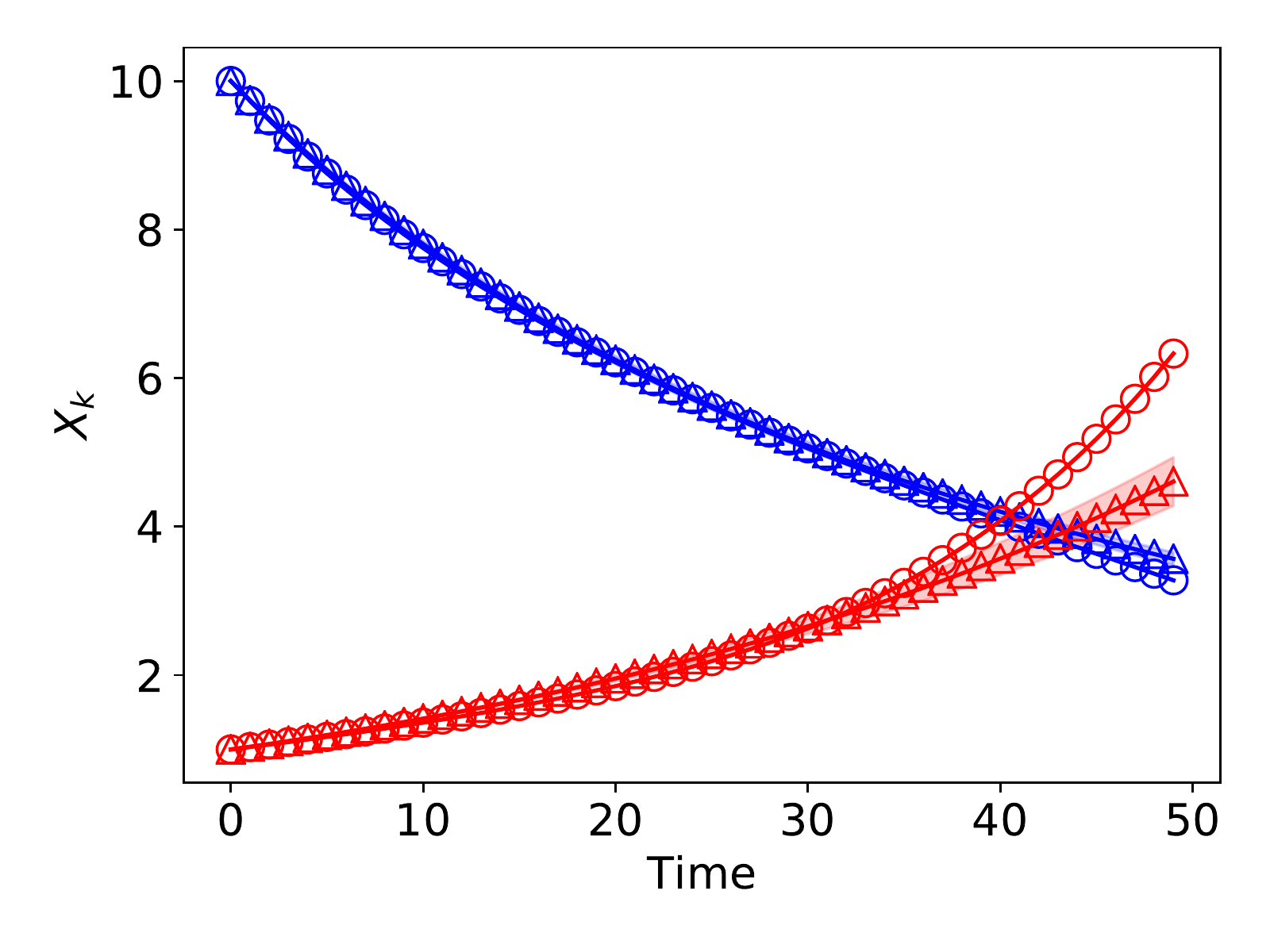}}
    \caption{Simulated states for coalescence equations with coefficients found by using (a) standard EKI and (b) sparse EKI. The data are generated by coalescence equations with a higher-order closure ($K=3$). The initial condition of the training dataset is $(X_0,X_1,X_2)=(10,2,0.6)$, and the initial condition of the simulations here is $(X_0,X_1,X_2)=(10,2,1)$.}
  \label{fig:states_ce_order_3}
\end{figure}

We further try to improve the performance of the identified system by using training sets with different initial conditions, noting that the coalescence equations are not ergodic, and the initial condition does have an effect on the system prediction. Specifically, we use two training sets with initial conditions $(X_0,X_1,X_2)=(10,2,0.6)$ and $(X_0,X_1,X_2)=(10,2,2)$, and we then test the performance of the identified systems with a different initial condition $(X_0,X_1,X_2)=(10,2,1)$. The results in Fig.~\ref{fig:states_ce_order_3_multi_training} show that the identified systems with multiple training sets provide better agreement of simulated trajectories with the true system, and the improvement of performance is more significant for the system identified by standard EKI. This is not surprising since we still fit 9 unknown coefficients here but with twice the data (10 elements in total).

\begin{figure}[!htbp]
  \centering
  \includegraphics[width=0.6\textwidth]{state_ce_legend}
  \subfloat[Standard EKI]{\includegraphics[width=0.49\textwidth]{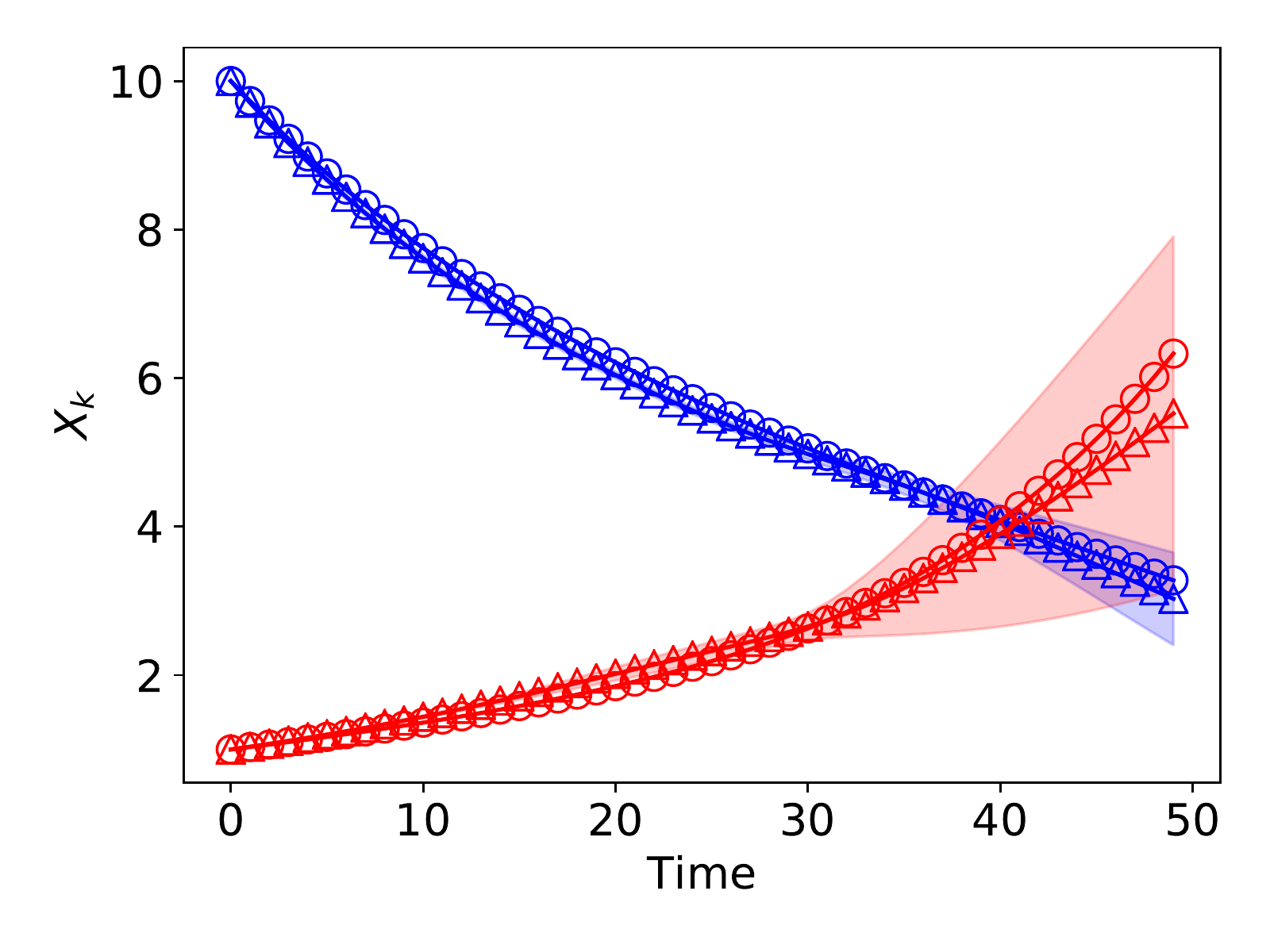}}
  \subfloat[Sparse EKI]{\includegraphics[width=0.49\textwidth]{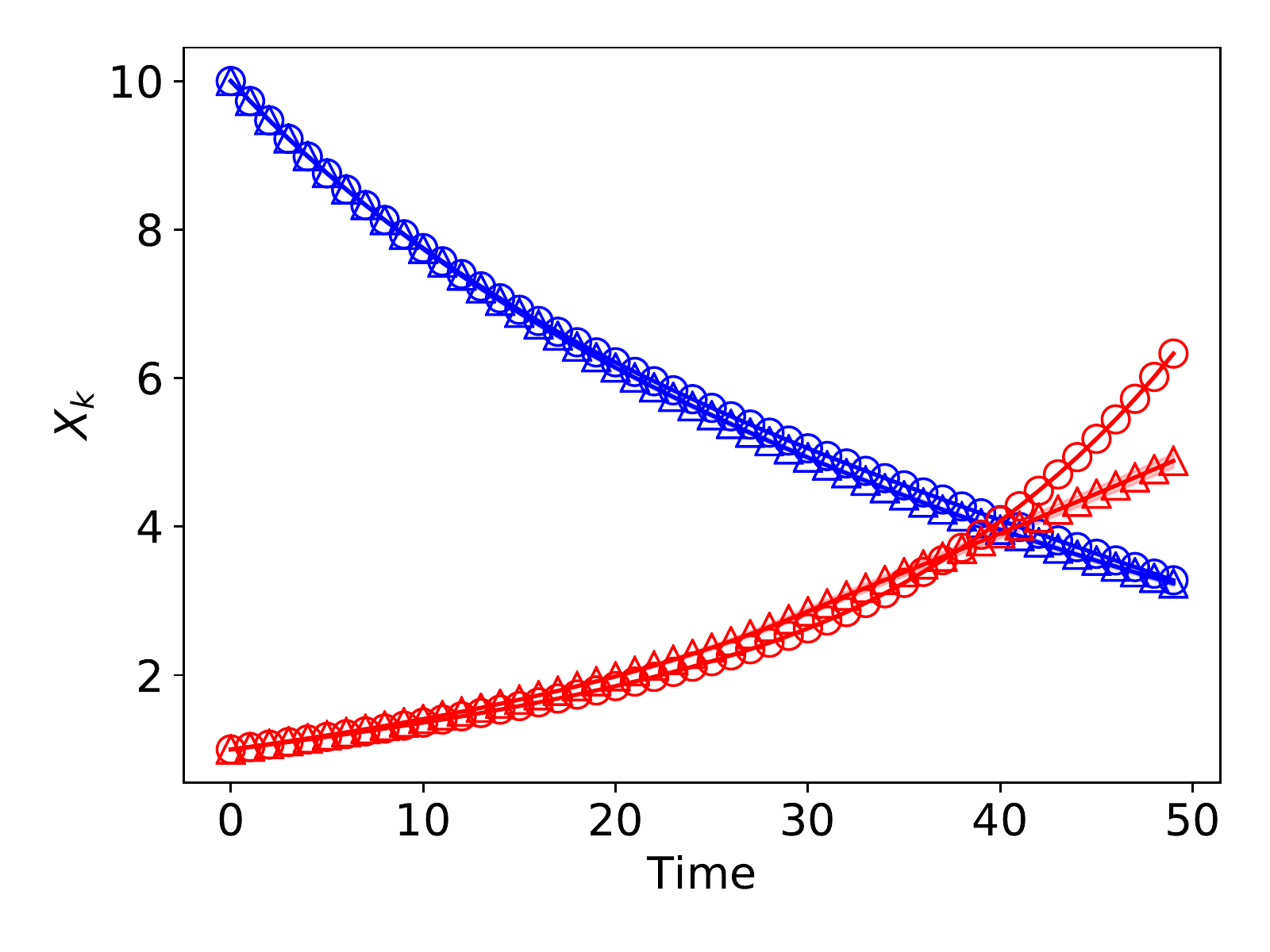}}
    \caption{Simulated states for coalescence equations with coefficients found by using (a) standard EKI and (b) sparse EKI. Two datasets are generated by coalescence equations with a higher order closure ($K=3$). The initial conditions of the training datasets are $(X_0,X_1,X_2)=(10,2,0.6)$ and $(X_0,X_1,X_2)=(10,2,2)$, and the initial condition of the simulations here is $(X_0,X_1,X_2)=(10,2,1)$.}
  \label{fig:states_ce_order_3_multi_training}
\end{figure}

\subsubsection{Gamma Versus Exponential Closure Data}
We perform a further study where, now, the true system and the modeled system have different closures. Specifically, the data are generated by simulating the coalescence equations in \eqref{eq:cem} with $K=2$, $r=3$, and with the exponential distribution closure in \eqref{eq:exp}. The goal is to fit a model with the same $K$ and $r$ but with a Gamma distribution closure, using EKI to estimate unknown coefficients $c_{ab}$. Results are presented in Figs.~\ref{fig:G_ce_exp} to~\ref{fig:states_ce_exp_multi_training}.

\begin{figure}[!htbp]
  \centering
  \includegraphics[width=0.2\textwidth]{G_legend}
  \subfloat[Standard EKI]{\includegraphics[width=0.49\textwidth]{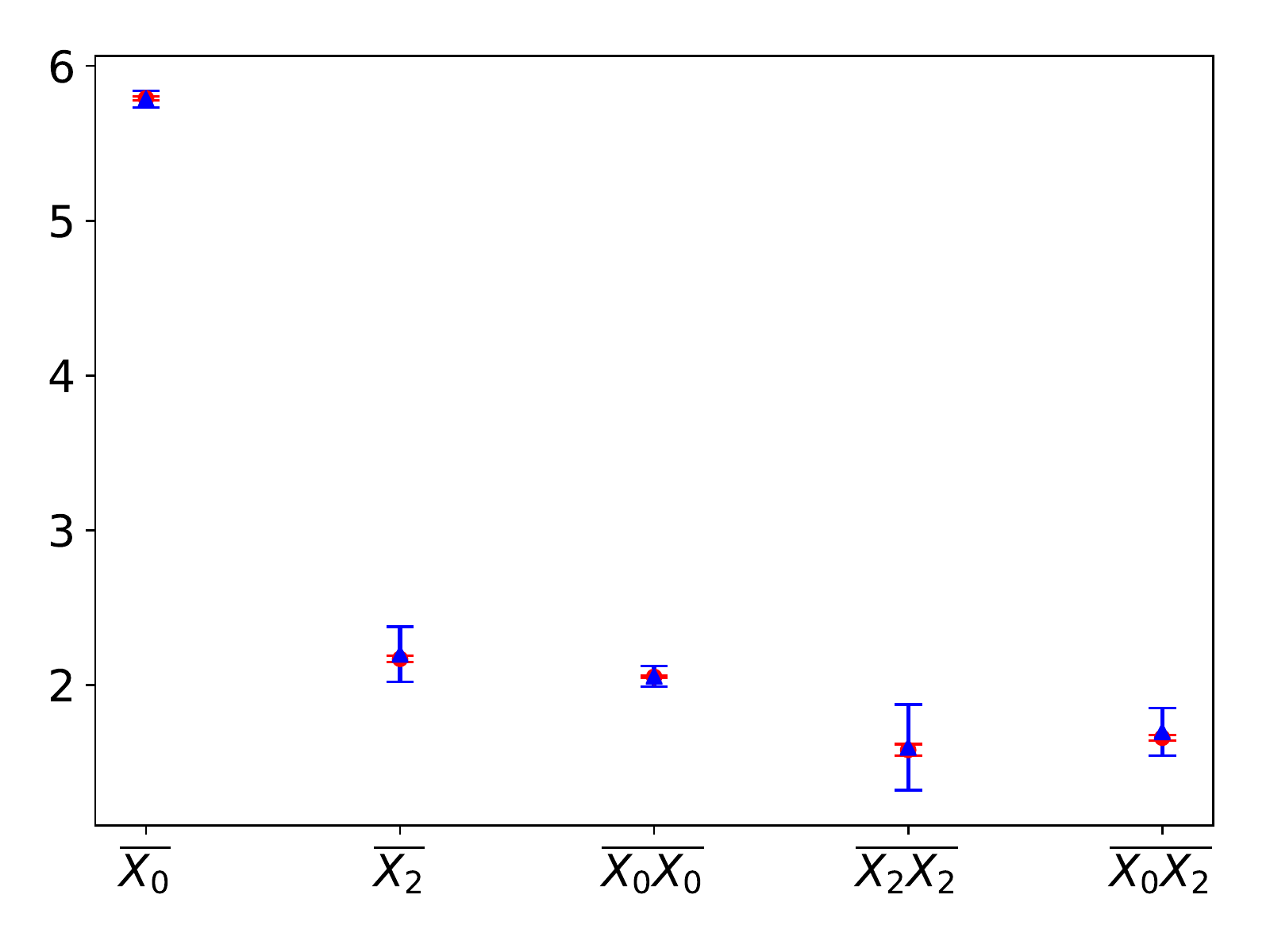}}
  \subfloat[Sparse EKI]{\includegraphics[width=0.49\textwidth]{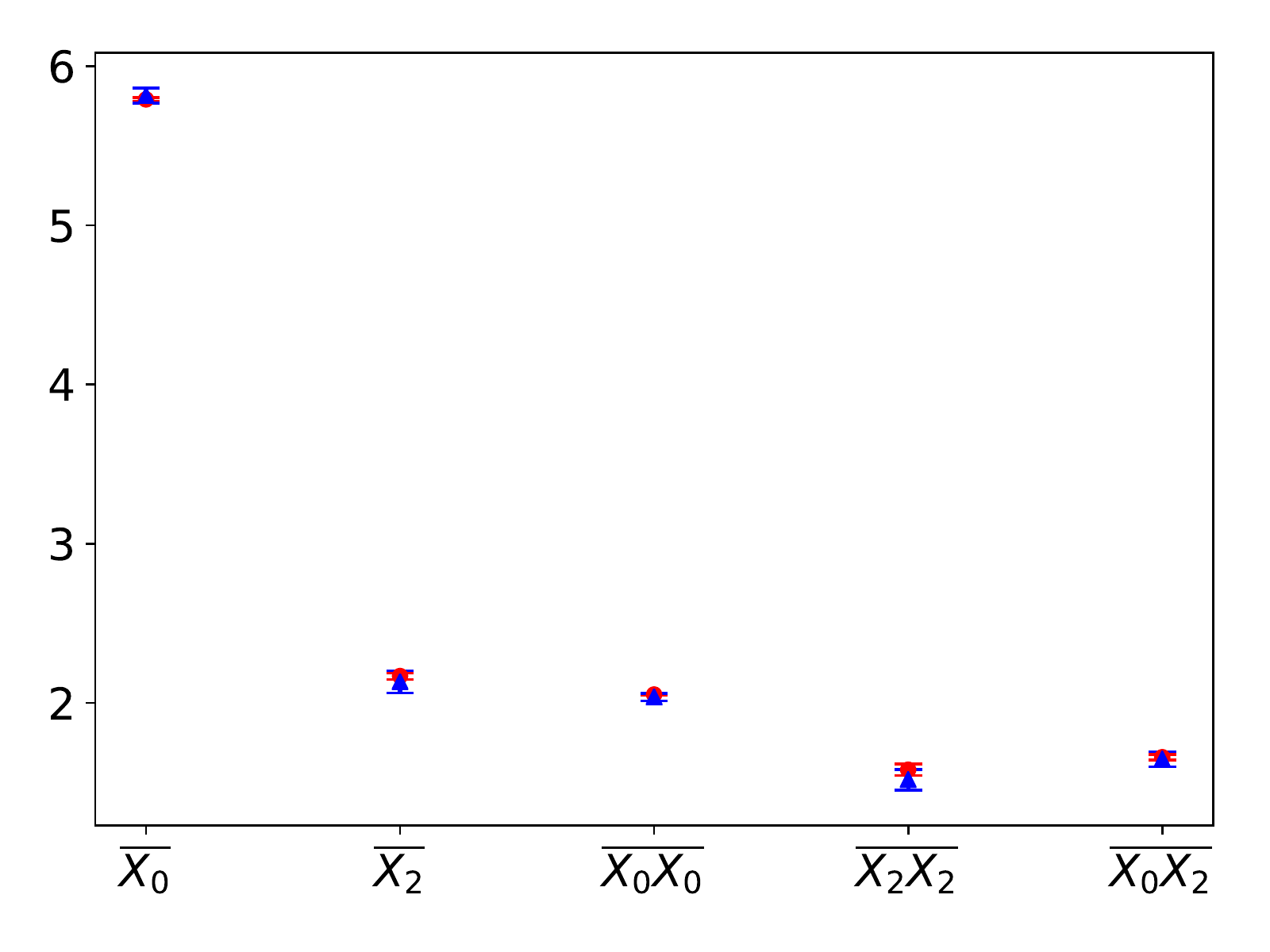}}
    \caption{First two moments of state $X$ for coalescence equations found by using (a) standard EKI and (b) sparse EKI. The data are generated by coalescence equations with an exponential closure.}
  \label{fig:G_ce_exp}
\end{figure}

The comparison of data in Fig.~\ref{fig:G_ce_exp} shows larger uncertainties for the results of standard EKI, while the ensemble mean agrees relatively well with the data of the true system. The larger uncertainties can also be seen in Fig.~\ref{fig:l1_norm_ce_exp}: the $\ell_1$-norm of all coefficients estimated using standard EKI remains relatively large, while the sparse EKI identifies another set of parameters with smaller $\ell_1$-norm.

\begin{figure}[!htbp]
  \centering
  \includegraphics[width=0.4\textwidth]{params_legend_sigma_no_truth}
  \subfloat[Standard EKI]{\includegraphics[width=0.49\textwidth]{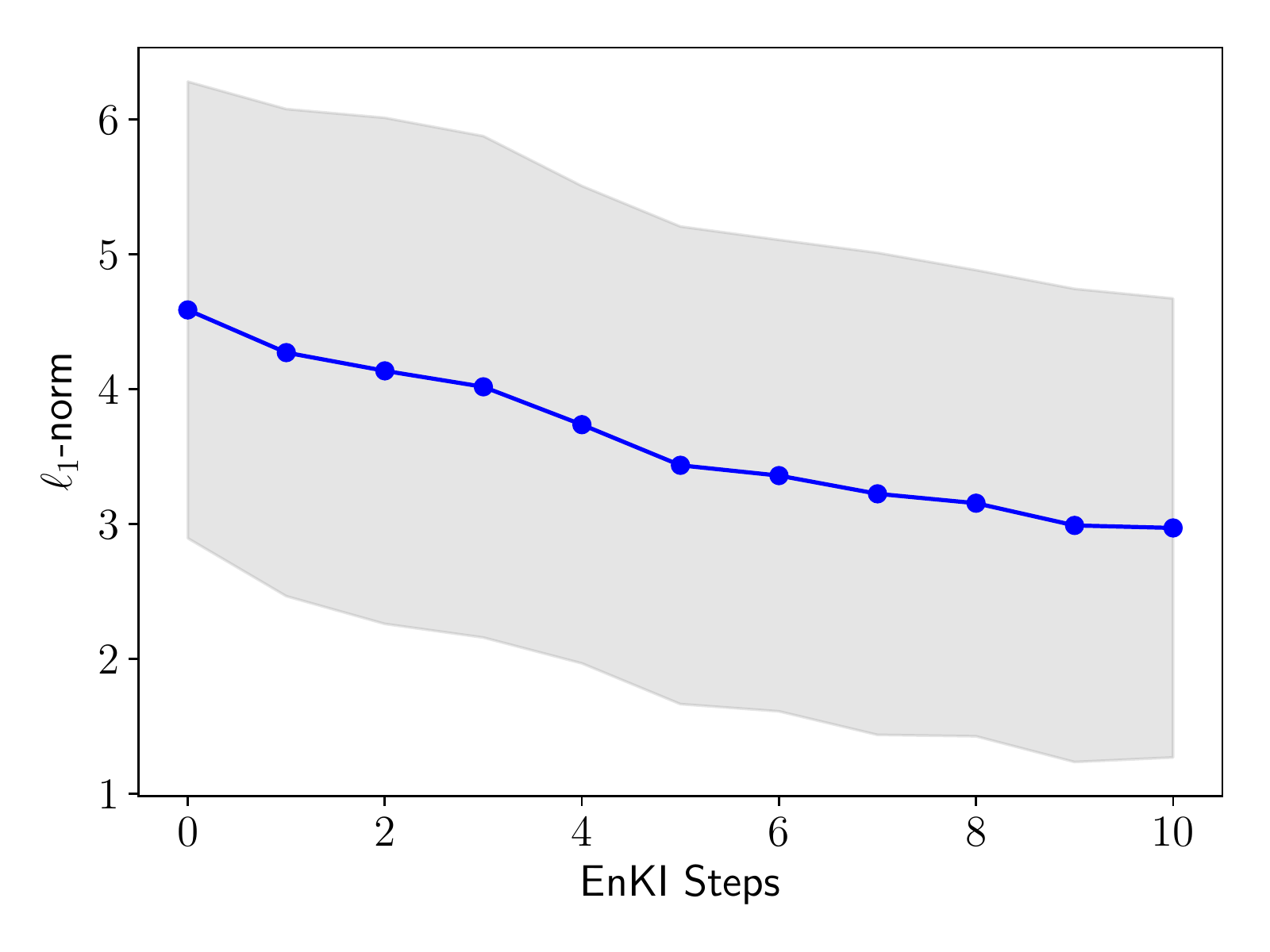}}
  \subfloat[Sparse EKI]{\includegraphics[width=0.49\textwidth]{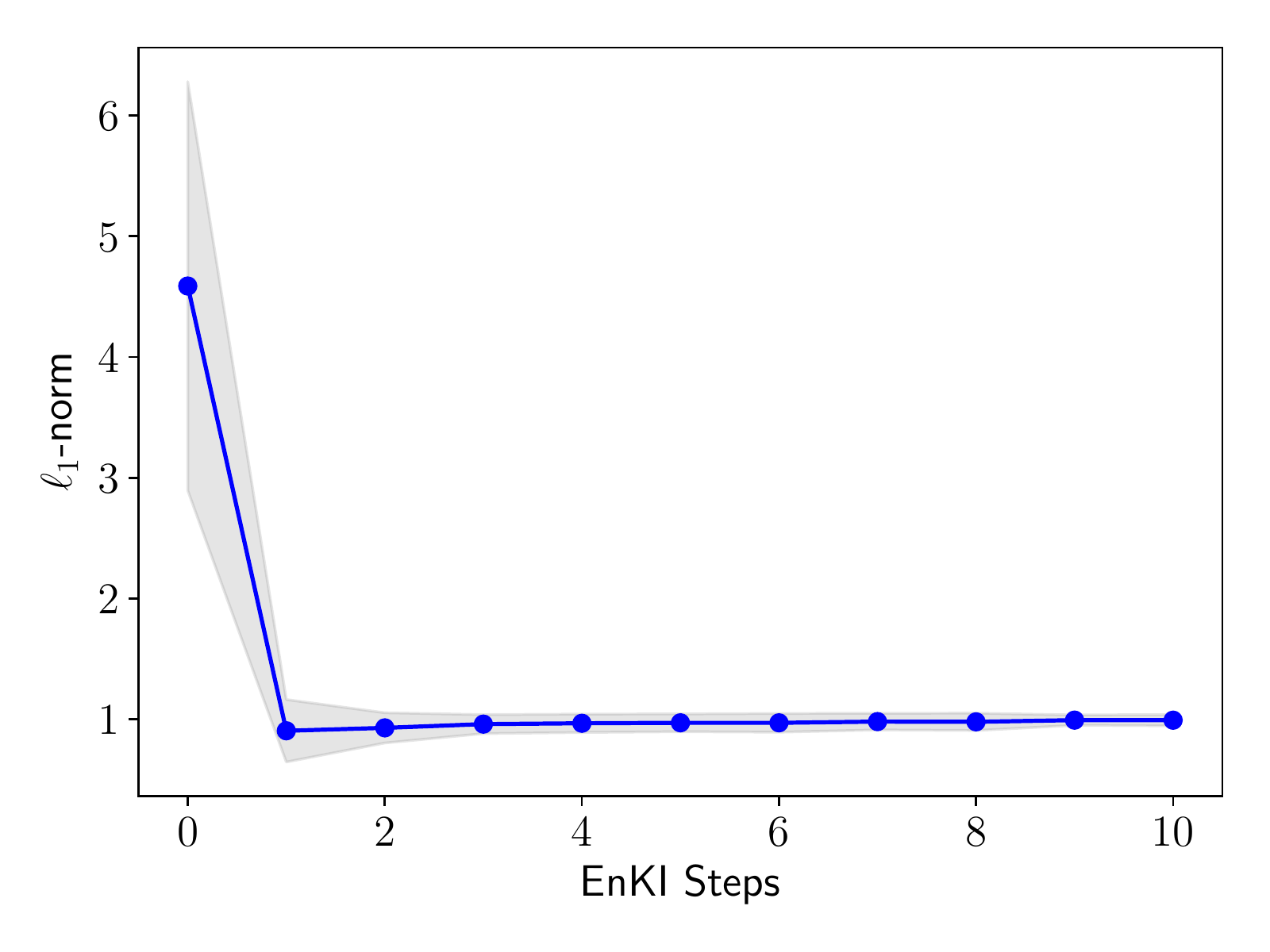}}
    \caption{$\ell_1$-norm of all coefficients for coalescence equations found by using (a) standard EKI and (b) sparse EKI. The data are generated by coalescence equations with an exponential closure.}
  \label{fig:l1_norm_ce_exp}
\end{figure}

We then investigate the performance of EKI identified systems by simulating state trajectories with an initial condition different from the training data. The comparison of simulated trajectories is presented in Fig.~\ref{fig:states_ce_exp}. Although the ensemble mean of either standard EKI or sparse EKI has similar agreement with the trajectories of the true system, there are also larger uncertainties in the ensemble of simulated trajectories for the results of standard EKI.

\begin{figure}[!htbp]
  \centering
  \includegraphics[width=0.6\textwidth]{state_ce_legend}
  \subfloat[Standard EKI]{\includegraphics[width=0.49\textwidth]{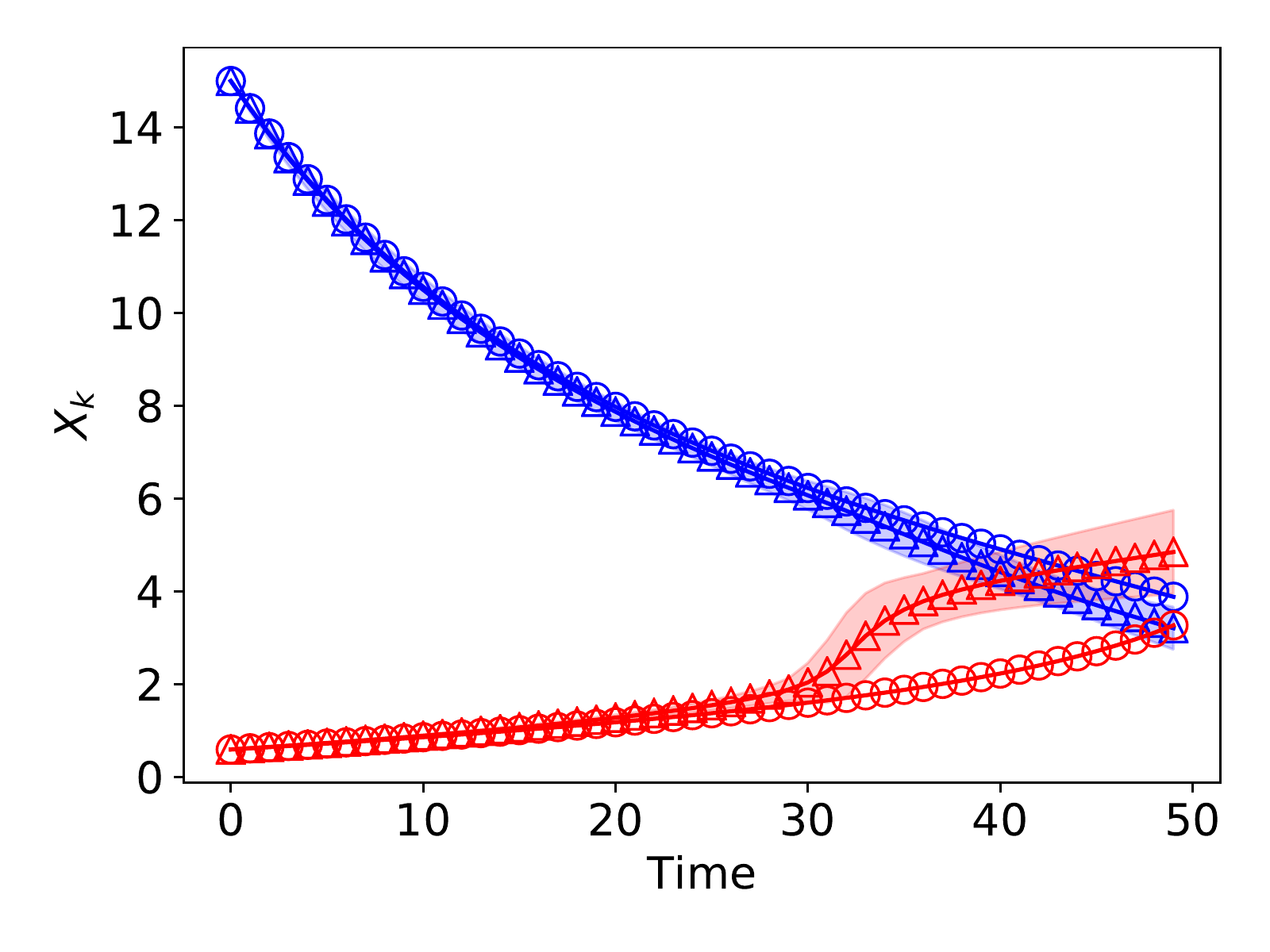}}
  \subfloat[Sparse EKI]{\includegraphics[width=0.49\textwidth]{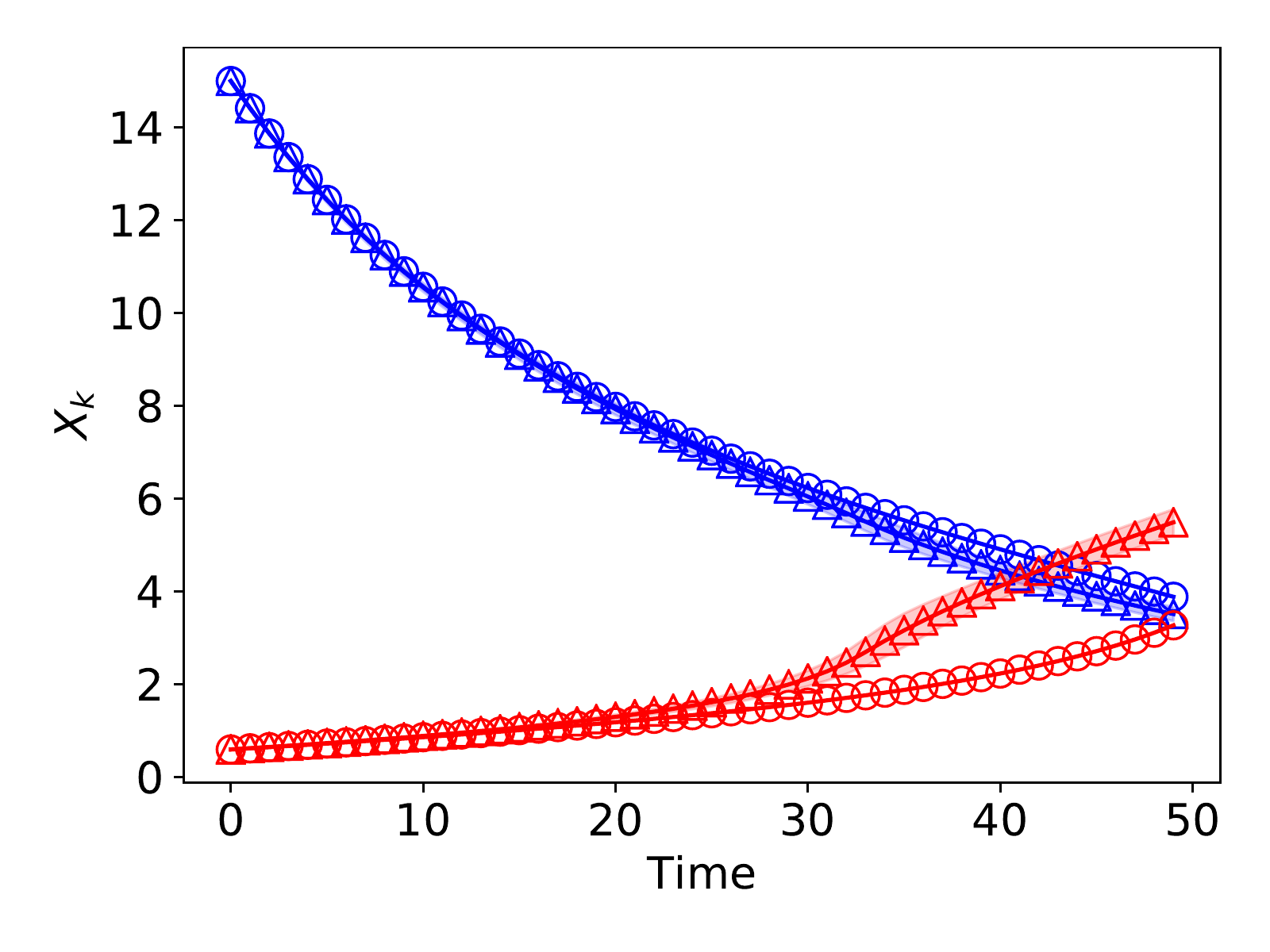}}
    \caption{Simulated states for coalescence equations with coefficients found by using (a) standard EKI and (b) sparse EKI.  The data are generated by coalescence equations with an exponential closure. The initial condition of the training dataset is $(X_0,X_1,X_2)=(10,2,0.6)$, and the initial condition of the simulations here is $(X_0,X_1,X_2)=(15,2,0.6)$.}
  \label{fig:states_ce_exp}
\end{figure}

We further demonstrate that the performance of EKI identified systems can be improved by using multiple training sets. Specifically, two training sets are used with initial conditions $(X_0,X_1,X_2)=(10,2,0.6)$ and $(X_0,X_1,X_2)=(20,2,0.6)$, and the initial condition of the test presented in Fig.~\ref{fig:states_ce_exp_multi_training} is $(X_0,X_1,X_2)=(15,2,0.6)$. Compared to the trajectories of EKI identified systems with a single training set in Fig.~\ref{fig:states_ce_exp}, the agreement of simulated trajectories with true ones is generally better in Fig.~\ref{fig:states_ce_exp_multi_training} when multiple training set being used.

\begin{figure}[!htbp]
  \centering
  \includegraphics[width=0.6\textwidth]{state_ce_legend}
  \subfloat[Standard EKI]{\includegraphics[width=0.49\textwidth]{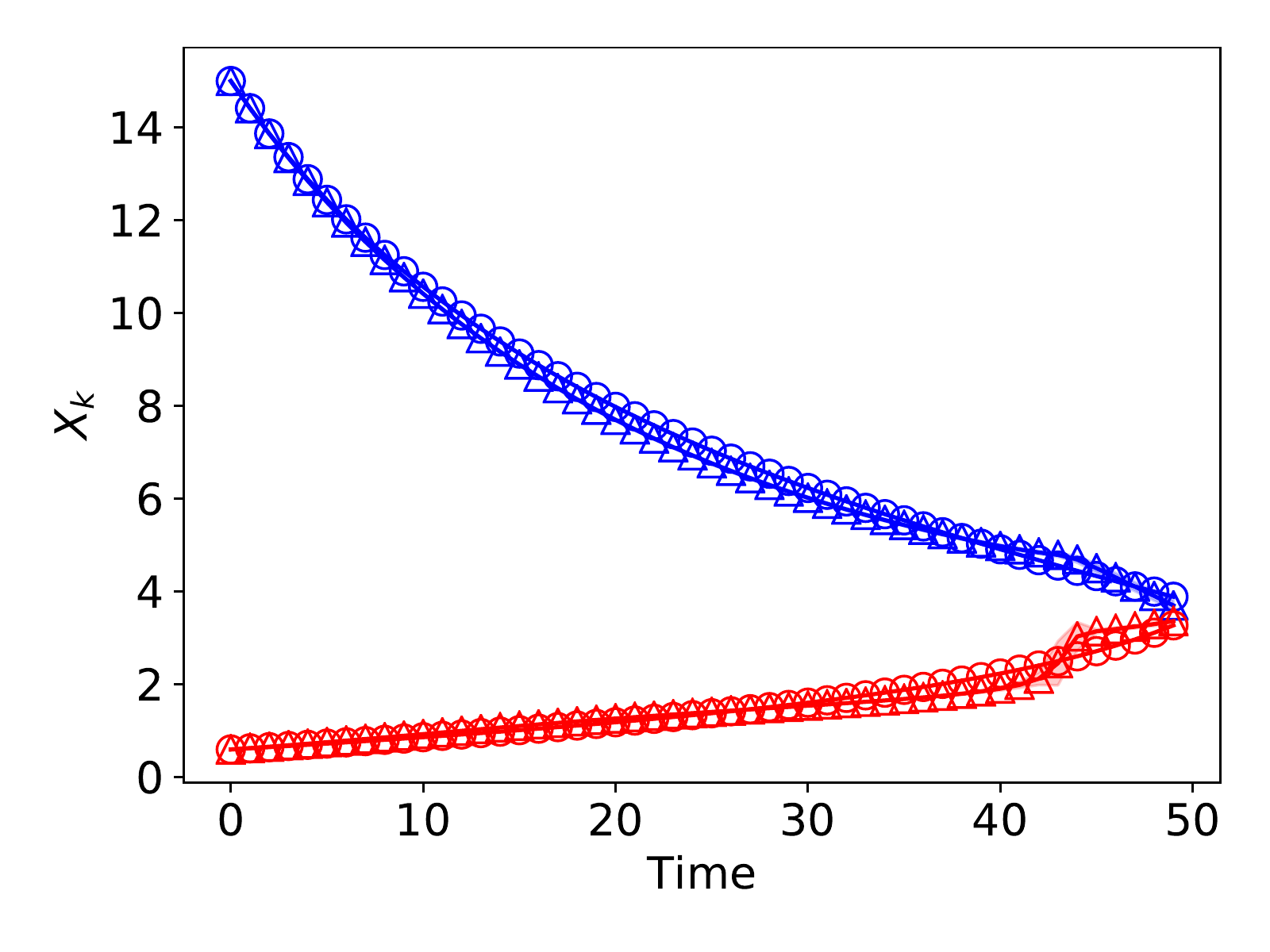}}
  \subfloat[Sparse EKI]{\includegraphics[width=0.49\textwidth]{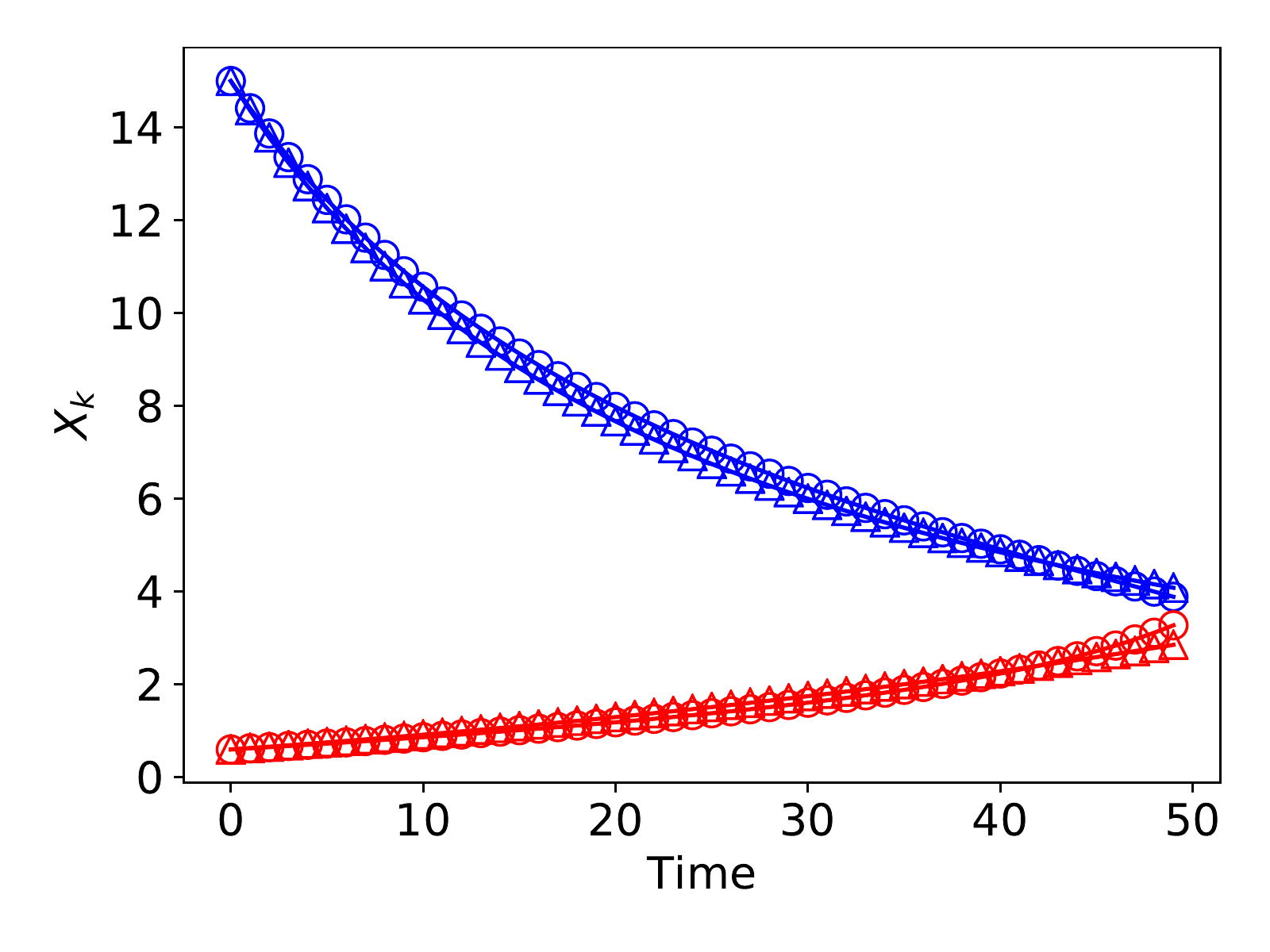}}
    \caption{Simulated states for coalescence equations with coefficients found by using (a) standard EKI and (b) sparse EKI.  Two datasets are generated by coalescence equations with an exponential closure. The initial conditions of the training datasets are $(X_0,X_1,X_2)=(10,2,0.6)$ and $(X_0,X_1,X_2)=(20,2,0.6)$, and the initial condition of the simulations here is $(X_0,X_1,X_2)=(15,2,0.6)$.}
  \label{fig:states_ce_exp_multi_training}
\end{figure}

\subsection{Kuramoto-Sivashinsky Equation}
\label{ssec:KS}
We conclude the numerical study by applying the sparse EKI to fit 
the Kuramoto-Sivashinsky Equation \eqref{eq:need1}. We first observe 
that applying the standard EKI approach to learn the equation,
from within the class represented in \eqref{eq:need2} and using
the same data detailed below for application of sparse EKI,
leads to a solution as presented in Table~\ref{tab:coeffs-KS}.
It fails to find a solution from within the class \eqref{eq:need2}
that is close to the data-generating equation \eqref{eq:need1},
clearly motivating the need for the sparse EKI method, 
results from which are also shown in the same table.

\begin{table}[htbp]
\caption{Mean value of coefficients estimated by standard EKI.}
\centering
\begin{tabular}{c|ccccc}
\hline
Linear terms & $\partial^1_x u$ & $\partial^2_x u$ & $\partial^3_x u$ & $\partial^4_x u$ & $\partial^5_x u$ \\
Coefficient (Standard EKI) & -0.330 & 1.385 & 0.659 & 1.262 & -0.130 \\
Coefficient (Sparse EKI) & 0 & 1.020 & 0 & 1.020 & 0 \\
Coefficient (Truth) & 0 & 1 & 0 & 1 & 0 \\
\hline
Non-linear terms & $u \partial_x u$ & $u^{2} \partial_x u$ & $u^{3} \partial_x u$ & $u^{4} \partial_x u$ & $u^{5} \partial_x u$ \\
Coefficient (Standard EKI) & 1.420 & 0.224 & -0.455 & 0.104 & 0.149 \\
Coefficient (Sparse EKI) & 1.024 & 0 & 0 & 0 & 0 \\
Coefficient (Truth) & 1 & 0 & 0 & 0 & 0 \\
\hline
\end{tabular}
\label{tab:coeffs-KS}
\end{table}

We now turn to the sparse setting. Working within the class
of models \eqref{eq:need2} requires 10 unknown coefficients $\{\alpha_j,\beta_j\}_{j=1}^5$ to be learnt. To do this we will use a data vector $y$ of dimension 114. Specifically, the data vector consists of: (i) the first to fourth moments at eight locations $\{x_j\}_{j=1}^8$ that are evenly distributed across the range of $x$,
namely $\{\overline{u_j},\{\overline{u_ju_k}\}_{k=1}^8,\overline{u_ju_ju_j},\overline{u_ju_ju_ju_j}\}_{j=1}^8$,  giving a total moment-data vector of size $8+36+8+8=60$; (ii) temporal autocorrelation  of $u(x_j,t)$ at the same eight locations of $x$ and using five points
in time, giving a total autocorrelation-data vector of size $40$; and (iii) the time-averaged spatial correlation function at
$14$ locations in space $x$. The time used for averaging is $T=1000$ and all simulations are performed on the torus $[0,L]$, with $L=128$. Details of the methods employed to solve the extended K-S equation
\eqref{eq:need2} are detailed in \ref{sec:KS_solver}, including
the Fourier-based approach to finding the spatial correlation function.

Results of the first sparse EKI (with all ten basis functions) are presented in Figs.~\ref{fig:G_KS_1st_EKI} and~\ref{fig:params_KS_1st_EKI}, and results of the second sparse EKI (using a reduced number (four) of basis functions,  informed by the
first phase of the algorithm, using the approach discussed in subsection~\ref{ssec:sEKI}), are presented in Figs.~\ref{fig:G_KS_2nd_EKI} and~\ref{fig:params_KS_2nd_EKI}.

\begin{figure}[!htbp]
  \centering
  \includegraphics[width=0.2\textwidth]{G_legend}
  \subfloat[Moments]{\includegraphics[width=0.49\textwidth]{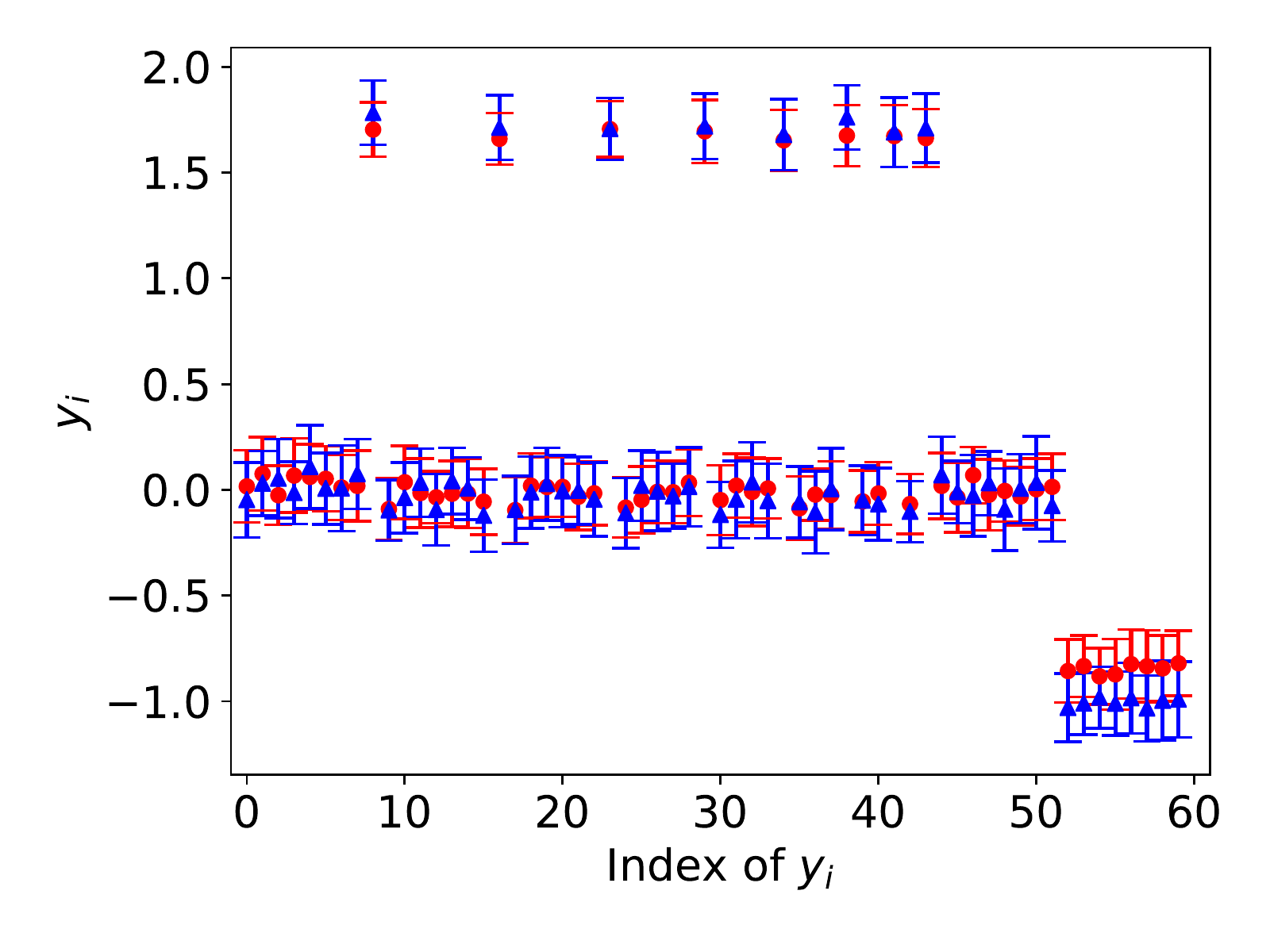}}
  \subfloat[Autocorrelation]{\includegraphics[width=0.49\textwidth]{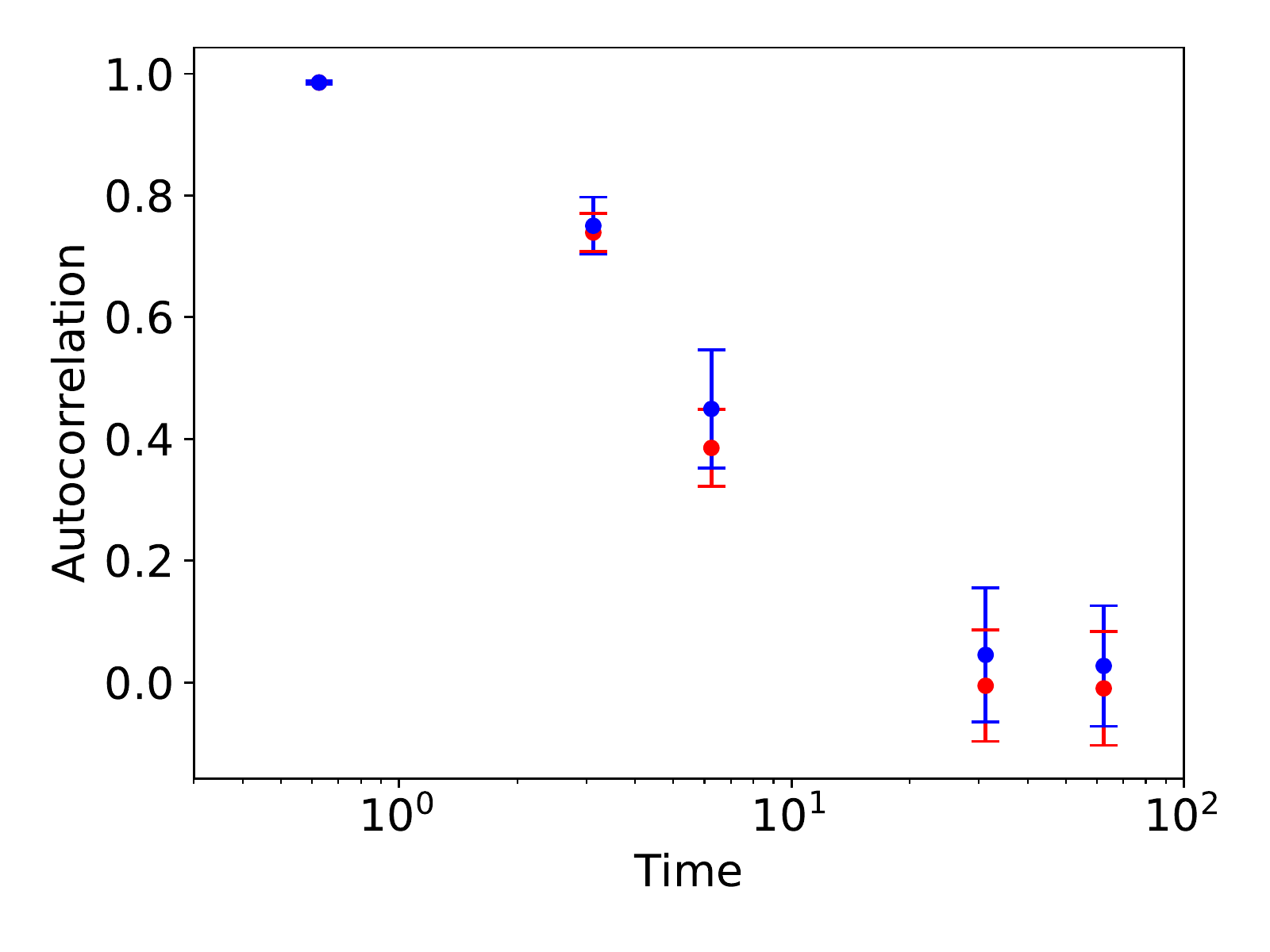}}\\
  \subfloat[Spatial correlation]{\includegraphics[width=0.99\textwidth]{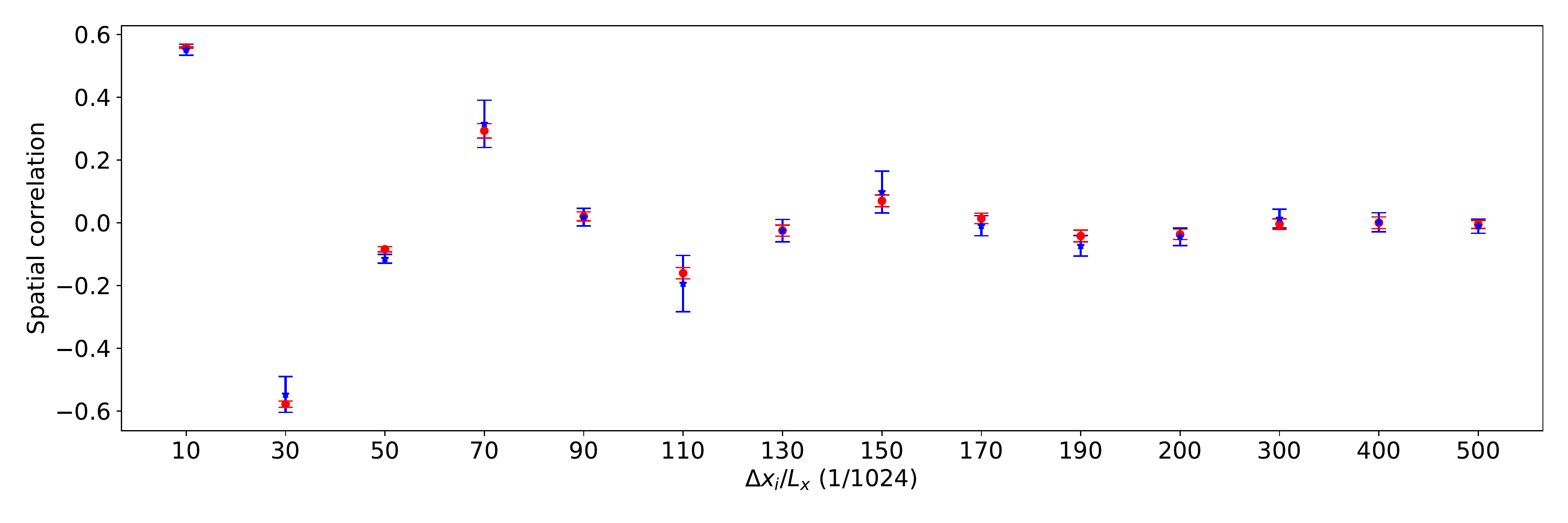}}
    \caption{Comparison between the data from the true system and results of the first sparse EKI, including (a) first four moments, (b) autocorrelation at $x=0$, and (c) time-averaged spatial correlation.}
  \label{fig:G_KS_1st_EKI}
\end{figure}

The comparison of data is presented in Fig.~\ref{fig:G_KS_1st_EKI} for the first sparse EKI. The comparison of the autocorrelation results at the eight locations is similar, and thus we only present the autocorrelation results at $x=0$. In terms of all three types of data, there are some mismatches between the results of sparse EKI and the true data. In order to evaluate the performance of sparse EKI more precisely, we present the learning of three necessary coefficients ($\alpha_2$, $\alpha_4$, and $\beta_1$) and the $\ell_1$-norm of redundant coefficients in Fig.~\ref{fig:params_KS_1st_EKI}. It is clear that there are some biases in the estimated parameters $\alpha_2$ and $\alpha_4$, while the sparse EKI successfully drives the $\ell_1$-norm of redundant coefficients close to zero.

\begin{figure}[!htbp]
  \centering
  \includegraphics[width=0.55\textwidth]{params_legend_sigma}
  \subfloat[$\alpha_2$]{\includegraphics[width=0.49\textwidth]{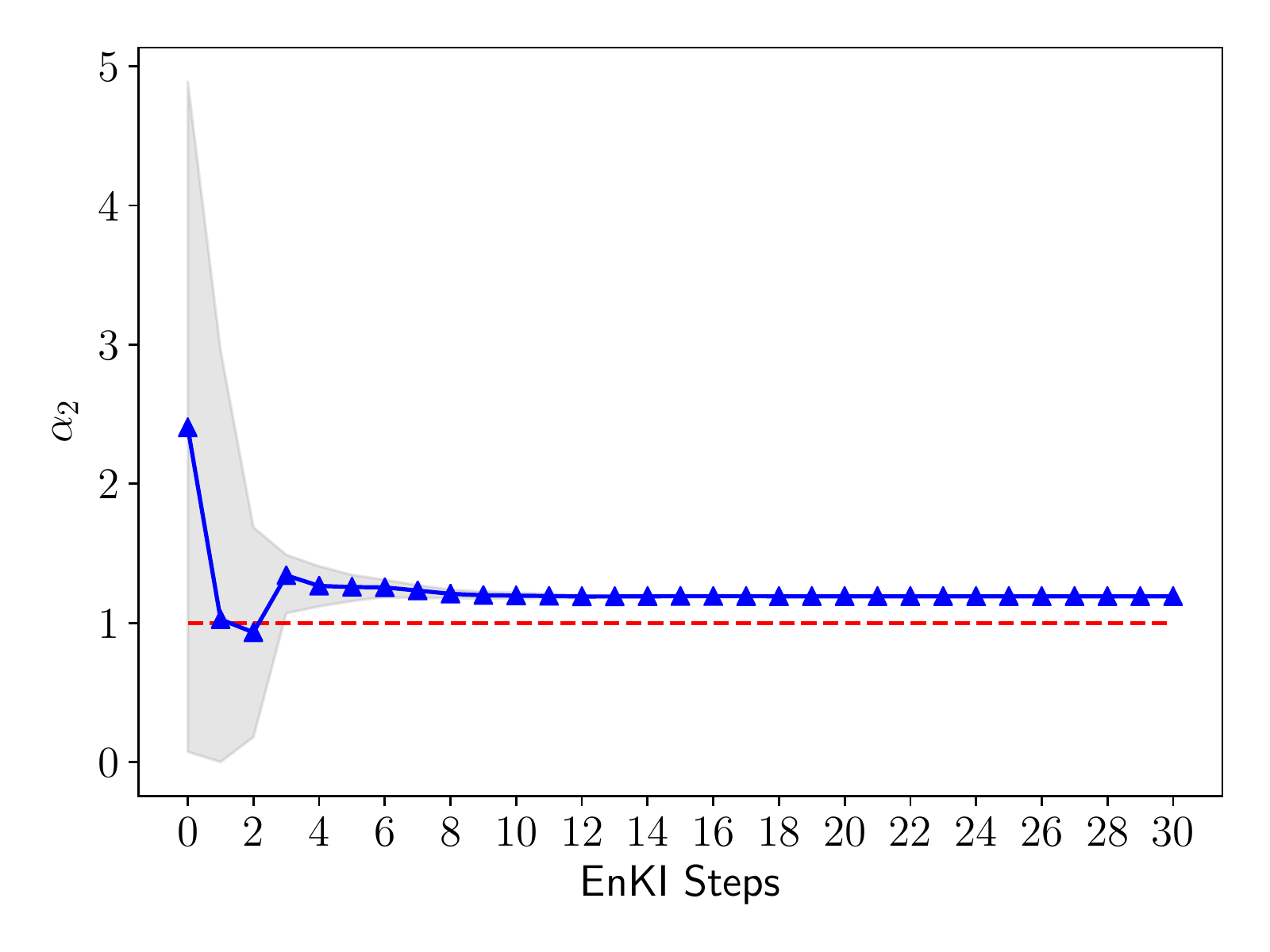}}
  \subfloat[$\alpha_4$]{\includegraphics[width=0.49\textwidth]{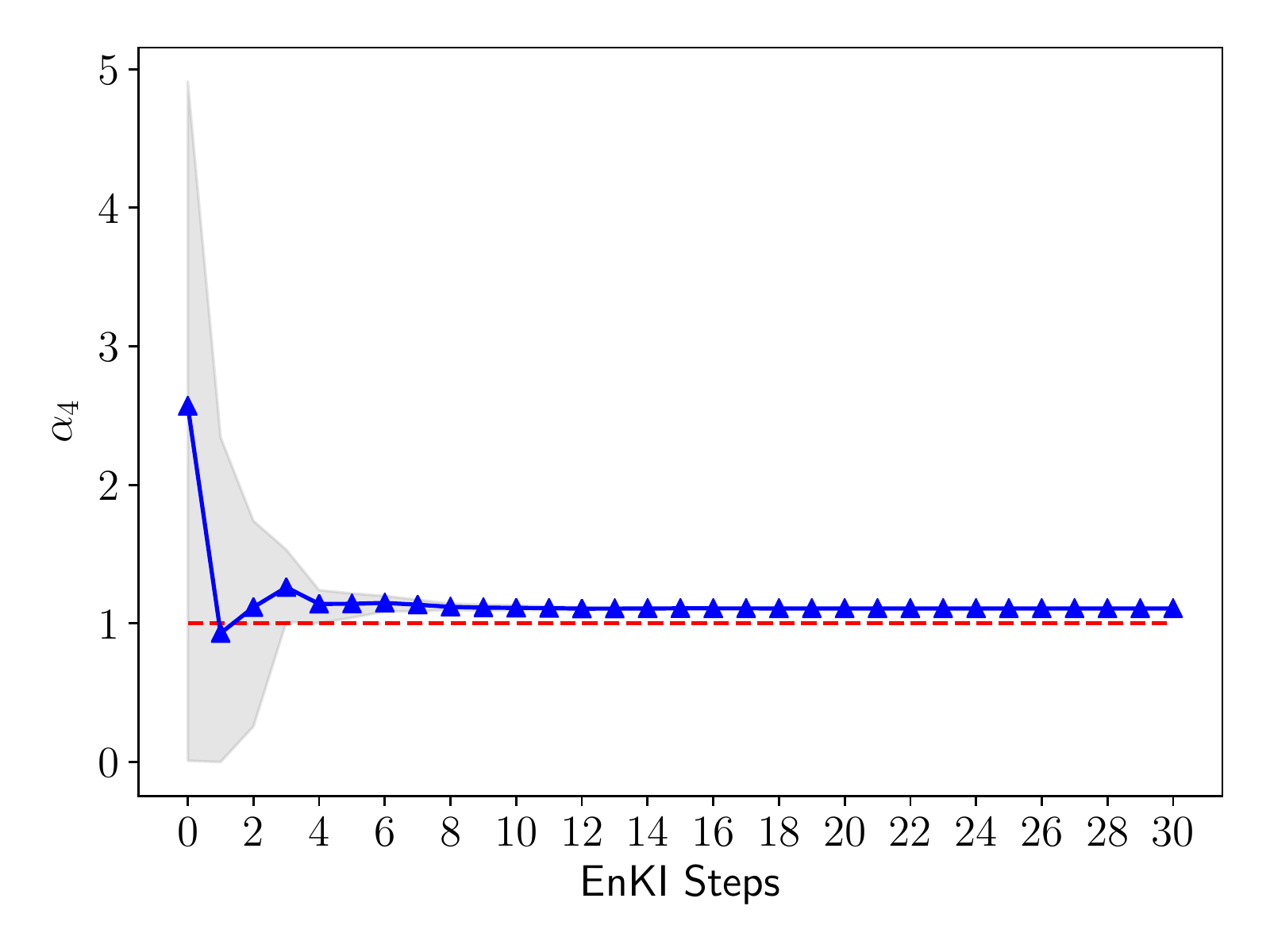}}\\
  \subfloat[$\beta_1$]{\includegraphics[width=0.49\textwidth]{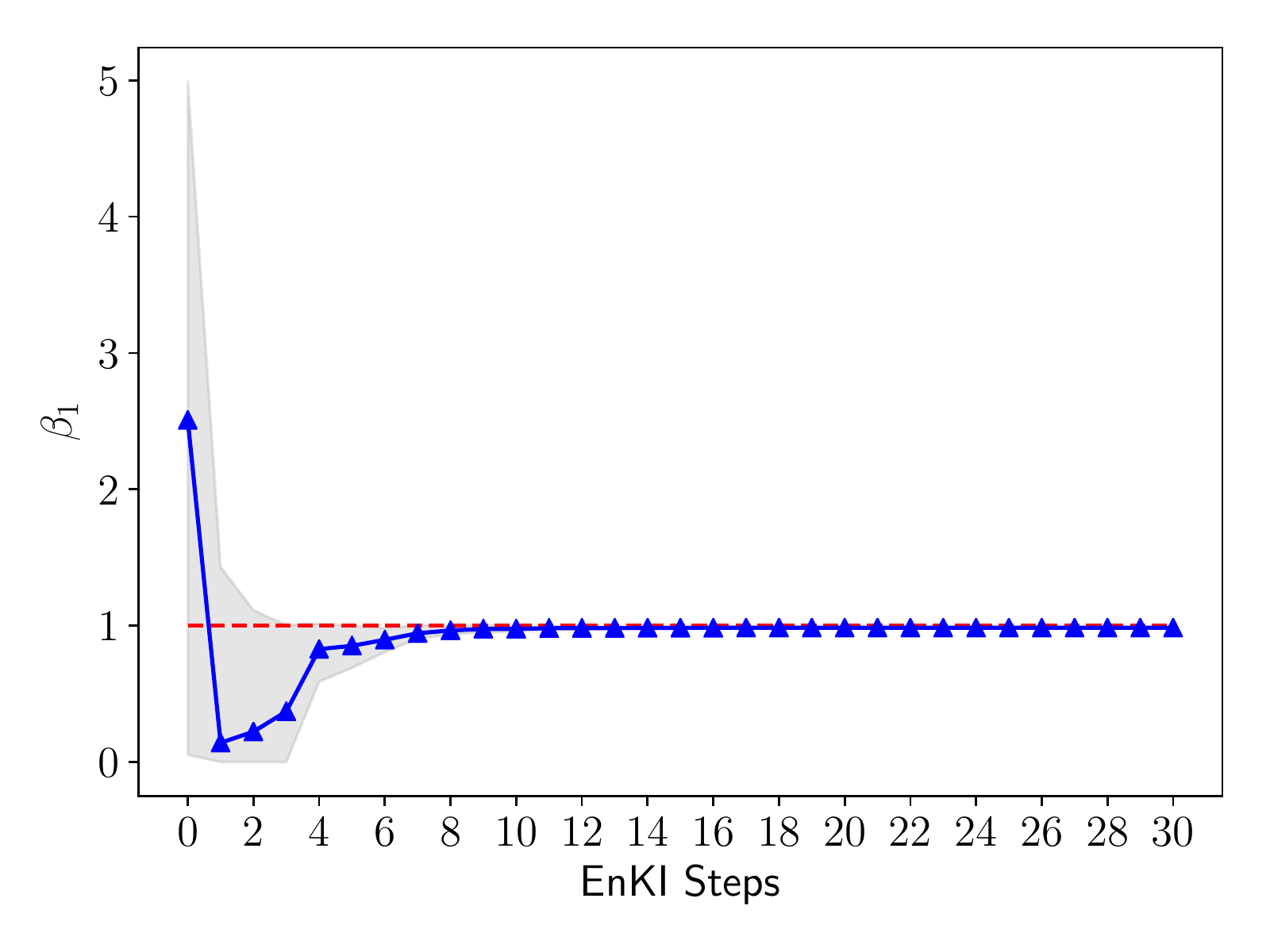}}
  \subfloat[Other coefficients]{\includegraphics[width=0.49\textwidth]{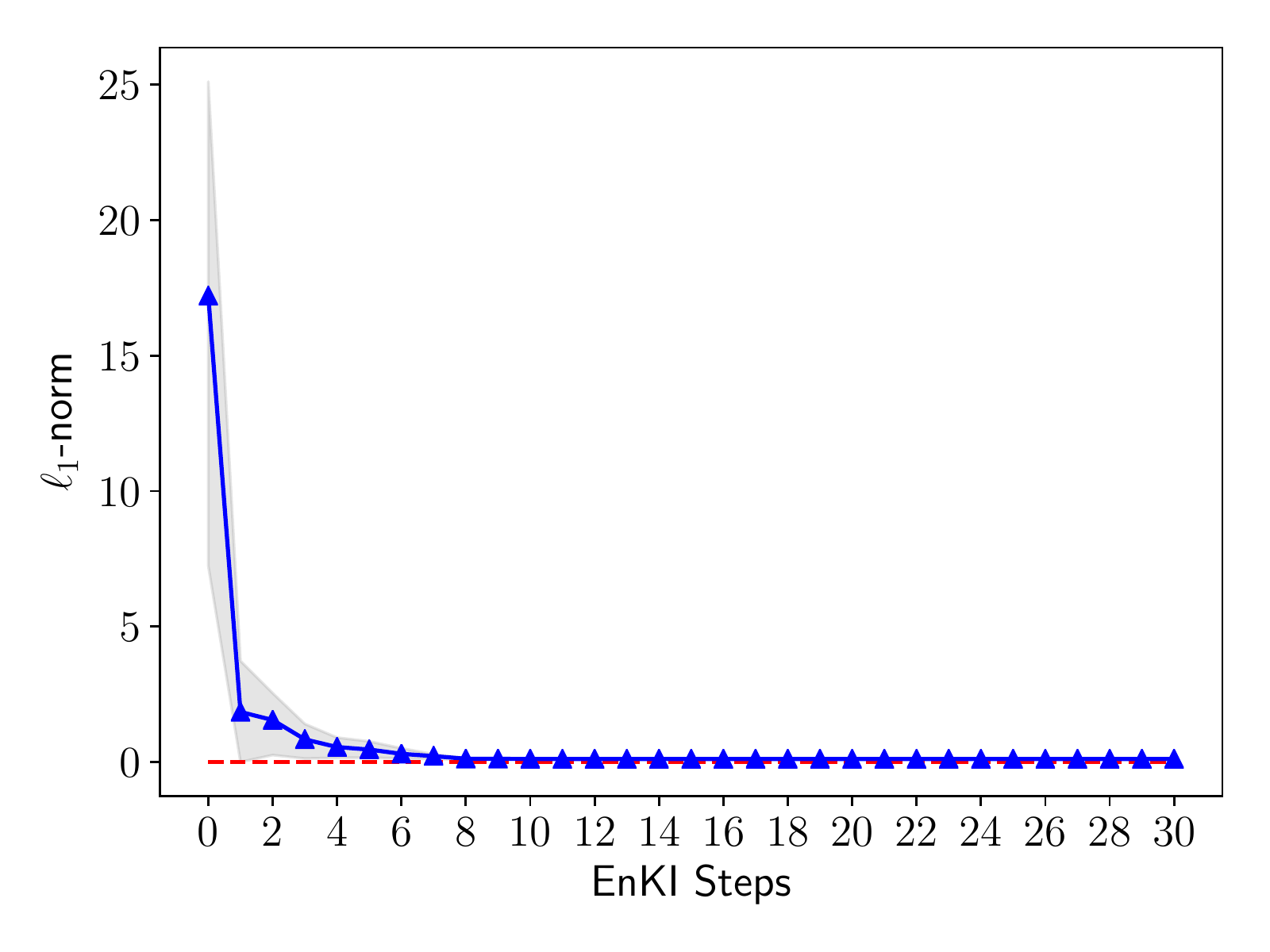}}
    \caption{Estimated coefficients from the first EKI, including (a)--(c) necessary coefficients of Kuramoto-Sivashinsky Equation ($\alpha_2$, $\alpha_4$, and $\beta_1$), and (d) the $\ell_1$-norm of other coefficients. There are four nonzero coefficients ($\alpha_2$, $\alpha_4$, $\beta_1$, and $\beta_3$) in the results of first EKI.}
  \label{fig:params_KS_1st_EKI}
\end{figure}

In order to improve the accuracy of the estimated parameters in Fig.~\ref{fig:params_KS_1st_EKI}, we perform a second sparse EKI with reduced basis functions, i.e., those with non-zero coefficients ($\alpha_2$, $\alpha_4$, $\beta_1$, and $\beta_3$) in the results obtained in the
first application of sparse EKI. The comparison of data is presented in Fig.~\ref{fig:G_KS_2nd_EKI} for the second sparse EKI, which shows a much better agreement with all three types of data. The estimated parameters from the second sparse EKI are presented in Fig.~\ref{fig:params_KS_2nd_EKI}, confirming that the Kuramoto-Sivashinsky equation can be accurately identified by using sparse EKI. The results are summarized in Table~\ref{tab:coeffs-KS}
demonstrating that the sparse EKI method correctly recovers the
three non-zero coefficients to an accuracy of less than $2.5\%$ and correctly
zeros out all other coefficients; in contrast, the standard EKI finds a non-sparse  fit to
the data in which all $10$ basis coefficients are active. This
concluding example demonstrates both the power of sparsity promoting
learning of dynamical systems, and the ability of the sparse EKI method to learn dynamical systems from indirect, partial and nonlinear observations.

\begin{figure}[!htbp]
  \centering
  \includegraphics[width=0.2\textwidth]{G_legend}
  \subfloat[Moments]{\includegraphics[width=0.49\textwidth]{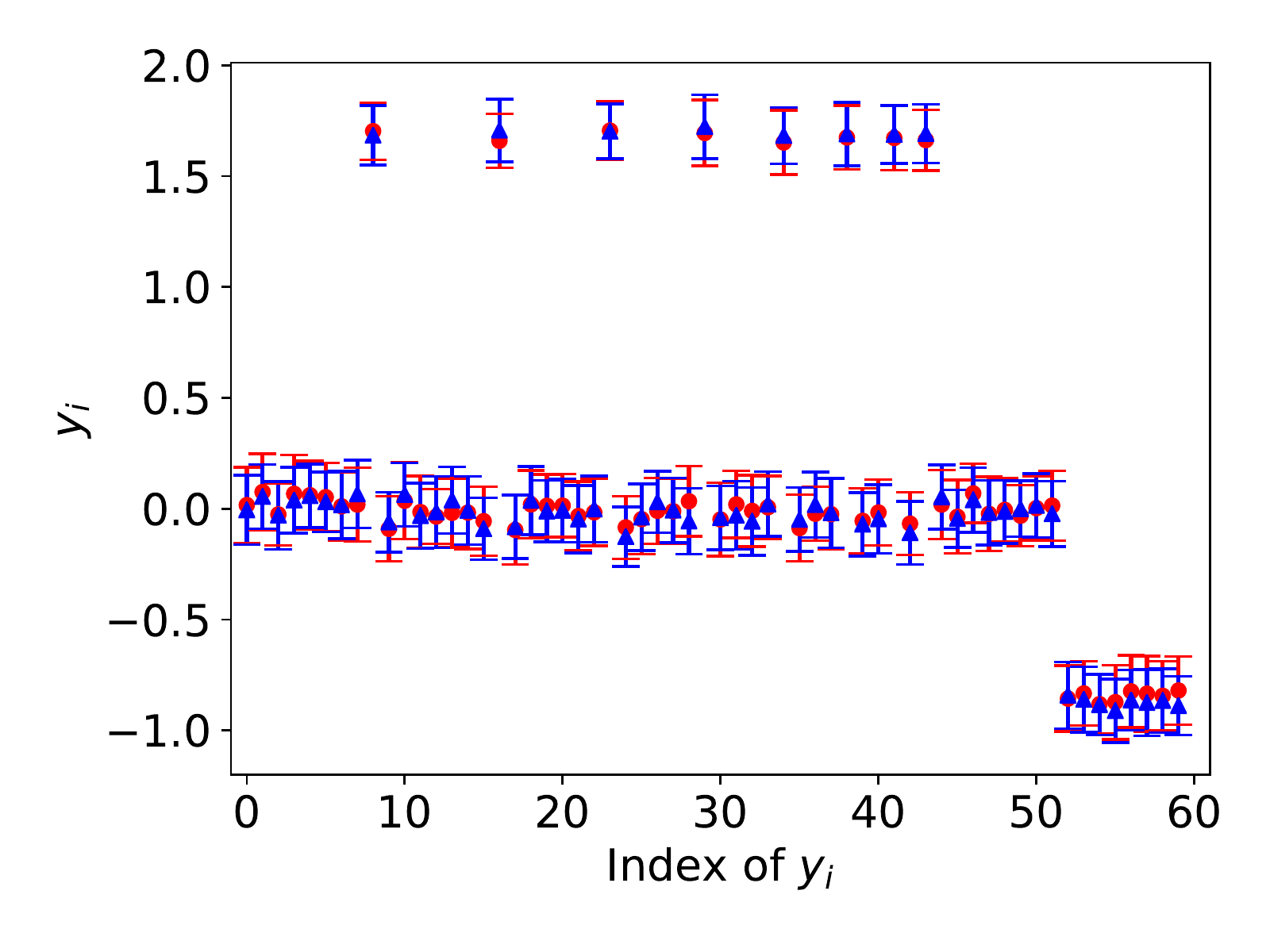}}
  \subfloat[Autocorrelation]{\includegraphics[width=0.49\textwidth]{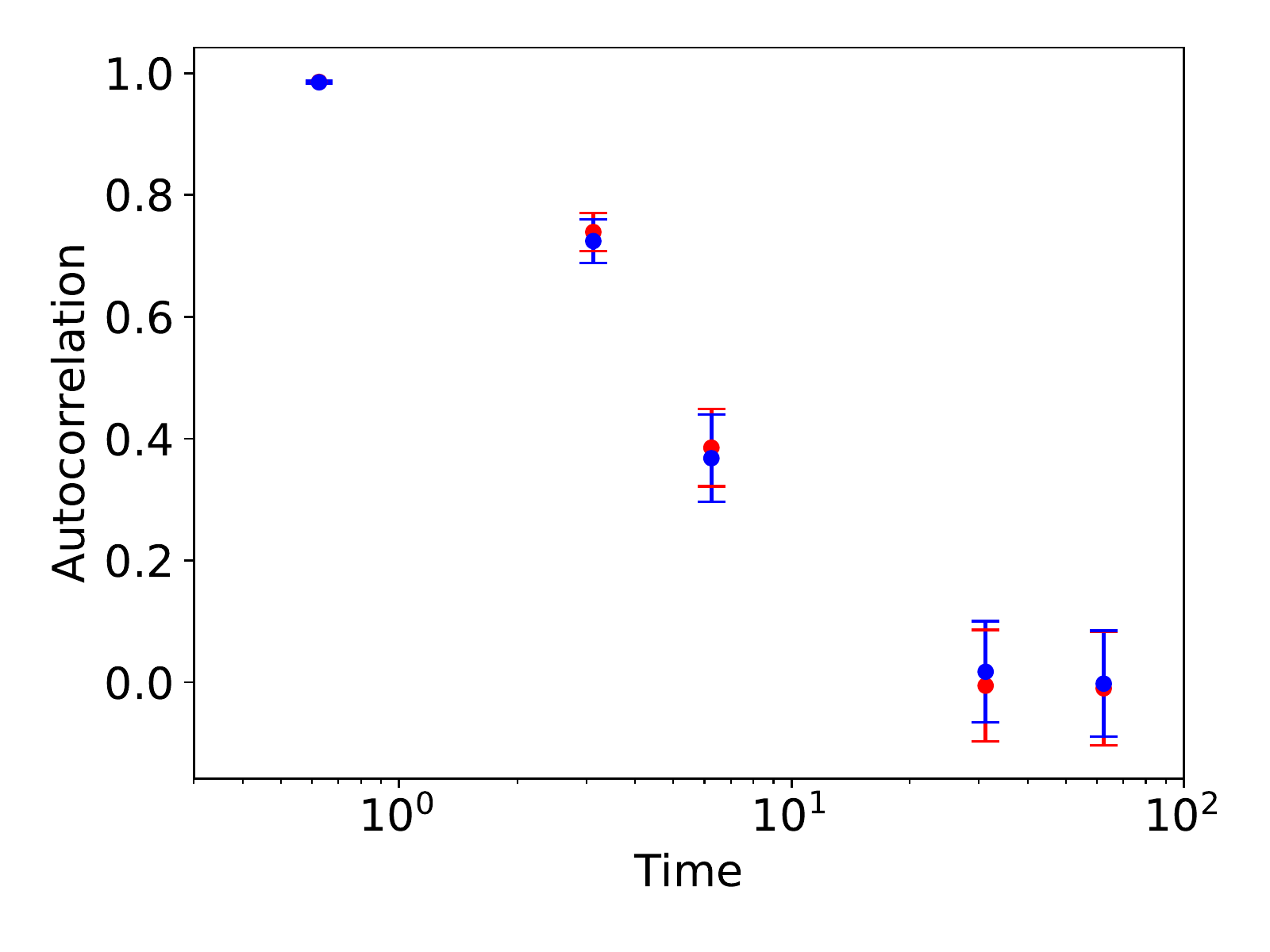}}\\
  \subfloat[Spatial correlation]{\includegraphics[width=0.99\textwidth]{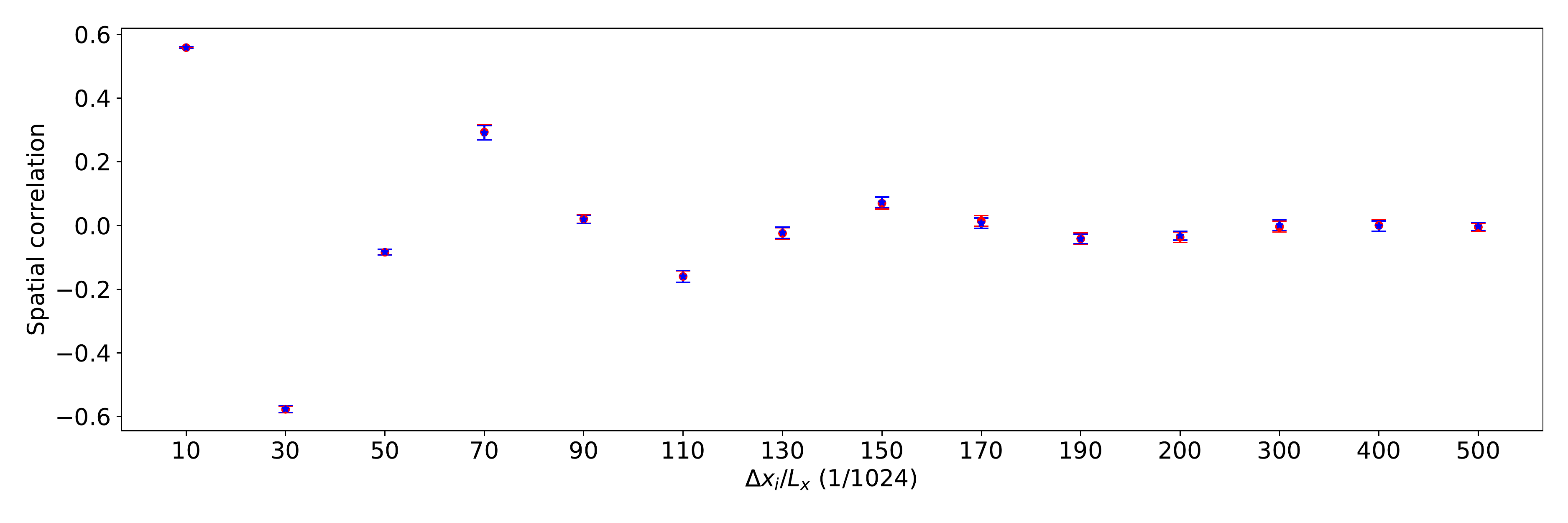}}
    \caption{Comparison between the data from the true system and results of the second sparse EKI, including (a) first four moments, (b) autocorrelation at $x=0$, and (c) time-averaged spatial correlation.}
  \label{fig:G_KS_2nd_EKI}
\end{figure}

\begin{figure}[!htbp]
  \centering
  \includegraphics[width=0.55\textwidth]{params_legend_sigma}
  \subfloat[$\alpha_2$]{\includegraphics[width=0.49\textwidth]{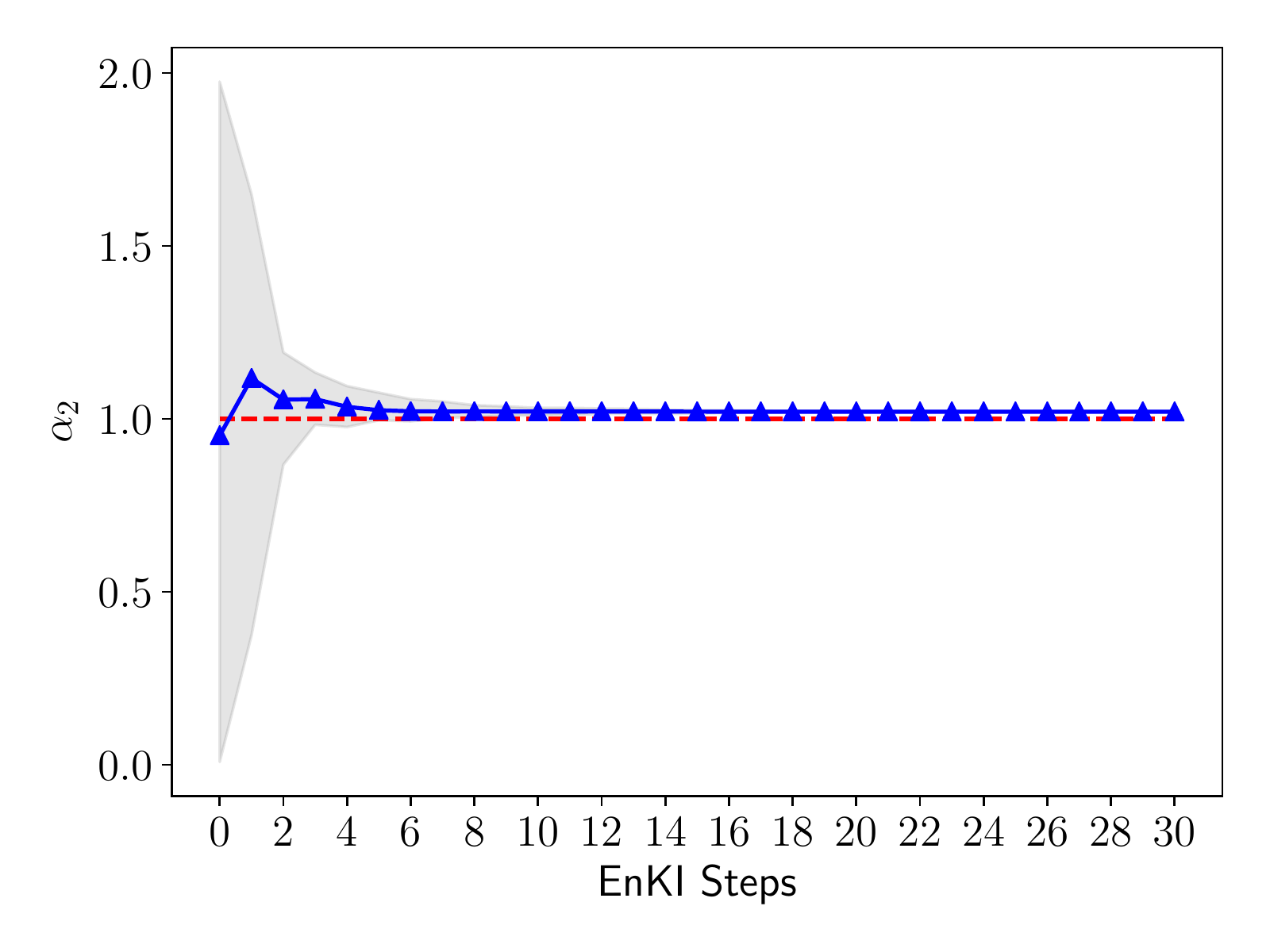}}
  \subfloat[$\alpha_4$]{\includegraphics[width=0.49\textwidth]{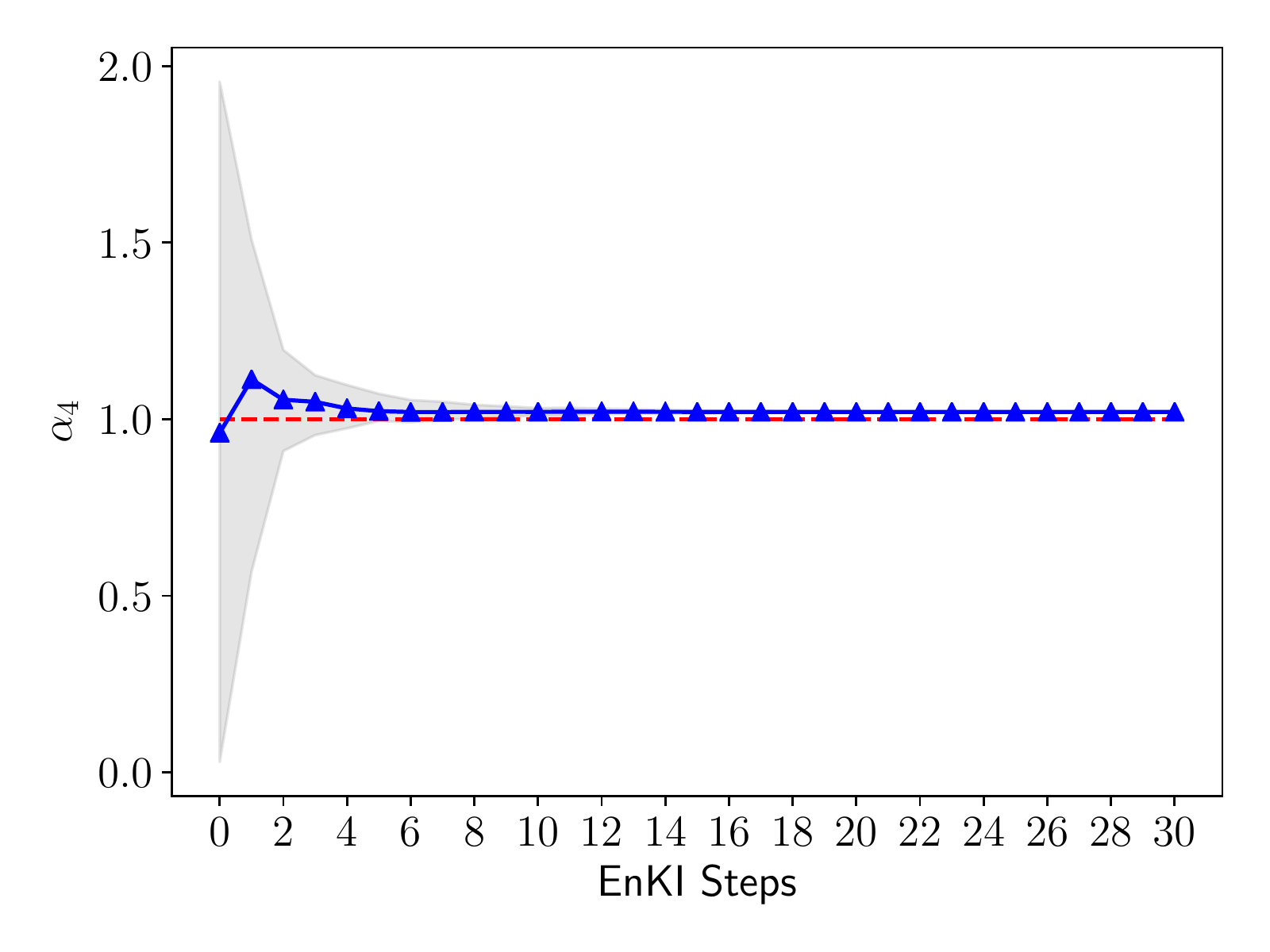}}\\
  \subfloat[$\beta_1$]{\includegraphics[width=0.49\textwidth]{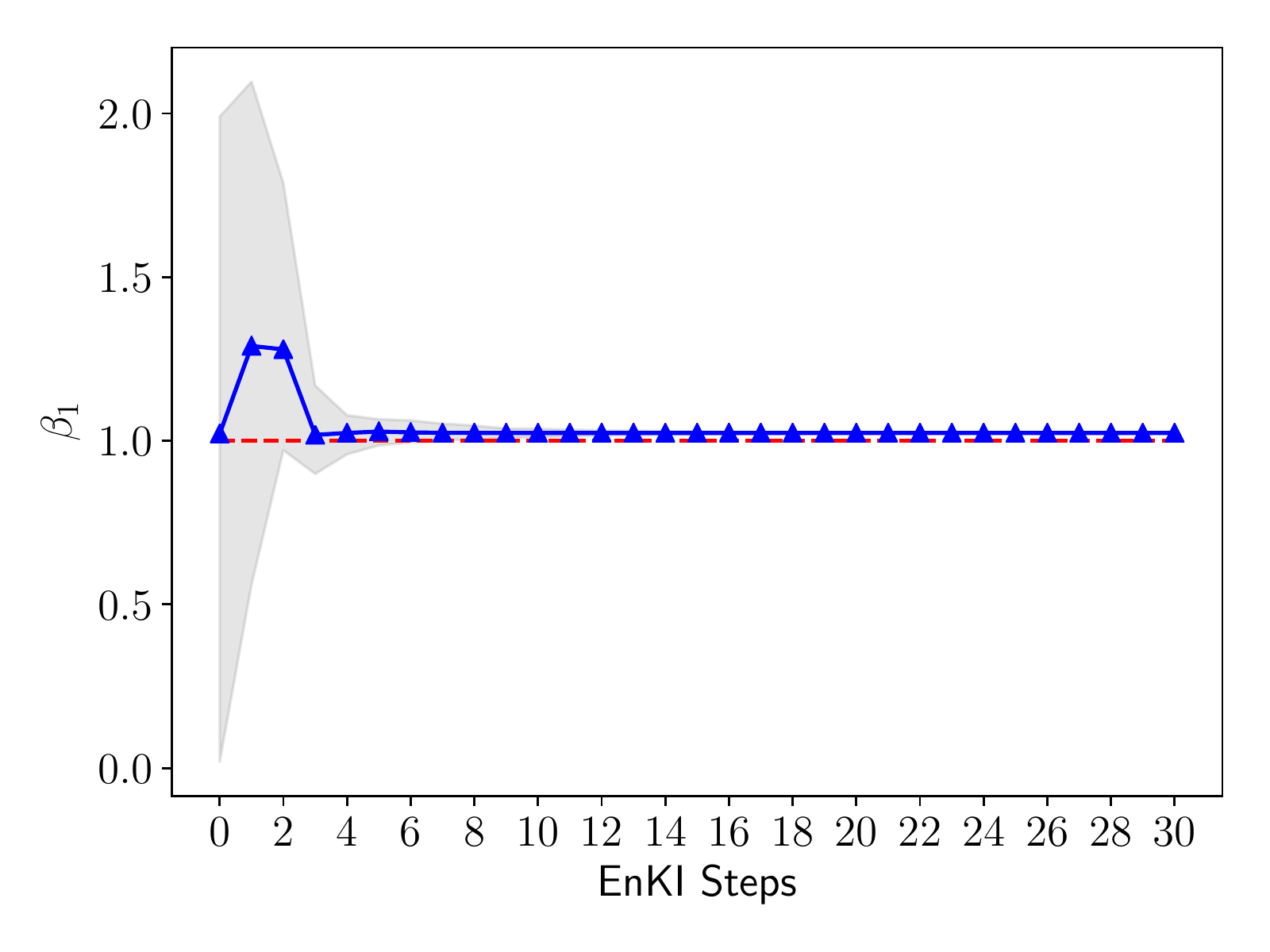}}
  \subfloat[$\beta_3$]{\includegraphics[width=0.49\textwidth]{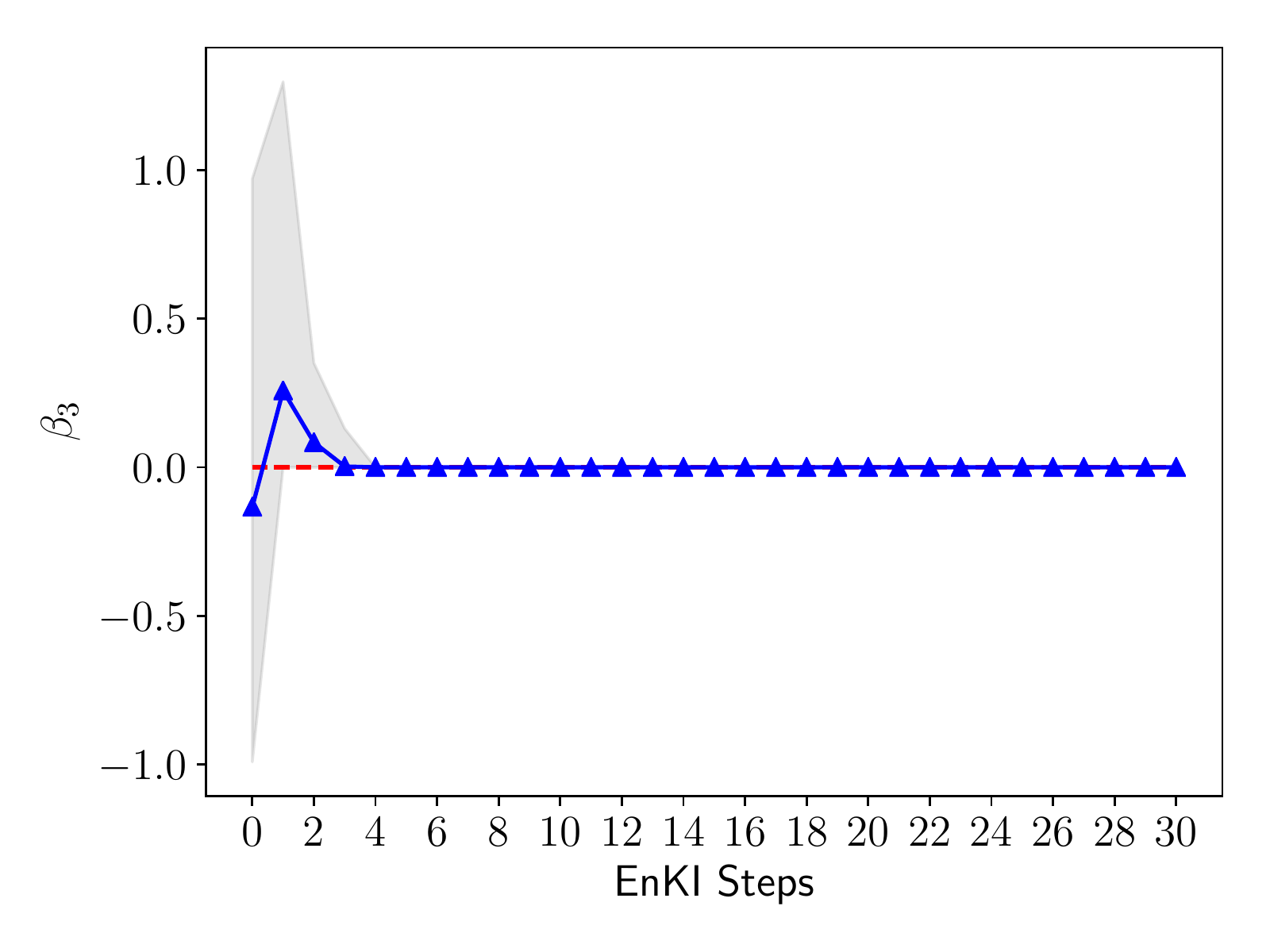}}
    \caption{Estimated coefficients from the second sparse EKI, including (a)--(c) necessary coefficients of Kuramoto-Sivashinsky Equation ($\alpha_2$, $\alpha_4$, and $\beta_1$), and (d) the redudant coefficient $\beta_3$.}
  \label{fig:params_KS_2nd_EKI}
\end{figure}

\clearpage
\section{Conclusions}
\label{sec:C}
We have demonstrated that sparsity may be naturally incorporated
into ensemble Kalman-based inversion methods, leading to the
sparse EKI algorithm. The focus of the paper is on learning
dynamical models from indirect, partial and nonlinear observations, because the solution of such inverse problems in the sparse setting has been an outstanding challenge in the field. We focus on time-averaged data as a canonical
example of such data. The numerical results  presented showcase the success of discovering dynamical models in a variety of different examples, demonstrating that sparse learning for model discovery from time-averaged functions of states can be effectively achieved. The proposed sparse learning methodology extends the scope of data-driven discovery of dynamical models from linear observation operators to previously challenging applications where the observation operator is nonlinear. The methodology may in principle be used for the solution of a wide class of nonlinear
inverse problems in which sparse solutions are sought. As with
existing methods to find sparse solutions of \emph{linear} inverse
problems, the core computational task is a quadratic programming
problem, subject to linear inequality constraints, and therefore
easily implemented. Remarkably, this same computational task
allows for solution of \emph{nonlinear} inverse problems when
using the proposed sparse EKI method.

Directions for future research stemming from our work
include:

\begin{itemize}

\item application of the method to other nonlinear inverse
problems where sparse learning from a dictionary of functions
is valuable;

\item development of theory to support the
use of the algorithm, noting however that, even
in the absence of constraints and imposition of
sparsity, the theoretical underpinnings of ensemble
inversion methods are only starting to be understood 
\cite{garbuno2020interacting,schillings2017analysis};

\item detailed study of the use of different algorithms
for the imposition of sparsity on the basic quadratic
programming task, which is undertaken iteratively during
ensemble Kalman inversion;

\item careful comparison of ensemble Kalman inversion with
other sparsity imposing methods for solving the nonlinear
inverse problem of learning dynamical systems from
data in time-averaged form.

\end{itemize}

\vspace{0.3in}

\noindent{\bf Acknowledgements} We thank Melanie Bieli, Tobias Bischoff and Anna Jaruga for sharing their formulation of the moment-based coalescence equation, and for discussions about it. All authors are supported by the generosity of Eric and Wendy Schmidt by recommendation of the Schmidt Futures program, by Earthrise Alliance, Mountain Philanthropies, the Paul G. Allen Family Foundation, and the National Science Foundation (NSF, award AGS1835860). A.M.S. is also supported by NSF (award DMS-1818977) and by the Office of Naval Research (award N00014-17-1-2079).

\clearpage

\clearpage
\appendix
\section{Numerical Solution Of The Extended K-S Equation}
\label{sec:KS_solver}
We consider the Extended Kuramoto--Sivashinsky (E-K-S) equation on a 
periodic domain in one dimension:
\begin{equation}
\begin{aligned}
\partial_t u&=-\sum_{j=1}^5 \Bigl(\alpha_j \partial^j_x u +
\beta_j u^{j} \partial_x u\Bigr),\quad x\in \mathbb{T}^L,\\
u|_{t=0}&=u_0;
\end{aligned}
\end{equation}
here $\mathbb{T}^L$ denotes the torus $[0,L]$. We write the E-K-S equation as
\begin{equation}
\label{eq:eks}
\partial_t u=\mathcal{L}u+\mathcal{N}(u),
\end{equation}
where
\begin{equation}
\begin{aligned}
\mathcal{L}u&=-\sum_{j=1}^5 \alpha_j \partial^j_x u,\\
\mathcal{N}(u)&=-\sum_{j=1}^5 \frac{\beta_j}{(j+1)} \partial_x (u^{j+1}).
\end{aligned}
\end{equation}

 Using the Crank-Nicolson/Adams-Bashforth scheme, the above equation can be discretized as
\begin{equation}
\frac{u_{(n+1)}-u_{(n)}}{\Delta t}=\mathcal{L}\frac{u_{(n+1)}+u_{(n)}}{2}+\frac{3}{2}\mathcal{N}(u_{(n)})-\frac{1}{2}\mathcal{N}(u_{(n-1)}),
\end{equation}
where $u_{(n)}=u(x,n\Delta t)$. We introduce the Fourier transform in the spatial domain:
\begin{equation}
\hat{u}(\xi)=\mathcal{F}(u)=\int_0^L u(x)e^{-2\pi ix\xi}dx.
\end{equation}
Using this notation we obtain the discretization:
\begin{equation}
\label{eq:ks_fourier}
\frac{\hat{u}_{(n+1)}-\hat{u}_{(n)}}{\Delta t}=\hat{\mathcal{L}}\Bigl(\frac{\hat{u}_{(n+1)}+\hat{u}_{(n)}}{2}\Bigr)+\frac{3}{2}\hat{\mathcal{N}}(\hat{u}_{(n)})-\frac{1}{2}\hat{\mathcal{N}}(\hat{u}_{(n-1)}),
\end{equation}
where
\begin{equation}
\begin{aligned}
\hat{\mathcal{L}}\hat{u}&=-\sum_{j=1}^5 \alpha_j (2\pi i \xi)^j\hat{u},\\
\hat{\mathcal{N}}(\hat{u})&=-\sum_{j=1}^5 \frac{2\pi i \xi \beta_j}{(j+1)} \mathcal{F}\left(\left(\mathcal{F}^{-1}\left(\hat{u}\right)\right)^{j+1}\right),
\end{aligned}
\end{equation}
and where $\mathcal{F}^{-1}$ denotes the inverse Fourier transform. 
Note that if $\alpha_4>0$ then
$\lim_{|\xi| \to \infty} Re(\hat{\mathcal{L}})=-\infty$ which makes the
equation well-posed. The discretization in \eqref{eq:ks_fourier} can be further formulated as
\begin{equation}
\label{eq:CM1}
\left(I-\frac{\Delta t}{2}\hat{\mathcal{L}} \right) \hat{u}_{(n+1)}=\left(I+\frac{\Delta t}{2}\hat{\mathcal{L}} \right) \hat{u}_{(n)}+\frac{3\Delta t}{2}\hat{\mathcal{N}}(\hat{u}_{(n)})-\frac{\Delta t}{2}\hat{\mathcal{N}}(\hat{u}_{(n-1)}).
\end{equation}

Alternatively, a standard integrating factor method can be obtained by introducing the Fourier transform in the spatial domain and rewriting the Fourier transform of \eqref{eq:eks} as
\begin{equation}
\partial_t \Bigl(e^{-\hat{\mathcal{L}}t}\hat{u}\Bigr)=e^{-\hat{\mathcal{L}}t}\hat{N}(\hat{u}).
\end{equation}
This equation can be solved numerically by using the second-order Adams--Bashforth scheme to obtain
\begin{equation}
\label{eq:CM2}
\hat{u}_{(n+1)}=e^{\hat{\mathcal{L}}\Delta t}\hat{u}_{(n)}+\frac{3\Delta t}{2}e^{\hat{\mathcal{L}}\Delta t}\hat{\mathcal{N}}(\hat{u}_{(n)})-\frac{\Delta t}{2}e^{2\hat{\mathcal{L}}\Delta t}\hat{\mathcal{N}}(\hat{u}_{(n-1)}).
\end{equation}
The two algorithms \eqref{eq:CM1}, \eqref{eq:CM2}
may be found in \cite{cox2002exponential}.

In this work, numerical clipping is implemented at every time step to avoid possible  blow-up induced by the numerical discretizarion \eqref{eq:CM1} or \eqref{eq:CM2}:
\begin{equation}
[\hat{u}_{(n+1)}]=\mathcal{F}\Bigl(\max\Bigl(\min\Bigl(\mathcal{F}^{-1}(\hat{u}_{(n+1)}), u^{\max}_{(n+1)}\Bigr), u^{\min}_{(n+1)}\Bigr)\Bigr),
\end{equation}
where $u^{\max}_{(n+1)}$ and $u^{\min}_{(n+1)}$ are upper and lower bounds imposed on the simulated state in the spatial domain.
Both algorithms \eqref{eq:CM1}, \eqref{eq:CM2}, with clipping,
were used, initially, to test the robustness of results
to choice of time-stepper; having verified this robustness, all results presented in the paper use algorithm \eqref{eq:CM1}.

\begin{figure}[!htbp]
  \centering
  \includegraphics[width=0.5\textwidth]{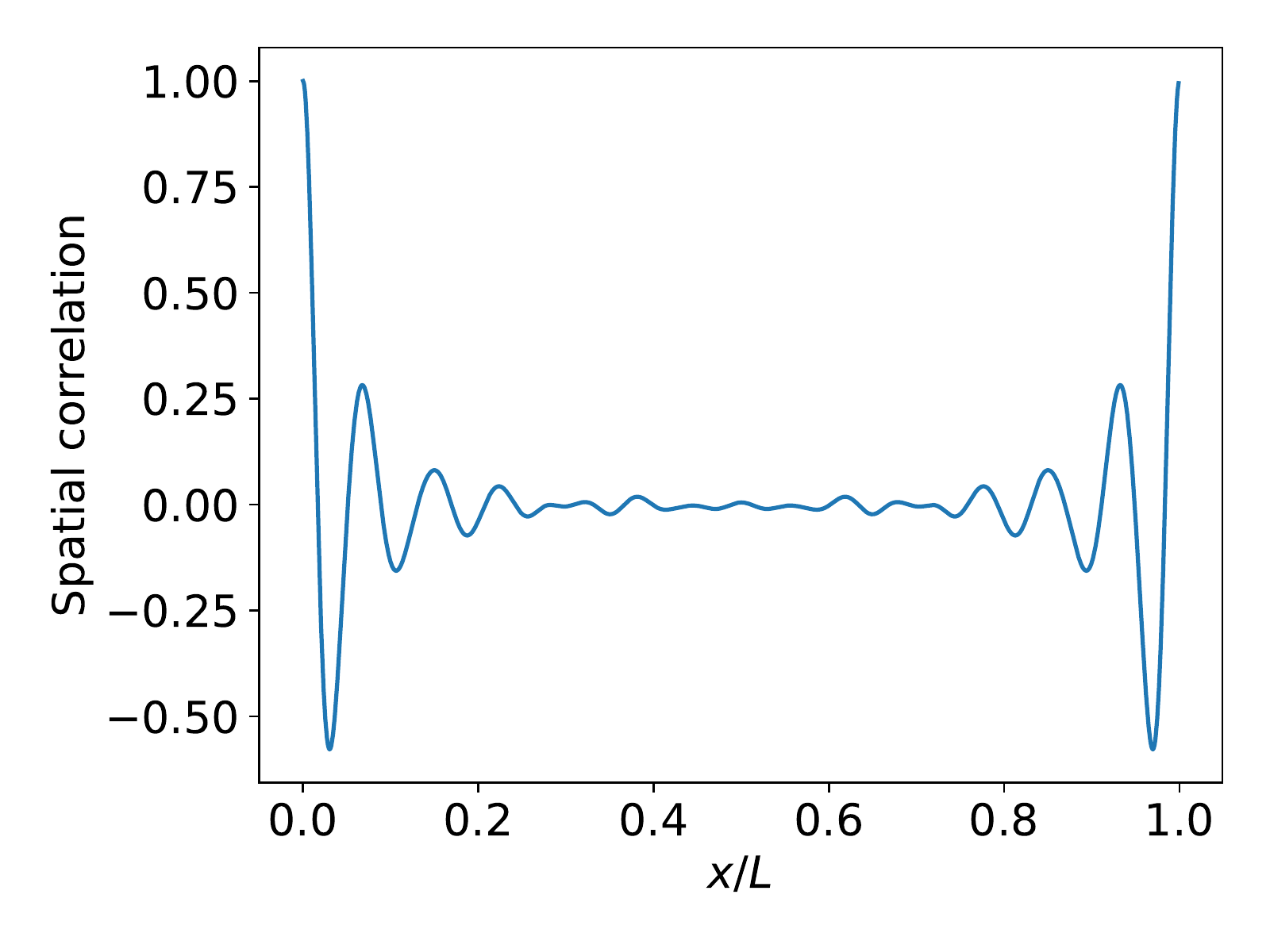}
    \caption{The time-averaged spatial correlation of the simulated states of K-S equation. Normalization has been performed to set the largest value to $1$.}
\label{fig:KS_spatial_corr_truth}
\end{figure}

Using algorithm \eqref{eq:CM1} we compute the time-averaged spatial correlation function defined by
\begin{equation}
\begin{aligned}
C(x)&=\frac{1}{T}\int_0^T\int_0^L u(z,t)u(z+x,t)dzdt. \\
\end{aligned}
\end{equation}
Notice that
$$(\mathcal{F}C)(\xi)=\frac{1}{T}\int_0^T |(\mathcal{F}u)(\xi,t)|^2 dt$$
facilitating straightforward computation in Fourier space.
The function $C(x)$ is shown in Fig.~\ref{fig:KS_spatial_corr_truth}.
It is used as part of the definition of $G$, along with moments
and autocorrelation information, from which we learn parameters.

\end{document}